\documentclass[review]{elsarticle}
\usepackage{amsmath}
\usepackage{amssymb}
\usepackage{amsthm}
\usepackage{multirow}
\usepackage{mathtools}
\usepackage{graphicx}
\usepackage{enumerate}
\usepackage{float}
\usepackage{url}
\usepackage{verbatim}
\usepackage{subcaption}
\usepackage{tikz}
\usepackage[colorlinks, linkcolor=red,
  anchorcolor=blus, citecolor=green]{hyperref}

\newcommand\bm{\boldsymbol}
\newcommand\dx{\mathrm{d}\boldsymbol{x}}
\newcommand{\lj}{[ \hspace{-2pt} [}
\newcommand{\rj}{] \hspace{-2pt} ]}
\newtheorem{assumption}{Assumption}
\newtheorem{theorem}{Theorem}
\newtheorem{lemma}{Lemma}

\numberwithin{equation}{section}

\graphicspath{./}

\newcommand{\mc}[1]{\mathcal{#1}}

\newcommand{\wt}[1]{\widetilde{#1}}
\def\MTh{\mc{T}_h}

\def\dx{\,\mathrm{d}x}
\def\ds{\mathrm{d}s}

\journal{Journal of Computational and Applied Mathematics}

\begin{document}

\begin{frontmatter}

\title{A finite element method by patch reconstruction for the Stokes problem
using mixed formulations}

\author[add1]{Ruo Li}
\ead{rli@math.pku.edu.cn}

\author[add2]{Zhiyuan Sun}
\ead{zysun@math.pku.edu.cn}

\author[add2]{Fanyi Yang}
\ead{yangfanyi@pku.edu.cn}

\author[add3]{Zhijian Yang}
\ead{zjyang.math@whu.edu.cn}

\address[add1]{CAPT, LMAM and School of Mathematical
  Sciences, Peking University, Beijing 100871, P. R. China}

\address[add2]{School of Mathematical
  Sciences, Peking University, Beijing 100871, P. R. China}

\address[add3]{School of Mathematics and Statistics, Wuhan University}

\begin{abstract}
  In this paper, we develop a patch reconstruction finite element
  method for the Stokes problem. The weak formulation of the interior
  penalty discontinuous Galerkin is employed. The proposed method has
  a great flexibility in velocity-pressure space pairs whose stability
  properties are confirmed by the inf-sup tests. Numerical examples
  show the applicability and efficiency of the proposed method.

\end{abstract}


\end{frontmatter}

\section{Introduction}
\label{sec:introduction}
We are concerned in this paper with the incompressible Stokes problem,
which has a wide range of applications on the approximation of low
Reynolds number flows and the time discretizations of the Oseen
equation or Naiver-Stokes equation. One of the major difficulties in
finite element discretizations for the Stokes problem is the
incompressible constraint, which leads to a saddle-point problem. The
stability condition often referred as the inf-sup (LBB) condition
requires the approximation spaces for velocity and pressure need to be
carefully chosen \cite{boffi2013mixed}. We refer to
\cite{girault1986finite, taylor1973numerical} for some specific spaces
used in the traditional finite element methods to solve the Stokes
problem.

Most recently, the discontinuous Galerkin (DG) methods have achieved a
great success in computational fluid dynamics, see the state of art
survey \cite{cockburn2000development}. Hansbo and Larson propose and
analyze an interior penalty DG method for incompressible and nearly
incompressible linear elasticity on triangular meshes in
\cite{hansbo2002discontinuous} where polynomial spaces of degree $k$
and $k-1$ are employed to approximate velocity and pressure,
respectively. In \cite{toselli2002hp} Toselli considers the
$hp$-approximations for the Stokes problem using piecewise polynomial
spaces. The uniform divergence stability and error estimates with
respect to $h$ and $p$ are proven for this DG formulation when
velocity is approximated one or two degrees higher than pressure.
Numerical results show that using equal order spaces for velocity and
pressure can also work well. Sch{\"o}tzau et al. improve the estimates
on tensor product meshes in \cite{schotzau2002mixed}. A local
discontinuous Galerkin method (LDG) for the Stokes problem is proposed
in \cite{cockburn2002local}. The LDG method can be considered as a
stabilized method when the approximation spaces for velocity and
pressure are chosen with the same order. {  Hybrid
discontinuous Galerkin methods are also of interest due to their
capability of providing a superconvergent post processing, we refer to
\cite{carrero2006hybridized, Lederer2018Hybrid,
Nguyen2010Hybridizable, Cockburn2009Hybridization} for more
discussion. }

Some special finite element spaces can be adopted to Stokes problem in
DG framework.  Karakashian and his coworkers \cite{baker1990piecewise,
karakashian1998nonconforming} propose a DG method with piecewise
solenoidal vector fields which are locally divergence-free. Cockburn
et al.  \cite{cockburn2005incompressiblei,
cockburn2005incompressibleii, carrero2006hybridized, cockburn2007note}
develop the LDG method with solenoidal vector fields. By introducing
the hybrid pressure, the pressure and the globally divergence-free
velocity can be obtained by a post-process of the LDG solution. While
Montlaur et al.\cite{montlaur2008discontinuous} present two DG
formulations for the incompressible flow, the first formulation is
derived from an interior penalty method such that the computation of
the velocity and the pressure is decoupled and the second formulation
follows the methodology in \cite{baker1990piecewise}. With an
inconsistent penalty, the velocity can be computed with absence of
pressure terms.  Liu \cite{liu2011penalty} presents a
penalty-factor-free DG formulation for the Stokes problem with optimal
error estimates. 

However, one of the limitations of DG methods is the computational
cost is higher than using continuous Galerkin method directly
\cite{zienkiewicz2003discontinuous, montlaur2009high} because of the
duplication of the degrees of freedom at interelement boundaries
especially in three-dimensional case. In this paper, we follow the
methodology in \cite{li2012efficient, 2018arXiv180300378L} to apply
the patch reconstruction finite element method to the Stokes problem.
Piecewise polynomial spaces built by patch reconstruction procedure
are taken to approximate velocity and pressure. The new space is a
sub-space of the common approximation space used in DG framework,
which allows us to employ the interior penalty formulation directly to
solve the Stokes problem. As we mentioned before, it is important to
verify the inf-sup condition for a mixed formulation to guarantee the
stability, which is often severe for a specific discretization
\cite{bathe2000inf}. We carry out a series of numerical inf-sup tests
proposed in \cite{chapelle1993inf, boffi2013mixed} to show this method
is numerically stable.

The proposed method provides many merits. First, the DOFs of the
system are totally decided by the mesh partition and have no
relationship with the interpolation order. Then the method is easy to
implement on arbitrary polygonal meshes because of the independence
between the process of the construction of the space and the geometry
structure of meshes. Third, we emphasize that the spaces to
approximate velocity and pressure can be engaged with great
flexibility. The results of numerical inf-sup tests exhibit the
robustness of our method even in some extreme cases.

The outline of this paper is organized as follows. In Section
\ref{sec:reconopreator}, we briefly introduce the patch reconstruction
procedure and the finite element space. Then the scheme of the mixed
interior penalty DG method and its error analysis for the Stokes
problem are presented in Section \ref{sec:weakform}. In Section
\ref{sec:infsuptest}, we briefly review the inf-sup test and carry out
a series of numerical inf-sup tests in several situations to show the
proposed method satisfies the inf-sup condition. Finally,
two-dimensional numerical examples are presented in Section
\ref{sec:numericalresults} to illustrate the accuracy and efficiency
of the proposed approach, and verify our theoretical results.

\section{Reconstruction operator}
\label{sec:reconopreator}
In this section, we will introduce a reconstruction operator which can
be constructed on any polygonal meshes and its corresponding
approximation properties.

Let $\Omega\subset\mathbb R^d, d=2, 3$, be a convex polygonal domain
with boundary $\partial\Omega$. We denote by $\mathcal{T}_h$ a
subdivision that partitions $\Omega$ into polygonal elements. And let
$\mathcal E_h$ be the set of $(d-1)$-dimensional interfaces (edges) of
all elements in $\mathcal T_h$, $\mathcal E_h^i$ the set of interior
faces and $\mathcal E_h^b$ the set of the faces on the domain boundary
$\partial\Omega$. We set
\begin{displaymath}
  h=\max_{K\in\mathcal T_h} h_K,\quad h_K=\text{diam}(K),\quad
  h_e=\text{diam}(e),
\end{displaymath}
for $\forall K\in \mathcal T_h,\ \forall e \in \mathcal E_h$. Further,
we assume that the partition $\MTh$ admits the following shape
regularity conditions ~\cite{Brezzi:2009,DaVeiga2014}:

\begin{enumerate}
\item[{\bf H1}\;]There exists an integer number $N$ independent of
  $h$, that any element $K$ admits a sub-decomposition $\wt{\mathcal
  T}_{h | K}$ made of at most $N$ triangles.
\item[{\bf H2}\;] $\wt{\MTh}$ is a compatible sub-decomposition, that
  any triangle $T\in\wt{\MTh}$ is shape-regular in the sense of
  Ciarlet-Raviart~\cite{ciarlet:1978}: there exists a real positive
  number $\sigma$ independent of $h$ such that $h_T/\rho_T\le\sigma$,
  where $\rho_T$ is the radius of the largest ball inscribed in $T$.
\end{enumerate}
There many useful properties using for the analysis in finite
difference schemes and DG framework can be derived from the above
assumptions, such as Agmon inequality and inverse inequality
\cite{DaVeiga2014, antonietti2013hp, 2018arXiv180300378L}:
\begin{enumerate}
  \item[{\bf M1}\;][Agmon inequality] There exists $C$ that depends on
    $N$ and $\sigma$ but independent of $h_K$ such that
    \begin{displaymath}
      \|v\|_{L^2(\partial K)}^2 \leq C\left( h_K^{-1}\|v\|_{L^2(K)}^2
      + h_K\|\nabla v\|_{L^2(K)}^2 \right), \quad \forall v \in
      H^1(K).
    \end{displaymath}
  \item [{\bf M2}\;][Inverse inequality] There exists $C$ that depends
    on $N$ and $\sigma$ but independent of $h_K$ such that
    \begin{displaymath}
      \|\nabla v\|_{L^2(K)} \leq Cm^2/h_K\|v\|_{L^2(K)}, \quad \forall
      v \in \mathbb P_m(K).
    \end{displaymath}
\end{enumerate}


{
Given the partition $\MTh$, we define the reconstruction operator as
follows. First in each element $K \in \MTh$, we specify a point $\bm
x_K \in K$ as the collocation point. Here we just let $\bm x_K$ be the
barycenter of $K$. Then for each $K \in \MTh$ we construct an element
patch $S(K)$, which is a set of $K$ itself and some elements around
$K$. Specifically, we construct $S(K)$ in a recursive manner. For
element $K$, we set $S(K) = \left\{ K \right\}$ first, and we enlarge
$S(K)$ by adding all the von Neumann neighbours (adjacent
edge-neighbouring elements) of $S(K)$ into $S(K)$ recursively until we
have collected enough elements into the element patch. We denote by
$\# S(K)$ the cardinality of $S(K)$ and an example of construction of
$S(K)$ with $\# S(K) = 12$ is shown in Fig \ref{fig:buildpatch}.  }

\begin{figure}
  \centering \captionsetup[subfigure]{labelformat=empty}
  \begin{subfigure}{.3\textwidth}
    \begin{tikzpicture}[scale=3.5]
    \draw[fill=red] (0.377541080267955,0.527424438107543) -- (0.51632192714807,0.445641468992195) -- (0.508058617199453,0.641022485885209);
\draw[] (0,0) -- (0.2,0);
\draw[] (0.2,0) -- (0.129566038401955,0.132370390921097);
\draw[] (0,0) -- (0.129566038401955,0.132370390921097);
\draw[] (0,0) -- (0,0.2);
\draw[] (0,0.2) -- (0.129566038401955,0.132370390921097);
\draw[] (0.2,0) -- (0.28924891280161,0.166947666487585);
\draw[] (0.28924891280161,0.166947666487585) -- (0.129566038401955,0.132370390921097);
\draw[] (0.158829078088844,0.295083754154879) -- (0.129566038401955,0.132370390921097);
\draw[] (0,0.2) -- (0.158829078088844,0.295083754154879);
\draw[] (0.158829078088844,0.295083754154879) -- (0.28924891280161,0.166947666487585);
\draw[] (0.4,0) -- (0.28924891280161,0.166947666487585);
\draw[] (0.2,0) -- (0.4,0);
\draw[] (0,0.4) -- (0.158829078088844,0.295083754154879);
\draw[] (0,0.4) -- (0,0.2);
\draw[] (0.158829078088844,0.295083754154879) -- (0.340392050716961,0.360085062591018);
\draw[] (0.28924891280161,0.166947666487585) -- (0.340392050716961,0.360085062591018);
\draw[] (0.4,0) -- (0.507386405366096,0.214426443909476);
\draw[] (0.507386405366096,0.214426443909476) -- (0.28924891280161,0.166947666487585);
\draw[] (0.507386405366096,0.214426443909476) -- (0.340392050716961,0.360085062591018);
\draw[] (0.194389338585002,0.511472795791168) -- (0.158829078088844,0.295083754154879);
\draw[] (0,0.4) -- (0.194389338585002,0.511472795791168);
\draw[] (0.194389338585002,0.511472795791168) -- (0.340392050716961,0.360085062591018);
\draw[] (0.4,0) -- (0.6,0);
\draw[] (0.6,0) -- (0.507386405366096,0.214426443909476);
\draw[] (0,0.6) -- (0.194389338585002,0.511472795791168);
\draw[] (0,0.6) -- (0,0.4);
\draw[] (0.507386405366096,0.214426443909476) -- (0.51632192714807,0.445641468992195);
\draw[] (0.340392050716961,0.360085062591018) -- (0.51632192714807,0.445641468992195);
\draw[] (0.340392050716961,0.360085062591018) -- (0.377541080267955,0.527424438107543);
\draw[] (0.194389338585002,0.511472795791168) -- (0.377541080267955,0.527424438107543);
\draw[] (0.507386405366096,0.214426443909476) -- (0.720367993373896,0.166464907980495);
\draw[] (0.6,0) -- (0.720367993373896,0.166464907980495);
\draw[] (0.377541080267955,0.527424438107543) -- (0.51632192714807,0.445641468992195);
\draw[] (0.194389338585002,0.511472795791168) -- (0.155048548464969,0.719006158403683);
\draw[] (0,0.6) -- (0.155048548464969,0.719006158403683);
\draw[] (0.507386405366096,0.214426443909476) -- (0.686726828176686,0.361868419354705);
\draw[] (0.686726828176686,0.361868419354705) -- (0.51632192714807,0.445641468992195);
\draw[] (0.720367993373896,0.166464907980495) -- (0.686726828176686,0.361868419354705);
\draw[] (0.194389338585002,0.511472795791168) -- (0.329936867383997,0.679347780724589);
\draw[] (0.329936867383997,0.679347780724589) -- (0.377541080267955,0.527424438107543);
\draw[] (0.155048548464969,0.719006158403683) -- (0.329936867383997,0.679347780724589);
\draw[] (0.8,0) -- (0.720367993373896,0.166464907980495);
\draw[] (0.6,0) -- (0.8,0);
\draw[] (0.508058617199453,0.641022485885209) -- (0.51632192714807,0.445641468992195);
\draw[] (0.377541080267955,0.527424438107543) -- (0.508058617199453,0.641022485885209);
\draw[] (0,0.8) -- (0.155048548464969,0.719006158403683);
\draw[] (0,0.8) -- (0,0.6);
\draw[] (0.329936867383997,0.679347780724589) -- (0.508058617199453,0.641022485885209);
\draw[] (0.67973159086246,0.56913545717858) -- (0.686726828176686,0.361868419354705);
\draw[] (0.67973159086246,0.56913545717858) -- (0.51632192714807,0.445641468992195);
\draw[] (0.720367993373896,0.166464907980495) -- (0.854172785710488,0.290767332144404);
\draw[] (0.854172785710488,0.290767332144404) -- (0.686726828176686,0.361868419354705);
\draw[] (0.67973159086246,0.56913545717858) -- (0.508058617199453,0.641022485885209);
\draw[] (0.279589050761926,0.849922371417845) -- (0.155048548464969,0.719006158403683);
\draw[] (0.279589050761926,0.849922371417845) -- (0.329936867383997,0.679347780724589);
\draw[] (0.8,0) -- (0.874828381578865,0.131478330490222);
\draw[] (0.720367993373896,0.166464907980495) -- (0.874828381578865,0.131478330490222);
\draw[] (0.854172785710488,0.290767332144404) -- (0.874828381578865,0.131478330490222);
\draw[] (0,0.8) -- (0.126891091107252,0.873742007888985);
\draw[] (0.155048548464969,0.719006158403683) -- (0.126891091107252,0.873742007888985);
\draw[] (0.466174482542919,0.827842017700953) -- (0.508058617199453,0.641022485885209);
\draw[] (0.466174482542919,0.827842017700953) -- (0.329936867383997,0.679347780724589);
\draw[] (0.279589050761926,0.849922371417845) -- (0.126891091107252,0.873742007888985);
\draw[] (0.844029404278513,0.484446265452234) -- (0.67973159086246,0.56913545717858);
\draw[] (0.844029404278513,0.484446265452234) -- (0.686726828176686,0.361868419354705);
\draw[] (0.844029404278513,0.484446265452234) -- (0.854172785710488,0.290767332144404);
\draw[] (0.279589050761926,0.849922371417845) -- (0.466174482542919,0.827842017700953);
\draw[] (0.68033851010831,0.796997572945071) -- (0.67973159086246,0.56913545717858);
\draw[] (0.68033851010831,0.796997572945071) -- (0.508058617199453,0.641022485885209);
\draw[] (0.8,0) -- (1,0);
\draw[] (1,0) -- (0.874828381578865,0.131478330490222);
\draw[] (0.466174482542919,0.827842017700953) -- (0.68033851010831,0.796997572945071);
\draw[] (0,1) -- (0,0.8);
\draw[] (0,1) -- (0.126891091107252,0.873742007888985);
\draw[] (1,0.2) -- (0.854172785710488,0.290767332144404);
\draw[] (1,0.2) -- (0.874828381578865,0.131478330490222);
\draw[] (0.84467940856602,0.684532224021434) -- (0.844029404278513,0.484446265452234);
\draw[] (0.84467940856602,0.684532224021434) -- (0.67973159086246,0.56913545717858);
\draw[] (0.84467940856602,0.684532224021434) -- (0.68033851010831,0.796997572945071);
\draw[] (0.2,1) -- (0.126891091107252,0.873742007888985);
\draw[] (0.2,1) -- (0.279589050761926,0.849922371417845);
\draw[] (0.4,1) -- (0.279589050761926,0.849922371417845);
\draw[] (0.4,1) -- (0.466174482542919,0.827842017700953);
\draw[] (1,0.4) -- (0.844029404278513,0.484446265452234);
\draw[] (1,0.4) -- (0.854172785710488,0.290767332144404);
\draw[] (1,0) -- (1,0.2);
\draw[] (0.2,1) -- (0,1);
\draw[] (1,0.2) -- (1,0.4);
\draw[] (0.4,1) -- (0.2,1);
\draw[] (0.6,1) -- (0.466174482542919,0.827842017700953);
\draw[] (0.6,1) -- (0.68033851010831,0.796997572945071);
\draw[] (1,0.6) -- (0.84467940856602,0.684532224021434);
\draw[] (1,0.6) -- (0.844029404278513,0.484446265452234);
\draw[] (0.6,1) -- (0.4,1);
\draw[] (0.68033851010831,0.796997572945071) -- (0.864919419945519,0.8563330434245);
\draw[] (0.84467940856602,0.684532224021434) -- (0.864919419945519,0.8563330434245);
\draw[] (1,0.4) -- (1,0.6);
\draw[] (0.8,1) -- (0.68033851010831,0.796997572945071);
\draw[] (0.8,1) -- (0.6,1);
\draw[] (0.8,1) -- (0.864919419945519,0.8563330434245);
\draw[] (1,0.6) -- (1,0.8);
\draw[] (1,0.8) -- (0.84467940856602,0.684532224021434);
\draw[] (1,0.8) -- (0.864919419945519,0.8563330434245);
\draw[] (1,1) -- (0.8,1);
\draw[] (1,1) -- (0.864919419945519,0.8563330434245);
\draw[] (1,0.8) -- (1,1);
  \end{tikzpicture}
  \caption{Step 1}
  \end{subfigure}
  \hspace{30pt}
  \begin{subfigure}{.3\textwidth}
    \begin{tikzpicture}[scale=3.5]
    \draw[fill=red] (0.377541080267955,0.527424438107543) -- (0.51632192714807,0.445641468992195) -- (0.508058617199453,0.641022485885209);
\draw[fill=cyan] (0.377541080267955,0.527424438107543) -- (0.340392050716961,0.360085062591018) -- (0.51632192714807,0.445641468992195);
\draw[fill=cyan] (0.51632192714807,0.445641468992195) -- (0.67973159086246,0.56913545717858) -- (0.508058617199453,0.641022485885209);
\draw[fill=cyan] (0.508058617199453,0.641022485885209) -- (0.329936867383997,0.679347780724589) -- (0.377541080267955,0.527424438107543);
\draw[] (0,0) -- (0.2,0);
\draw[] (0.2,0) -- (0.129566038401955,0.132370390921097);
\draw[] (0,0) -- (0.129566038401955,0.132370390921097);
\draw[] (0,0) -- (0,0.2);
\draw[] (0,0.2) -- (0.129566038401955,0.132370390921097);
\draw[] (0.2,0) -- (0.28924891280161,0.166947666487585);
\draw[] (0.28924891280161,0.166947666487585) -- (0.129566038401955,0.132370390921097);
\draw[] (0.158829078088844,0.295083754154879) -- (0.129566038401955,0.132370390921097);
\draw[] (0,0.2) -- (0.158829078088844,0.295083754154879);
\draw[] (0.158829078088844,0.295083754154879) -- (0.28924891280161,0.166947666487585);
\draw[] (0.4,0) -- (0.28924891280161,0.166947666487585);
\draw[] (0.2,0) -- (0.4,0);
\draw[] (0,0.4) -- (0.158829078088844,0.295083754154879);
\draw[] (0,0.4) -- (0,0.2);
\draw[] (0.158829078088844,0.295083754154879) -- (0.340392050716961,0.360085062591018);
\draw[] (0.28924891280161,0.166947666487585) -- (0.340392050716961,0.360085062591018);
\draw[] (0.4,0) -- (0.507386405366096,0.214426443909476);
\draw[] (0.507386405366096,0.214426443909476) -- (0.28924891280161,0.166947666487585);
\draw[] (0.507386405366096,0.214426443909476) -- (0.340392050716961,0.360085062591018);
\draw[] (0.194389338585002,0.511472795791168) -- (0.158829078088844,0.295083754154879);
\draw[] (0,0.4) -- (0.194389338585002,0.511472795791168);
\draw[] (0.194389338585002,0.511472795791168) -- (0.340392050716961,0.360085062591018);
\draw[] (0.4,0) -- (0.6,0);
\draw[] (0.6,0) -- (0.507386405366096,0.214426443909476);
\draw[] (0,0.6) -- (0.194389338585002,0.511472795791168);
\draw[] (0,0.6) -- (0,0.4);
\draw[] (0.507386405366096,0.214426443909476) -- (0.51632192714807,0.445641468992195);
\draw[] (0.340392050716961,0.360085062591018) -- (0.51632192714807,0.445641468992195);
\draw[] (0.340392050716961,0.360085062591018) -- (0.377541080267955,0.527424438107543);
\draw[] (0.194389338585002,0.511472795791168) -- (0.377541080267955,0.527424438107543);
\draw[] (0.507386405366096,0.214426443909476) -- (0.720367993373896,0.166464907980495);
\draw[] (0.6,0) -- (0.720367993373896,0.166464907980495);
\draw[] (0.377541080267955,0.527424438107543) -- (0.51632192714807,0.445641468992195);
\draw[] (0.194389338585002,0.511472795791168) -- (0.155048548464969,0.719006158403683);
\draw[] (0,0.6) -- (0.155048548464969,0.719006158403683);
\draw[] (0.507386405366096,0.214426443909476) -- (0.686726828176686,0.361868419354705);
\draw[] (0.686726828176686,0.361868419354705) -- (0.51632192714807,0.445641468992195);
\draw[] (0.720367993373896,0.166464907980495) -- (0.686726828176686,0.361868419354705);
\draw[] (0.194389338585002,0.511472795791168) -- (0.329936867383997,0.679347780724589);
\draw[] (0.329936867383997,0.679347780724589) -- (0.377541080267955,0.527424438107543);
\draw[] (0.155048548464969,0.719006158403683) -- (0.329936867383997,0.679347780724589);
\draw[] (0.8,0) -- (0.720367993373896,0.166464907980495);
\draw[] (0.6,0) -- (0.8,0);
\draw[] (0.508058617199453,0.641022485885209) -- (0.51632192714807,0.445641468992195);
\draw[] (0.377541080267955,0.527424438107543) -- (0.508058617199453,0.641022485885209);
\draw[] (0,0.8) -- (0.155048548464969,0.719006158403683);
\draw[] (0,0.8) -- (0,0.6);
\draw[] (0.329936867383997,0.679347780724589) -- (0.508058617199453,0.641022485885209);
\draw[] (0.67973159086246,0.56913545717858) -- (0.686726828176686,0.361868419354705);
\draw[] (0.67973159086246,0.56913545717858) -- (0.51632192714807,0.445641468992195);
\draw[] (0.720367993373896,0.166464907980495) -- (0.854172785710488,0.290767332144404);
\draw[] (0.854172785710488,0.290767332144404) -- (0.686726828176686,0.361868419354705);
\draw[] (0.67973159086246,0.56913545717858) -- (0.508058617199453,0.641022485885209);
\draw[] (0.279589050761926,0.849922371417845) -- (0.155048548464969,0.719006158403683);
\draw[] (0.279589050761926,0.849922371417845) -- (0.329936867383997,0.679347780724589);
\draw[] (0.8,0) -- (0.874828381578865,0.131478330490222);
\draw[] (0.720367993373896,0.166464907980495) -- (0.874828381578865,0.131478330490222);
\draw[] (0.854172785710488,0.290767332144404) -- (0.874828381578865,0.131478330490222);
\draw[] (0,0.8) -- (0.126891091107252,0.873742007888985);
\draw[] (0.155048548464969,0.719006158403683) -- (0.126891091107252,0.873742007888985);
\draw[] (0.466174482542919,0.827842017700953) -- (0.508058617199453,0.641022485885209);
\draw[] (0.466174482542919,0.827842017700953) -- (0.329936867383997,0.679347780724589);
\draw[] (0.279589050761926,0.849922371417845) -- (0.126891091107252,0.873742007888985);
\draw[] (0.844029404278513,0.484446265452234) -- (0.67973159086246,0.56913545717858);
\draw[] (0.844029404278513,0.484446265452234) -- (0.686726828176686,0.361868419354705);
\draw[] (0.844029404278513,0.484446265452234) -- (0.854172785710488,0.290767332144404);
\draw[] (0.279589050761926,0.849922371417845) -- (0.466174482542919,0.827842017700953);
\draw[] (0.68033851010831,0.796997572945071) -- (0.67973159086246,0.56913545717858);
\draw[] (0.68033851010831,0.796997572945071) -- (0.508058617199453,0.641022485885209);
\draw[] (0.8,0) -- (1,0);
\draw[] (1,0) -- (0.874828381578865,0.131478330490222);
\draw[] (0.466174482542919,0.827842017700953) -- (0.68033851010831,0.796997572945071);
\draw[] (0,1) -- (0,0.8);
\draw[] (0,1) -- (0.126891091107252,0.873742007888985);
\draw[] (1,0.2) -- (0.854172785710488,0.290767332144404);
\draw[] (1,0.2) -- (0.874828381578865,0.131478330490222);
\draw[] (0.84467940856602,0.684532224021434) -- (0.844029404278513,0.484446265452234);
\draw[] (0.84467940856602,0.684532224021434) -- (0.67973159086246,0.56913545717858);
\draw[] (0.84467940856602,0.684532224021434) -- (0.68033851010831,0.796997572945071);
\draw[] (0.2,1) -- (0.126891091107252,0.873742007888985);
\draw[] (0.2,1) -- (0.279589050761926,0.849922371417845);
\draw[] (0.4,1) -- (0.279589050761926,0.849922371417845);
\draw[] (0.4,1) -- (0.466174482542919,0.827842017700953);
\draw[] (1,0.4) -- (0.844029404278513,0.484446265452234);
\draw[] (1,0.4) -- (0.854172785710488,0.290767332144404);
\draw[] (1,0) -- (1,0.2);
\draw[] (0.2,1) -- (0,1);
\draw[] (1,0.2) -- (1,0.4);
\draw[] (0.4,1) -- (0.2,1);
\draw[] (0.6,1) -- (0.466174482542919,0.827842017700953);
\draw[] (0.6,1) -- (0.68033851010831,0.796997572945071);
\draw[] (1,0.6) -- (0.84467940856602,0.684532224021434);
\draw[] (1,0.6) -- (0.844029404278513,0.484446265452234);
\draw[] (0.6,1) -- (0.4,1);
\draw[] (0.68033851010831,0.796997572945071) -- (0.864919419945519,0.8563330434245);
\draw[] (0.84467940856602,0.684532224021434) -- (0.864919419945519,0.8563330434245);
\draw[] (1,0.4) -- (1,0.6);
\draw[] (0.8,1) -- (0.68033851010831,0.796997572945071);
\draw[] (0.8,1) -- (0.6,1);
\draw[] (0.8,1) -- (0.864919419945519,0.8563330434245);
\draw[] (1,0.6) -- (1,0.8);
\draw[] (1,0.8) -- (0.84467940856602,0.684532224021434);
\draw[] (1,0.8) -- (0.864919419945519,0.8563330434245);
\draw[] (1,1) -- (0.8,1);
\draw[] (1,1) -- (0.864919419945519,0.8563330434245);
\draw[] (1,0.8) -- (1,1);
  \end{tikzpicture}
  \caption{Step 2}
  \end{subfigure}

  \vspace{20pt}

  \begin{subfigure}{.3\textwidth}
    \begin{tikzpicture}[scale=3.5]
    \draw[fill=red] (0.377541080267955,0.527424438107543) -- (0.51632192714807,0.445641468992195) -- (0.508058617199453,0.641022485885209);
\draw[fill=cyan] (0.377541080267955,0.527424438107543) -- (0.340392050716961,0.360085062591018) -- (0.51632192714807,0.445641468992195);
\draw[fill=cyan] (0.51632192714807,0.445641468992195) -- (0.67973159086246,0.56913545717858) -- (0.508058617199453,0.641022485885209);
\draw[fill=cyan] (0.508058617199453,0.641022485885209) -- (0.329936867383997,0.679347780724589) -- (0.377541080267955,0.527424438107543);
\draw[fill=cyan] (0.51632192714807,0.445641468992195) -- (0.340392050716961,0.360085062591018) -- (0.507386405366096,0.214426443909476);
\draw[fill=cyan] (0.377541080267955,0.527424438107543) -- (0.194389338585002,0.511472795791168) -- (0.340392050716961,0.360085062591018);
\draw[fill=cyan] (0.377541080267955,0.527424438107543) -- (0.329936867383997,0.679347780724589) -- (0.194389338585002,0.511472795791168);
\draw[fill=cyan] (0.51632192714807,0.445641468992195) -- (0.686726828176686,0.361868419354705) -- (0.67973159086246,0.56913545717858);
\draw[fill=cyan] (0.466174482542919,0.827842017700953) -- (0.329936867383997,0.679347780724589) -- (0.508058617199453,0.641022485885209);
\draw[fill=cyan] (0.68033851010831,0.796997572945071) -- (0.508058617199453,0.641022485885209) -- (0.67973159086246,0.56913545717858);
\draw[] (0,0) -- (0.2,0);
\draw[] (0.2,0) -- (0.129566038401955,0.132370390921097);
\draw[] (0,0) -- (0.129566038401955,0.132370390921097);
\draw[] (0,0) -- (0,0.2);
\draw[] (0,0.2) -- (0.129566038401955,0.132370390921097);
\draw[] (0.2,0) -- (0.28924891280161,0.166947666487585);
\draw[] (0.28924891280161,0.166947666487585) -- (0.129566038401955,0.132370390921097);
\draw[] (0.158829078088844,0.295083754154879) -- (0.129566038401955,0.132370390921097);
\draw[] (0,0.2) -- (0.158829078088844,0.295083754154879);
\draw[] (0.158829078088844,0.295083754154879) -- (0.28924891280161,0.166947666487585);
\draw[] (0.4,0) -- (0.28924891280161,0.166947666487585);
\draw[] (0.2,0) -- (0.4,0);
\draw[] (0,0.4) -- (0.158829078088844,0.295083754154879);
\draw[] (0,0.4) -- (0,0.2);
\draw[] (0.158829078088844,0.295083754154879) -- (0.340392050716961,0.360085062591018);
\draw[] (0.28924891280161,0.166947666487585) -- (0.340392050716961,0.360085062591018);
\draw[] (0.4,0) -- (0.507386405366096,0.214426443909476);
\draw[] (0.507386405366096,0.214426443909476) -- (0.28924891280161,0.166947666487585);
\draw[] (0.507386405366096,0.214426443909476) -- (0.340392050716961,0.360085062591018);
\draw[] (0.194389338585002,0.511472795791168) -- (0.158829078088844,0.295083754154879);
\draw[] (0,0.4) -- (0.194389338585002,0.511472795791168);
\draw[] (0.194389338585002,0.511472795791168) -- (0.340392050716961,0.360085062591018);
\draw[] (0.4,0) -- (0.6,0);
\draw[] (0.6,0) -- (0.507386405366096,0.214426443909476);
\draw[] (0,0.6) -- (0.194389338585002,0.511472795791168);
\draw[] (0,0.6) -- (0,0.4);
\draw[] (0.507386405366096,0.214426443909476) -- (0.51632192714807,0.445641468992195);
\draw[] (0.340392050716961,0.360085062591018) -- (0.51632192714807,0.445641468992195);
\draw[] (0.340392050716961,0.360085062591018) -- (0.377541080267955,0.527424438107543);
\draw[] (0.194389338585002,0.511472795791168) -- (0.377541080267955,0.527424438107543);
\draw[] (0.507386405366096,0.214426443909476) -- (0.720367993373896,0.166464907980495);
\draw[] (0.6,0) -- (0.720367993373896,0.166464907980495);
\draw[] (0.377541080267955,0.527424438107543) -- (0.51632192714807,0.445641468992195);
\draw[] (0.194389338585002,0.511472795791168) -- (0.155048548464969,0.719006158403683);
\draw[] (0,0.6) -- (0.155048548464969,0.719006158403683);
\draw[] (0.507386405366096,0.214426443909476) -- (0.686726828176686,0.361868419354705);
\draw[] (0.686726828176686,0.361868419354705) -- (0.51632192714807,0.445641468992195);
\draw[] (0.720367993373896,0.166464907980495) -- (0.686726828176686,0.361868419354705);
\draw[] (0.194389338585002,0.511472795791168) -- (0.329936867383997,0.679347780724589);
\draw[] (0.329936867383997,0.679347780724589) -- (0.377541080267955,0.527424438107543);
\draw[] (0.155048548464969,0.719006158403683) -- (0.329936867383997,0.679347780724589);
\draw[] (0.8,0) -- (0.720367993373896,0.166464907980495);
\draw[] (0.6,0) -- (0.8,0);
\draw[] (0.508058617199453,0.641022485885209) -- (0.51632192714807,0.445641468992195);
\draw[] (0.377541080267955,0.527424438107543) -- (0.508058617199453,0.641022485885209);
\draw[] (0,0.8) -- (0.155048548464969,0.719006158403683);
\draw[] (0,0.8) -- (0,0.6);
\draw[] (0.329936867383997,0.679347780724589) -- (0.508058617199453,0.641022485885209);
\draw[] (0.67973159086246,0.56913545717858) -- (0.686726828176686,0.361868419354705);
\draw[] (0.67973159086246,0.56913545717858) -- (0.51632192714807,0.445641468992195);
\draw[] (0.720367993373896,0.166464907980495) -- (0.854172785710488,0.290767332144404);
\draw[] (0.854172785710488,0.290767332144404) -- (0.686726828176686,0.361868419354705);
\draw[] (0.67973159086246,0.56913545717858) -- (0.508058617199453,0.641022485885209);
\draw[] (0.279589050761926,0.849922371417845) -- (0.155048548464969,0.719006158403683);
\draw[] (0.279589050761926,0.849922371417845) -- (0.329936867383997,0.679347780724589);
\draw[] (0.8,0) -- (0.874828381578865,0.131478330490222);
\draw[] (0.720367993373896,0.166464907980495) -- (0.874828381578865,0.131478330490222);
\draw[] (0.854172785710488,0.290767332144404) -- (0.874828381578865,0.131478330490222);
\draw[] (0,0.8) -- (0.126891091107252,0.873742007888985);
\draw[] (0.155048548464969,0.719006158403683) -- (0.126891091107252,0.873742007888985);
\draw[] (0.466174482542919,0.827842017700953) -- (0.508058617199453,0.641022485885209);
\draw[] (0.466174482542919,0.827842017700953) -- (0.329936867383997,0.679347780724589);
\draw[] (0.279589050761926,0.849922371417845) -- (0.126891091107252,0.873742007888985);
\draw[] (0.844029404278513,0.484446265452234) -- (0.67973159086246,0.56913545717858);
\draw[] (0.844029404278513,0.484446265452234) -- (0.686726828176686,0.361868419354705);
\draw[] (0.844029404278513,0.484446265452234) -- (0.854172785710488,0.290767332144404);
\draw[] (0.279589050761926,0.849922371417845) -- (0.466174482542919,0.827842017700953);
\draw[] (0.68033851010831,0.796997572945071) -- (0.67973159086246,0.56913545717858);
\draw[] (0.68033851010831,0.796997572945071) -- (0.508058617199453,0.641022485885209);
\draw[] (0.8,0) -- (1,0);
\draw[] (1,0) -- (0.874828381578865,0.131478330490222);
\draw[] (0.466174482542919,0.827842017700953) -- (0.68033851010831,0.796997572945071);
\draw[] (0,1) -- (0,0.8);
\draw[] (0,1) -- (0.126891091107252,0.873742007888985);
\draw[] (1,0.2) -- (0.854172785710488,0.290767332144404);
\draw[] (1,0.2) -- (0.874828381578865,0.131478330490222);
\draw[] (0.84467940856602,0.684532224021434) -- (0.844029404278513,0.484446265452234);
\draw[] (0.84467940856602,0.684532224021434) -- (0.67973159086246,0.56913545717858);
\draw[] (0.84467940856602,0.684532224021434) -- (0.68033851010831,0.796997572945071);
\draw[] (0.2,1) -- (0.126891091107252,0.873742007888985);
\draw[] (0.2,1) -- (0.279589050761926,0.849922371417845);
\draw[] (0.4,1) -- (0.279589050761926,0.849922371417845);
\draw[] (0.4,1) -- (0.466174482542919,0.827842017700953);
\draw[] (1,0.4) -- (0.844029404278513,0.484446265452234);
\draw[] (1,0.4) -- (0.854172785710488,0.290767332144404);
\draw[] (1,0) -- (1,0.2);
\draw[] (0.2,1) -- (0,1);
\draw[] (1,0.2) -- (1,0.4);
\draw[] (0.4,1) -- (0.2,1);
\draw[] (0.6,1) -- (0.466174482542919,0.827842017700953);
\draw[] (0.6,1) -- (0.68033851010831,0.796997572945071);
\draw[] (1,0.6) -- (0.84467940856602,0.684532224021434);
\draw[] (1,0.6) -- (0.844029404278513,0.484446265452234);
\draw[] (0.6,1) -- (0.4,1);
\draw[] (0.68033851010831,0.796997572945071) -- (0.864919419945519,0.8563330434245);
\draw[] (0.84467940856602,0.684532224021434) -- (0.864919419945519,0.8563330434245);
\draw[] (1,0.4) -- (1,0.6);
\draw[] (0.8,1) -- (0.68033851010831,0.796997572945071);
\draw[] (0.8,1) -- (0.6,1);
\draw[] (0.8,1) -- (0.864919419945519,0.8563330434245);
\draw[] (1,0.6) -- (1,0.8);
\draw[] (1,0.8) -- (0.84467940856602,0.684532224021434);
\draw[] (1,0.8) -- (0.864919419945519,0.8563330434245);
\draw[] (1,1) -- (0.8,1);
\draw[] (1,1) -- (0.864919419945519,0.8563330434245);
\draw[] (1,0.8) -- (1,1);
  \end{tikzpicture}
  \caption{Step 3}
  \end{subfigure}
  \hspace{30pt}
  \begin{subfigure}{.3\textwidth}
    \begin{tikzpicture}[scale=3.5]
    \draw[fill=red] (0.377541080267955,0.527424438107543) -- (0.51632192714807,0.445641468992195) -- (0.508058617199453,0.641022485885209);
\draw[fill=cyan] (0.51632192714807,0.445641468992195) -- (0.340392050716961,0.360085062591018) -- (0.507386405366096,0.214426443909476);
\draw[fill=cyan] (0.377541080267955,0.527424438107543) -- (0.194389338585002,0.511472795791168) -- (0.340392050716961,0.360085062591018);
\draw[fill=cyan] (0.377541080267955,0.527424438107543) -- (0.340392050716961,0.360085062591018) -- (0.51632192714807,0.445641468992195);
\draw[fill=cyan] (0.51632192714807,0.445641468992195) -- (0.507386405366096,0.214426443909476) -- (0.686726828176686,0.361868419354705);
\draw[fill=cyan] (0.377541080267955,0.527424438107543) -- (0.329936867383997,0.679347780724589) -- (0.194389338585002,0.511472795791168);
\draw[fill=cyan] (0.377541080267955,0.527424438107543) -- (0.51632192714807,0.445641468992195) -- (0.508058617199453,0.641022485885209);
\draw[fill=cyan] (0.508058617199453,0.641022485885209) -- (0.329936867383997,0.679347780724589) -- (0.377541080267955,0.527424438107543);
\draw[fill=cyan] (0.51632192714807,0.445641468992195) -- (0.686726828176686,0.361868419354705) -- (0.67973159086246,0.56913545717858);
\draw[fill=cyan] (0.51632192714807,0.445641468992195) -- (0.67973159086246,0.56913545717858) -- (0.508058617199453,0.641022485885209);
\draw[fill=cyan] (0.466174482542919,0.827842017700953) -- (0.329936867383997,0.679347780724589) -- (0.508058617199453,0.641022485885209);
\draw[fill=cyan] (0.68033851010831,0.796997572945071) -- (0.508058617199453,0.641022485885209) -- (0.67973159086246,0.56913545717858);
\draw[fill=cyan] (0.68033851010831,0.796997572945071) -- (0.466174482542919,0.827842017700953) -- (0.508058617199453,0.641022485885209);
\draw[] (0,0) -- (0.2,0);
\draw[] (0.2,0) -- (0.129566038401955,0.132370390921097);
\draw[] (0,0) -- (0.129566038401955,0.132370390921097);
\draw[] (0,0) -- (0,0.2);
\draw[] (0,0.2) -- (0.129566038401955,0.132370390921097);
\draw[] (0.2,0) -- (0.28924891280161,0.166947666487585);
\draw[] (0.28924891280161,0.166947666487585) -- (0.129566038401955,0.132370390921097);
\draw[] (0.158829078088844,0.295083754154879) -- (0.129566038401955,0.132370390921097);
\draw[] (0,0.2) -- (0.158829078088844,0.295083754154879);
\draw[] (0.158829078088844,0.295083754154879) -- (0.28924891280161,0.166947666487585);
\draw[] (0.4,0) -- (0.28924891280161,0.166947666487585);
\draw[] (0.2,0) -- (0.4,0);
\draw[] (0,0.4) -- (0.158829078088844,0.295083754154879);
\draw[] (0,0.4) -- (0,0.2);
\draw[] (0.158829078088844,0.295083754154879) -- (0.340392050716961,0.360085062591018);
\draw[] (0.28924891280161,0.166947666487585) -- (0.340392050716961,0.360085062591018);
\draw[] (0.4,0) -- (0.507386405366096,0.214426443909476);
\draw[] (0.507386405366096,0.214426443909476) -- (0.28924891280161,0.166947666487585);
\draw[] (0.507386405366096,0.214426443909476) -- (0.340392050716961,0.360085062591018);
\draw[] (0.194389338585002,0.511472795791168) -- (0.158829078088844,0.295083754154879);
\draw[] (0,0.4) -- (0.194389338585002,0.511472795791168);
\draw[] (0.194389338585002,0.511472795791168) -- (0.340392050716961,0.360085062591018);
\draw[] (0.4,0) -- (0.6,0);
\draw[] (0.6,0) -- (0.507386405366096,0.214426443909476);
\draw[] (0,0.6) -- (0.194389338585002,0.511472795791168);
\draw[] (0,0.6) -- (0,0.4);
\draw[] (0.507386405366096,0.214426443909476) -- (0.51632192714807,0.445641468992195);
\draw[] (0.340392050716961,0.360085062591018) -- (0.51632192714807,0.445641468992195);
\draw[] (0.340392050716961,0.360085062591018) -- (0.377541080267955,0.527424438107543);
\draw[] (0.194389338585002,0.511472795791168) -- (0.377541080267955,0.527424438107543);
\draw[] (0.507386405366096,0.214426443909476) -- (0.720367993373896,0.166464907980495);
\draw[] (0.6,0) -- (0.720367993373896,0.166464907980495);
\draw[] (0.377541080267955,0.527424438107543) -- (0.51632192714807,0.445641468992195);
\draw[] (0.194389338585002,0.511472795791168) -- (0.155048548464969,0.719006158403683);
\draw[] (0,0.6) -- (0.155048548464969,0.719006158403683);
\draw[] (0.507386405366096,0.214426443909476) -- (0.686726828176686,0.361868419354705);
\draw[] (0.686726828176686,0.361868419354705) -- (0.51632192714807,0.445641468992195);
\draw[] (0.720367993373896,0.166464907980495) -- (0.686726828176686,0.361868419354705);
\draw[] (0.194389338585002,0.511472795791168) -- (0.329936867383997,0.679347780724589);
\draw[] (0.329936867383997,0.679347780724589) -- (0.377541080267955,0.527424438107543);
\draw[] (0.155048548464969,0.719006158403683) -- (0.329936867383997,0.679347780724589);
\draw[] (0.8,0) -- (0.720367993373896,0.166464907980495);
\draw[] (0.6,0) -- (0.8,0);
\draw[] (0.508058617199453,0.641022485885209) -- (0.51632192714807,0.445641468992195);
\draw[] (0.377541080267955,0.527424438107543) -- (0.508058617199453,0.641022485885209);
\draw[] (0,0.8) -- (0.155048548464969,0.719006158403683);
\draw[] (0,0.8) -- (0,0.6);
\draw[] (0.329936867383997,0.679347780724589) -- (0.508058617199453,0.641022485885209);
\draw[] (0.67973159086246,0.56913545717858) -- (0.686726828176686,0.361868419354705);
\draw[] (0.67973159086246,0.56913545717858) -- (0.51632192714807,0.445641468992195);
\draw[] (0.720367993373896,0.166464907980495) -- (0.854172785710488,0.290767332144404);
\draw[] (0.854172785710488,0.290767332144404) -- (0.686726828176686,0.361868419354705);
\draw[] (0.67973159086246,0.56913545717858) -- (0.508058617199453,0.641022485885209);
\draw[] (0.279589050761926,0.849922371417845) -- (0.155048548464969,0.719006158403683);
\draw[] (0.279589050761926,0.849922371417845) -- (0.329936867383997,0.679347780724589);
\draw[] (0.8,0) -- (0.874828381578865,0.131478330490222);
\draw[] (0.720367993373896,0.166464907980495) -- (0.874828381578865,0.131478330490222);
\draw[] (0.854172785710488,0.290767332144404) -- (0.874828381578865,0.131478330490222);
\draw[] (0,0.8) -- (0.126891091107252,0.873742007888985);
\draw[] (0.155048548464969,0.719006158403683) -- (0.126891091107252,0.873742007888985);
\draw[] (0.466174482542919,0.827842017700953) -- (0.508058617199453,0.641022485885209);
\draw[] (0.466174482542919,0.827842017700953) -- (0.329936867383997,0.679347780724589);
\draw[] (0.279589050761926,0.849922371417845) -- (0.126891091107252,0.873742007888985);
\draw[] (0.844029404278513,0.484446265452234) -- (0.67973159086246,0.56913545717858);
\draw[] (0.844029404278513,0.484446265452234) -- (0.686726828176686,0.361868419354705);
\draw[] (0.844029404278513,0.484446265452234) -- (0.854172785710488,0.290767332144404);
\draw[] (0.279589050761926,0.849922371417845) -- (0.466174482542919,0.827842017700953);
\draw[] (0.68033851010831,0.796997572945071) -- (0.67973159086246,0.56913545717858);
\draw[] (0.68033851010831,0.796997572945071) -- (0.508058617199453,0.641022485885209);
\draw[] (0.8,0) -- (1,0);
\draw[] (1,0) -- (0.874828381578865,0.131478330490222);
\draw[] (0.466174482542919,0.827842017700953) -- (0.68033851010831,0.796997572945071);
\draw[] (0,1) -- (0,0.8);
\draw[] (0,1) -- (0.126891091107252,0.873742007888985);
\draw[] (1,0.2) -- (0.854172785710488,0.290767332144404);
\draw[] (1,0.2) -- (0.874828381578865,0.131478330490222);
\draw[] (0.84467940856602,0.684532224021434) -- (0.844029404278513,0.484446265452234);
\draw[] (0.84467940856602,0.684532224021434) -- (0.67973159086246,0.56913545717858);
\draw[] (0.84467940856602,0.684532224021434) -- (0.68033851010831,0.796997572945071);
\draw[] (0.2,1) -- (0.126891091107252,0.873742007888985);
\draw[] (0.2,1) -- (0.279589050761926,0.849922371417845);
\draw[] (0.4,1) -- (0.279589050761926,0.849922371417845);
\draw[] (0.4,1) -- (0.466174482542919,0.827842017700953);
\draw[] (1,0.4) -- (0.844029404278513,0.484446265452234);
\draw[] (1,0.4) -- (0.854172785710488,0.290767332144404);
\draw[] (1,0) -- (1,0.2);
\draw[] (0.2,1) -- (0,1);
\draw[] (1,0.2) -- (1,0.4);
\draw[] (0.4,1) -- (0.2,1);
\draw[] (0.6,1) -- (0.466174482542919,0.827842017700953);
\draw[] (0.6,1) -- (0.68033851010831,0.796997572945071);
\draw[] (1,0.6) -- (0.84467940856602,0.684532224021434);
\draw[] (1,0.6) -- (0.844029404278513,0.484446265452234);
\draw[] (0.6,1) -- (0.4,1);
\draw[] (0.68033851010831,0.796997572945071) -- (0.864919419945519,0.8563330434245);
\draw[] (0.84467940856602,0.684532224021434) -- (0.864919419945519,0.8563330434245);
\draw[] (1,0.4) -- (1,0.6);
\draw[] (0.8,1) -- (0.68033851010831,0.796997572945071);
\draw[] (0.8,1) -- (0.6,1);
\draw[] (0.8,1) -- (0.864919419945519,0.8563330434245);
\draw[] (1,0.6) -- (1,0.8);
\draw[] (1,0.8) -- (0.84467940856602,0.684532224021434);
\draw[] (1,0.8) -- (0.864919419945519,0.8563330434245);
\draw[] (1,1) -- (0.8,1);
\draw[] (1,1) -- (0.864919419945519,0.8563330434245);
\draw[] (1,0.8) -- (1,1);
\draw[fill=red] (0.377541080267955,0.527424438107543) -- (0.51632192714807,0.445641468992195) -- (0.508058617199453,0.641022485885209);
  \end{tikzpicture}
  \caption{Step 4}
  \end{subfigure}
  \caption{Build patch for $\# S(K) = 12$}
  \label{fig:buildpatch}
\end{figure}
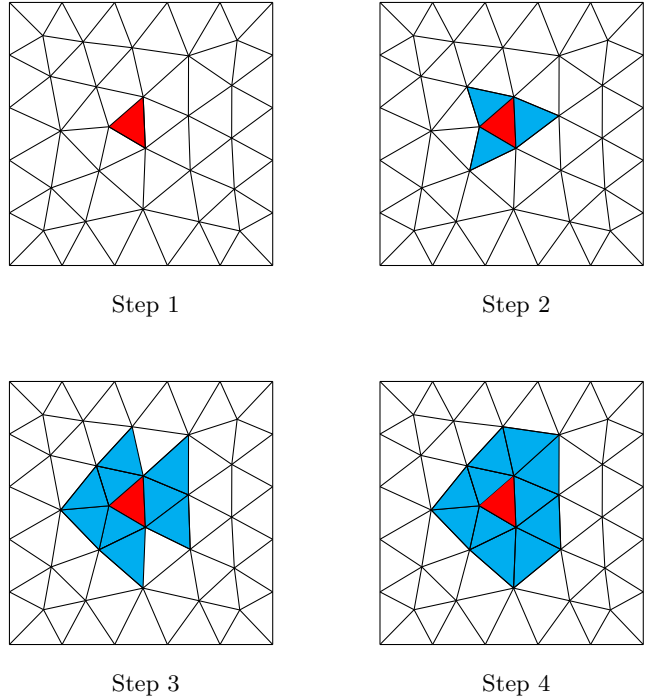

For element $K$, we collect all collocation points in a set $\mathcal
I_K$:
\begin{displaymath}
  \mathcal I_K\triangleq\left\{\bm x_{\widetilde K}\ |\ \bm
  x_{\widetilde K}\ \text{is the barycenter of}\ \widetilde
  K,\ \forall\widetilde K\in S(K)\right\}.
\end{displaymath}

Let $U_h$ be the space consisting of piecewise constant functions:
\begin{displaymath}
    U_h\triangleq\left\{ v\in L^2(\Omega)\ \big |\ v|_K \in \mathbb
    P_0(K),\ \forall K\in \mathcal T_h\right\},
\end{displaymath}
where $\mathbb P_n$ is the polynomial space of degree not greater than
$n$.  For any $v\in U_h$, we reconstruct a $m$th-order polynomial
denoted by $\mathcal R_K^m v$ on $S(K)$ by the following least squares
problem:
\begin{equation}
  \mathcal R_K^mv= \mathop{\arg\min}_{p\in\mathbb P_m(S(K))}
  \ \sum_{\bm x\in \mathcal I_K}|v(\bm x)-p(\bm x)|^2.
  \label{eq:lsproblem}
\end{equation}

The uniqueness condition for the problem \eqref{eq:lsproblem} is
provided by the condition $\# S(K)\geq\text{dim}(\mathbb P_m) $ and
the following assumption \cite{li2012efficient, li2016discontinuous}:
\begin{assumption}
  For $\forall K\in\mathcal T_h$ and $\forall p \in \mathbb
  P_m(S(K))$, problem \eqref{eq:lsproblem} satisfies
  \begin{displaymath}
    p|_{\mathcal I(K)}=\bm0\quad\Longrightarrow\quad p|_{S(K)}\equiv0.
  \end{displaymath}
  \label{as:unique}
\end{assumption}
Hereafter, we assume the uniqueness condition for \eqref{eq:lsproblem}
always holds. For any $g \in U_h$, we restrict the definition domain
of the polynomial $\mathcal R_K^m g$ on element $K$ to define a global
reconstruction operator which is denoted by $\mathcal R^m$:
\begin{displaymath}
  \mathcal R^mg|_K = (\mathcal R_K^mg)|_K,\quad \forall K\in \mathcal
  T_h.
\end{displaymath} 
Then we extend the reconstruction operator to an operator defined on
$C^0(\Omega)$, still denoted as $\mathcal R^m$:
\begin{displaymath}
  \mathcal R^mu=\mathcal R^m\tilde u,\quad \tilde u\in U_h, \quad
  \tilde u(\bm x_K)=u(\bm x_K), \quad \forall u\in C^0(\Omega).
\end{displaymath}

We note that $\mathcal R^m$ is a linear operator whose image is
actually a piecewise $m$th-order polynomial space which is denoted as
\begin{displaymath}
 V_{h}^{m} = \mathcal R^m U_h.
\end{displaymath} 
{
Further, we give a group of basis functions of the space $V_h^m$. We
define $w_K(\bm x) \in C^0(\Omega)$ such that
\begin{displaymath}
  w_K(\bm x) = \begin{cases} 1, \quad &\bm x = \bm x_K, \\ 0, \quad
    &\bm x \in \widetilde K, \quad \widetilde K \neq K.
  \end{cases}
\end{displaymath}
Then we denote $\left\{ \lambda_K\ |\ \lambda_K = \mc R^m w_K
\right\}$ as a group of basis functions. Given ${\lambda_K}$, we
may write the reconstruction operator in an explicit way:
\begin{equation}
  \mc R^m g = \sum_{K \in \MTh} g(\bm x_K) \lambda_K(x), \quad \forall
  g \in C^0(\Omega).
  \label{eq:explicit}
\end{equation}
From \eqref{eq:explicit}, it is clear that the degrees of freedom of
$\mc R^m$ are the values of the unknown function at the collocation
points of all elements in partition. We present a 2D example in
Section \ref{sec:2dexample} to demonstrate the reconstruction process
and the implementation of basis functions.

We note that} $\mathcal R^m u(\forall u\in C^0(\Omega))$ may be
discontinuous across the inter-element boundaries. The fact inspires
us to share some well-developed theories of DG methods and enjoy its
advantages.

We first introduce the traditional average and jump notations in DG
method. Let $e$ be an interior edge shared by two adjacent elements
$e=\partial K^{+} \cap \partial K^{-}$ with the unit outward normal
vector $\bm{\mathrm n}^{+}$ and $\bm{\mathrm n}^{-}$, respectively.
Let $v$ and $\bm{v}$ be the scalar-valued and vector-valued functions
on $\mathcal T_h$, respectively, we define the $average$ operator $\{
\cdot \}$ as follows:
\begin{displaymath}
\{v\}=\frac{1}{2}(v^{+} + v^{-}), \quad \{ \bm{v} \} =
\frac{1}{2}(\bm{v}^{+} + \bm{v}^{-}) , \quad\text{on }\ e\in\mathcal
E_h^i,
\end{displaymath}
with $v^+=v|_{K^+},\ v^-=v|_{K^-},\ \bm v^+=\bm v|_{K^+},\ \bm v^-=\bm
v|_{K^-}$.

Further, we set the $jump$ operator $\lj \cdot \rj$ as
\begin{displaymath}
  \begin{aligned}
    \lj v \rj =v^{+} \bm{\mathrm n}^{+} + v^{-} \bm{\mathrm n}^{-},
    \quad \lj \bm{v} \rj =\bm{v}^{+}\cdot \bm{\mathrm
    n}^{+}+\bm{v}^{-}\cdot \bm{\mathrm n}^{-}, \\
    \lj \bm{v} \otimes\bm{\mathrm n} \rj =\bm{v}^{+}\otimes
    \bm{\mathrm n}^{+}+\bm{v}^{-}\otimes \bm{\mathrm n}^{-},\quad
    \text{on}\ e\in\mathcal E_h^i.  \\
\end{aligned}
\end{displaymath}

For $e \in \mathcal E^b_h$, we set
\begin{displaymath}
  \begin{aligned}
    \{v\}=v&,\quad \{\bm v\}=\bm v,\quad \lj v \rj=v\bm{\mathrm n},
    \\ \lj \bm v \rj=\bm v\cdot\bm{\mathrm n}&, \quad \lj
    \bm{v}\otimes \bm{\mathrm n} \rj = \bm{v}\otimes \bm{\mathrm n}
    ,\quad\text{on}\ e \in \mathcal E^b_h.\\
  \end{aligned}
\end{displaymath}

Now we will present the error analysis of $\mathcal R^m$. We begin by
defining broken Sobolev spaces of composite order $\bm{\mathrm s}=\{
s_K\geq0: \forall K\in\mathcal T_h\}$:
\begin{displaymath}
  \begin{aligned}
    H^{\bm{\mathrm s}}(\Omega,\mathcal T_h)\triangleq\{u\in
    L^2(\Omega)&: u|_K \in H^{s_K}(K),\forall K\in\mathcal T_h\}, \\
  \end{aligned}
\end{displaymath}
where $H^{s_K}(K)$ is the standard Sobolev spaces on element $K$.  The
associated broken norm is defined as
\begin{displaymath}
  \begin{aligned}
     \|u\|_{H^{\bm{\mathrm s}}(\Omega,\mathcal T_h)}^2=
     \sum_{K\in\mathcal T_h}\|u\|_{H^{s_K}(K)}^2,
   \end{aligned}
\end{displaymath}
where $\|\cdot\|_{H^{s_K}(K)}$ is the standard Sobolev norm on element
$K$. For $\bm u\in [H^{\bm{\mathrm s}}(\Omega, \mathcal T_h)]^d$, the
norm is defined as
\begin{displaymath}
  \|\bm u\|_{H^{\bm{\mathrm s}}(\Omega,\mathcal T_h)}^2= \sum_{i=1}^d
  \|\bm u_i\|_{H^{\bm{\mathrm s}}(\Omega, \mathcal T_h)}^2.
\end{displaymath}
When $s_K=s$ for all elements in $\mathcal T_h$, we simply write
$H^s(\Omega,\mathcal T_h)$ and $[H^s(\Omega, \mathcal T_h)]^d$.


Then we define a constant $\Lambda(m,\mathcal{I}_K)$ for $K\in
\mathcal T_h$:
\begin{equation}
  \label{eq:constant}
  \Lambda(m, \mathcal{I}_K) \triangleq \max_{p\in \mathbb{P}_m(S(K))}
  \frac{\max_{\bm{x} \in S(K)} |p(\bm{x})|}{\max_{\bm{x} \in
      \mathcal{I}_K} |p(\bm{x})|},
\end{equation}
the Assumption \ref{as:unique} is equivalent to
\begin{displaymath}
\Lambda(m, \mathcal{I}_K) < \infty.
\end{displaymath}
The uniform upper bound of $\Lambda(m, \mathcal I_K)$ exists if
element patches are convex and the triangulation is quasi-uniform
\cite{li2012efficient}. We also refer to \cite{li2016discontinuous}
for the estimate of $\Lambda(m, \mathcal I_K)$ in more general cases
such as polygonal partition and non-convex element patch. We denote by
$\Lambda_m$ the uniform upper bound of $\Lambda(m, \mathcal I_K)$.


 With $\Lambda_m$, we have the following estimates.
\begin{lemma}
  Let $g\in H^{m+1}(\Omega)(m\geq0)$ and $ K\in\mathcal T_h$, then
  \begin{equation}
    \|g-\mathcal{R}^m g\|_{L^2(K)}\lesssim \Lambda_{m}
    h^{m+1}\|g\|_{H^{m+1}(S(K))}.
    \label{eq:reconL2error}
  \end{equation}
  \label{le:reconL2error}
\end{lemma}
For convenience, the symbol $\lesssim$ and $\gtrsim$ will be used in
this paper. That $X_1\lesssim Y_1$ and $X_2\gtrsim Y_2$ mean that
$X_1\leq C_1Y_1$ and $X_2\geq C_2Y_2$ for some positive constants
$C_1$ and $C_2$ which are independent of mesh size $h$.
\begin{lemma}
  Let $g\in H^{m+1}(\Omega)(m\geq0)$ and $ K\in\mathcal T_h$, then
  \begin{equation}
    \|g-\mathcal R^m g\|_{L^2(\partial K)}\lesssim \Lambda_{m}
    h^{m+\frac12}\|g\|_{H^{m+1}(S(K))}.
    \label{eq:recontraceinequality}
  \end{equation}
  \label{th:recontraceinequality}
\end{lemma}
For the standard Sobolev norm, we have the following estimates:
\begin{lemma}
  Let $g\in H^{m+1}(\Omega)(m\geq0)$ and $K\in\mathcal T_h$, then
  \begin{equation}
    \|g-\mathcal R^mg\|_{H^1(K)}\lesssim \Lambda_{m}
    h^{m}\|g\|_{H^{m+1}(S(K))}.
    \label{eq:reconSobleverror}
  \end{equation}
  \label{th:reconSobleverror}
\end{lemma}
We refer to \cite{li2012efficient, li2016discontinuous} for detailed
proofs and more discuss about $S(K)$ and $\# S(K)$. Here we note that
one of the conditions of guaranteeing the uniform upper bound
$\Lambda_m$ is $\# S(K)$ should be much larger than
$\text{dim}(\mathbb P_m)$. In Section \ref{sec:infsuptest} we will
list the values of $\# S(K)$ used in all numerical experiments.

Finally, we derive the estimate in DG energy norm. For the
scalar-valued function, the DG energy norm is defined as:
\begin{displaymath}
  \begin{aligned}
  \|u\|_{\mathrm{DG}}^2&\triangleq\sum_{K\in\mathcal
    T_h}|u|_{H^1(K)}^2 + \sum_{e\in\mathcal E_h} \frac1{h_e}\|\lj u
  \rj\|_{L^2(e)}^2,\quad \forall u\in H^1(\Omega, \mathcal T_h),\\
\end{aligned}
\end{displaymath}
\begin{theorem}
  Let $g\in H^{m+1}(\Omega)(m\geq0)$, then
  \begin{equation}
    \|g-\mathcal R^mg\|_{\mathrm{DG}}\lesssim \Lambda_{m} h^{m}
    \|g\|_{H^{m+1}(\Omega)}.
    \label{eq:interpolation}
  \end{equation}
  \label{th:interpolationerrorDG}
\end{theorem}
\begin{proof}
  From Lemma \ref{th:reconSobleverror}, we have
  \begin{displaymath}
    \begin{aligned}
      \sum_{K\in\mathcal T_h}|g-\mathcal R^mg|_{H^1(K)}&\lesssim
      \sum_{K\in\mathcal T_h}\Lambda_m
      h^{m}\|g\|_{H^{m+1}(S(K))}\\ &\lesssim \Lambda_m
      h^{m}\|g\|_{H^{m+1}(\Omega)}.\\
    \end{aligned}
  \end{displaymath}
  For any $e\in\mathcal E_h^i$ shared by elements $K_1$ and $K_2$, we
  have
  \begin{displaymath}
    \frac1{h_e}\|\lj g-\mathcal R^m g\rj \|_{L^2(e)}^2\leq
    \frac{1}{h_e}\left( \|g-\mathcal R^m g\|_{L^2(\partial K_1)}^2
    +\|g-\mathcal R^m g \|_{L^2(\partial K_2)}^2\right).
  \end{displaymath}
  From Lemma \ref{th:recontraceinequality}, we get
  \begin{displaymath}
    \begin{aligned}
    \frac1{h_e}\|g-\mathcal R^m g\|_{L^2(\partial
      K_1)}^2&\lesssim\Lambda_m
    h^{2m}\|g\|_{H^{m+1}(K_1)}^2,\\ \frac1{h_e}\|g-\mathcal R^m
    g\|_{L^2(\partial K_2)}^2&\lesssim\Lambda_m
    h^{2m}\|g\|_{H^{m+1}(K_2)}^2.\\
  \end{aligned}
  \end{displaymath}
  For any $e\in\mathcal E_h^b$, assume $e$ is a face of element $K$,
  we have
  \begin{displaymath}
    \begin{aligned}
     \frac1{h_e}\|\lj g-\mathcal R^m g\rj \|_{L^2(e)}^2&\leq
     \frac{1}{h_e}\|g-\mathcal R^m g\|_{L^2(\partial K)}^2\\ &\lesssim
     \Lambda_m h^{2m}\|g\|_{H^{m+1}(K)}^2.\\
   \end{aligned}
  \end{displaymath}
  Combining the above inequalities gives the estimate
  \eqref{eq:interpolation}, which completes the proof.
\end{proof}

For the vector-valued function, the DG energy norm is defined as:
\begin{displaymath}
  \|\bm u\|_{\mathrm{DG}}^2\triangleq\sum_{i=1}^d\|\bm u_i\|_{
    \mathrm{DG}}^2,\quad\forall \bm u\in [H^1(\Omega,\mathcal T_h)]^d,
\end{displaymath}
and the reconstruction operator is defined component-wisely for
$[U_h]^d$, still denoted by $\mathcal R^m$:
\begin{displaymath}
  \begin{aligned}
    {\mathcal R}^m\bm v&=[\mathcal R^m\bm v_i]^d, \quad 1\leq i\leq
    d,\quad \forall \bm v \in [U_h]^d.\\
  \end{aligned}
\end{displaymath}

Then the operator can be extended on $[C^0(\Omega)]^d$ and the
corresponding estimate is written as:
\begin{theorem}
  let $\bm g\in [H^{m+1}(\Omega)]^d(m\geq0)$, then
  \begin{displaymath}
    \|\bm g-{\mathcal R}^m\bm g\|_{\mathrm{DG}}\lesssim \Lambda_m
    h^{m}\|\bm g\|_{H^{m+1}(\Omega)}.
  \end{displaymath}
\end{theorem}
\begin{proof}
  It is a direct extension from Theorem \ref{th:interpolationerrorDG}.
\end{proof}

\section{The weak form of the stokes problem}
\label{sec:weakform}
In this section, we consider the incompressible Stokes problem with
Dirichlet boundary condition, which seeks the velocity field $\bm u$
and its associated pressure $p$ satisfying
\begin{equation}
  \begin{aligned}
    -\Delta\bm u+\nabla p&=\bm f \qquad \text{in}
    \ \Omega,\\ \nabla\cdot\bm u&=0 \qquad\text{in}\ \Omega,\\ \bm
    u&=\bm g \qquad\text{on}\ \partial\Omega,\\
  \end{aligned}
  \label{eq:stokes}
\end{equation}
where $\bm f$ is the given source term and $\bm g$ is a Dirichlet
boundary condition that satisfies the compatibility condition
\begin{displaymath}
  \int_{\partial \Omega} \bm g\cdot\bm{\mathrm n}\mathrm ds=0.
\end{displaymath}

For positive integer $k,k'$, we define the following finite element
spaces to approximate velocity and pressure:
\begin{displaymath}
  \begin{aligned}
    \bm V_h^k=[V_h^k]^d,\quad Q_{h}^{k'}=V_{h}^{k'}.\\
  \end{aligned}
\end{displaymath}

We note that finite element spaces $\bm V_{h}^{k}$ and $Q_{h}^{k'}$
are the subspace of the common discontinuous Galerkin finite element
spaces, which implies that the interior penalty discontinuous Galerkin
method \cite{hansbo2002discontinuous,montlaur2008discontinuous} can be
directly applied to the Stokes problem \eqref{eq:stokes}.

For a vector $\bm u$, we define the second-order tensor $\nabla\bm u$
by
\begin{displaymath}
  (\nabla\bm u)_{i,j}=\frac{\partial\bm u_i}{\partial x_j}, \quad
  1\leq i,j\leq d.
\end{displaymath}

The discrete problem for the Stokes problem \eqref{eq:stokes} is as:
find $(\bm u_h, p_h)\in \bm V_{h}^{k} \times Q_{h}^{k'}$ such that
\begin{equation}
  \begin{aligned}
    a(\bm u_h, \bm v_h)+b(\bm v_h, p_h)&=l(\bm v_h),\quad \forall \bm
    v_h\in \bm V_{h}^{k},\\ b(\bm u_h, q_h)&=(q_h,\bm{\mathrm
      n}\cdot\bm g)_{\partial \Omega},\quad \forall q_h\in
    Q_{h}^{k'},\\
  \end{aligned}
  \label{eq:weakform}
\end{equation}
where symmetric bilinear form $a(\cdot, \cdot)$ is given by
\begin{equation}
  \begin{aligned}
    a(\bm{u}, \bm{v})&=\int_{\Omega}\nabla \bm{u} : \nabla \bm{v}
    \dx\\ &-\int_{\mathcal E_h} (\{\nabla \bm{u}\} : \lj \bm{v} \otimes
    \bm{n} \rj+\lj \bm{u} \otimes \bm{n} \rj : \{\nabla
    \bm{v}\})\mathrm ds \\ &+\int_{\mathcal E_h}\eta\lj \bm{u} \otimes
    \bm{n} \rj : \lj \bm{v} \otimes \bm{n} \rj \mathrm ds, \quad
    \forall \bm{u},\bm{v}\in [H^1(\Omega, \mathcal T_h)]^d.\\
  \end{aligned}
  \label{eq:ellipticform}
\end{equation}

The term $\eta$ is referred to as the penalty parameter which is
defined on $\mathcal E_h$ by
\begin{displaymath}
  \eta|_e = \eta_e,\quad \forall e\in \mathcal E_h,
\end{displaymath}
and will be specified later. The bilinear form $b(\cdot, \cdot)$ and
the linear form $l(\cdot)$ are defined as
\begin{equation}
  \begin{aligned}
    b(\bm v, p)&=-\int_{\Omega} p\nabla\cdot\bm v\dx + \int_{\mathcal
    E_h} \{p\}\lj \bm{v} \rj\mathrm ds,\\ 
    l(\bm v)&=\int_{\Omega}\bm f
    \cdot \bm v\dx-\int_{\mathcal E_h^b}\bm g \cdot (\nabla
    \bm v\cdot\bm{\mathrm n}) \mathrm ds +
    \int_{\mathcal E_h^b}\eta \bm g\cdot \bm v\mathrm ds,\\
  \end{aligned}
  \label{eq:divergenceform}
\end{equation}
for $\forall \bm v\in [H^1(\Omega,\mathcal T_h)]^d$ and $p\in
L^2(\Omega)$.

Now we present the standard continuity and coercivity properties of
the bilinear form $a(\cdot. \cdot)$. Actually the bilinear form
$a(\cdot, \cdot)$ is a direct extension from the interior penalty
bilinear form used for solving the elliptic problems
\cite{arnold1982interior}. It is easy to extend the theoretical
results of solving the elliptic problems to $a(\cdot, \cdot)$.

\begin{lemma}
  The bilinear form $a(\cdot, \cdot)$, defined in
  \eqref{eq:ellipticform}, is continuous when $\eta\geq0$. The
  following inequality holds:
  \begin{displaymath}
    |a(\bm u, \bm v)|\lesssim\|\bm u\|_{\mathrm{\mathrm{DG}}}\|\bm
    v\|_{\mathrm{\mathrm{DG}}},\quad \forall \bm{u,v}\in [H^1(\Omega, \mathcal
      T_h)]^d.
  \end{displaymath}
  \label{le:ellipticcontinuous}
\end{lemma}
\begin{lemma}
  Let
  \begin{displaymath}
    \eta|_e=\frac{\mu}{h_e},\quad \forall e\in \mathcal E_h,
  \end{displaymath}
  where $\mu$ is a positive constant. With sufficiently large $\mu$,
  the following inequality holds:
  \begin{displaymath}
    |a(\bm u_h, \bm u_h)|\gtrsim \|\bm u_h\|_{\mathrm{\mathrm{DG}}}^2, \quad
    \forall \bm u_h\in \bm V_{h}^{k}.
  \end{displaymath}
  \label{le:ellipticcoercivity}
\end{lemma}
The detailed proofs of Lemma \ref{le:ellipticcontinuous} and Lemma
\ref{le:ellipticcoercivity} could be found in
\cite{arnold2002unified,hansbo2002discontinuous,
  montlaur2008discontinuous}. We also refer to
\cite{arnold2002unified} where a unified method is employed to analyse
the choices of the penalty parameter $\eta$.

For $b(\cdot,\cdot)$, we have the analogous continuity property.
\begin{lemma}
  The bilinear form $b(\cdot, \cdot)$, defined in
  \eqref{eq:divergenceform}, is continuous. The following inequality
  holds:
  \begin{displaymath}
    |b(\bm v, q)|\lesssim\|\bm v\|_{\mathrm{\mathrm{DG}}} \| q
    \|_{L^2(\Omega)},\quad \forall \bm{v}\in [H^1(\Omega,
      \MTh)]^d,\ \forall q\in L^2(\Omega).
  \end{displaymath}
  \label{le:divergencecontinuous}
\end{lemma}


Besides the continuity of $a(\cdot, \cdot), b(\cdot,\cdot)$ and the
coercivity of $a(\cdot, \cdot)$, 
{ the existence of a
stable finite element approximation solution $(\bm u_h, p_h)$ depends
on choosing a pair of spaces $\bm V_{h}^{k}$ and $Q_h^{k'}$ such that
the following inf-sup condition holds \cite{boffi2013mixed}
}:
\begin{equation}
  \mathop{\inf\qquad\sup}_{q_h\in Q_{h}^{k'}\ \bm v_h\in \bm
    V_{h}^{k}} \frac{b(\bm v_h, q_h)}{ \|\bm
    v_h\|_{\mathrm{\mathrm{DG}}}\|q_h\|_{L^2(\Omega)}}\geq \beta,
  \label{eq:infsupcondition}
\end{equation}
where $\beta$ is a positive constant.

The finite element space we build depends on the collocation points
and element patches, the theoretical verification of the inf-sup
condition for the pair $\bm V_{h}^{k} \times Q_{h}^{k'}$ is very
difficult in all situations. Chapelle and Bathe \cite{chapelle1993inf}
propose a numerical test on whether the inf-sup condition is passed
for a given finite element discretization.  In next section, we will
carry out a series of numerical evaluations for different $k$ and $k'$
to give an indication of the verification of the inf-sup condition.

Then if the inf-sup condition holds, we could state a standard priori
error estimate of the mixed method \eqref{eq:weakform}.
\begin{theorem}
  Let the exact solution $(\bm u, p)$ to the Stokes problem
  \eqref{eq:stokes} belong to $[H^{k+1}(\Omega)]^d
  \times H^{k'+1}(\Omega)$ with $k\geq 1$ and $k'\geq
  0$, and let $(\bm u_h, p_h)$ be the numerical solution to
  \eqref{eq:weakform}, and assume that the inf-sup condition
  \eqref{eq:infsupcondition} holds and the penalty parameter $\eta$ is
  set properly. Then the following estimate holds:
  \begin{equation}
    \|\bm u- \bm u_h\|_{\mathrm{\mathrm{DG}}}+\|p -
    p_h\|_{L^2(\Omega)}\lesssim h^s\left( \|\bm u\|_{H^{k+1}(
    \Omega)}+\|p\|_{H^{k'+1}(\Omega)}\right),
    \label{eq:prioriestimate}
  \end{equation}
  \label{th:prioriestimate}
  where $s=\min(k,k'+1)$.
\end{theorem}
\begin{proof}
We define $\bm Z (\bm g)\subset \bm V_h^{k}$ by
\begin{equation}\label{eq:kernel_space}
\bm Z(\bm g)=\{\bm v \in \bm V_h:b(\bm v ,q )=
\int_{\mathcal E_h^b} \bm g\cdot \bm n q \ds, \ \forall q\in
Q_h^{k'} \}.
\end{equation}

Consider $\bm w \in \bm Z (\bm g)$ and $q\in Q_h^{k'}$. Since Lemma
\ref{le:ellipticcoercivity}, we have
\begin{align*}
\|\bm w -\bm u_h\|_{\mathrm{DG}}^2 &\lesssim a(\bm w -\bm u_h, \bm w- \bm u_h)
\\ &\lesssim a(\bm w -\bm u, \bm w- \bm u_h)+ a(\bm u -\bm u_h, \bm w-
\bm u_h)\\ &=a(\bm w -\bm u, \bm w- \bm u_h) - b(\bm w -\bm u_h, p-
p_h).
\end{align*}
Since $\bm w -\bm u_h\in \bm Z(0)$, the $q_h$ can be replaced by any
$q\in Q_h^{k'}$, we obtain
\begin{displaymath}
\|\bm w -\bm u_h\|_{\mathrm{DG}}^2\lesssim a(\bm w -\bm u, \bm w- \bm u_h) -
b(\bm w -\bm u_h, p- q).
\end{displaymath}
Using Lemma \ref{le:ellipticcontinuous} and
\ref{le:divergencecontinuous} gives
\begin{equation}\label{eq:kernel_ineq}
\|\bm u -\bm u_h\|_{\mathrm{DG}} \lesssim \|\bm u- \bm w\|_{\mathrm{DG}} + \| p-
q\|_{L^2(\Omega)}, \ \bm w\in \bm Z(\bm g), q\in Q_h^{k'}.
\end{equation}
Then we deal with an arbitrary function in $\bm V_h^k$. For the fixed
$\bm v \in \bm V_h^k$, we consider the problem of finding $\bm z(\bm
v) \in \bm V_h^k$, such that
\begin{displaymath}
 b(\bm z(\bm v) , q)= b(\bm u-\bm u_h,q), \ q\in Q_h^{k'}.
\end{displaymath}
Thanks to the inf-sup condition \eqref{eq:infsupcondition} and
\cite[Proposition 5.1.1,p.270]{boffi2013mixed}. We can find a solution
$\bm z\in \bm V_h^k$, such that
\begin{equation}\label{eq:inf_sup_ineq}
\| \bm z(\bm v) \|_{\mathrm{DG}} \lesssim \sup_{0\neq q\in Q_{h}^{k'}}
\frac{b(\bm z(\bm v), q)}{\|q\|_{L^2(\Omega)}}= \sup_{0 \neq q\in
  Q_{h}^{k'}} \frac{b(\bm u -\bm u_h, q)}{\|q\|_{L^2(\Omega)}}
\lesssim \|\bm u -\bm u_h\|_{\mathrm{DG}}.
\end{equation}
Since
\begin{displaymath}
  b(\bm z(\bm v) +\bm v,q)=b (\bm u_h, q)=\int_{\mathcal E_h^b}
\bm g\cdot \bm n q \ds, \ \forall q\in Q_h^{k'},
\end{displaymath}
we have $\bm z(\bm v) +\bm v \in \bm Z (\bm g)$. Taking $\bm w=\bm
z(\bm v) +\bm v$ in \eqref{eq:kernel_ineq} yields
\begin{equation}\label{eq:kernel_app}
\|\bm u -\bm u_h\|_{\mathrm{DG}} \lesssim \|\bm u -\bm v\|_{\mathrm{DG}} + \|\bm z(\bm
v)\|_{\mathrm{DG}} + \|p-q\|_{L^2(\Omega)}.
\end{equation}
together with \eqref{eq:inf_sup_ineq},
\begin{equation}\label{eq:velocity_app}
\begin{split}
\|\bm u -\bm u_h\|_{\mathrm{DG}} &\lesssim \inf_{\bm v \in \bm V_h^k} \|\bm u
-\bm v\|_{\mathrm{DG}} + \inf_{q\in Q_h^{k'}}\|p-q\|_{L^2(\Omega)}\\ &\lesssim
h^{k}\|\bm u\|_{H^{k+1}(\Omega)} + h^{k'+1}
\|p\|_{H^{k'+1}(\Omega)}.
\end{split}
\end{equation}

Next we consider the pressure term, let $q\in Q_h^{k'}$. Using the
inf-sup condition in \eqref{eq:infsupcondition} we have
\begin{equation}\label{eq:pressure_ineq}
\begin{split}
\|q-p_h\|_{L^2(\Omega)}& \lesssim \sup_{0\neq \bm v \in \bm V_h^k}
\frac{b(\bm v, q-p_h)}{\|\bm v\|_{\mathrm{DG}}}\\ &=\sup_{0\neq \bm v \in \bm
  V_h^k} \frac{b(\bm v, q-p)+b(\bm v, p-p_h)}{\|\bm
  v\|_{\mathrm{DG}}}\\ &=\sup_{0\neq \bm v \in \bm V_h^k} \frac{b(\bm v, q-p)-
  a(\bm u -\bm u_h ,\bm v)}{\|\bm v\|_{\mathrm{DG}}} \\ &\lesssim
\|p-q\|_{L^2(\Omega)} + \|\bm u -\bm u_h\|_{\mathrm{DG}}.
\end{split}
\end{equation}
From the triangle inequality and \eqref{eq:pressure_ineq}, we obtain
\begin{equation}\label{eq:pressure_app}
\begin{split}
\|p-p_h\|_{L^2(\Omega)} &\lesssim \|\bm u -\bm u_h\|_{\mathrm{DG}} + \inf_{q\in
  Q_h^{k'}}\|p-q\|_{L^2(\Omega)} \\ 
  &\lesssim h^{k}\|\bm
u\|_{H^{k+1}(\Omega)} + h^{k'+1} \|p\|_{H^{k'+1}(\Omega)},
\end{split}
\end{equation}
and the proof is concluded by combining \eqref{eq:velocity_app} and
\eqref{eq:pressure_app}.
\end{proof}

\section{Inf-sup test}
\label{sec:infsuptest}
In this section, we perform the inf-sup tests with some
velocity-pressure finite element space pairs to validate the inf-sup
condition numerically.  After the discretization, the matrix form of
the problem \eqref{eq:weakform} is obtained,
\begin{displaymath}
  \begin{bmatrix}
    A & B^T\\ B & 0 \\
  \end{bmatrix}
  \begin{bmatrix}
    \bm{\mathrm U}\\ \bm{\mathrm P}\\
  \end{bmatrix}
  = \begin{bmatrix} \bm{\mathrm F}\\ \bm{\mathrm G}\\
  \end{bmatrix},
\end{displaymath}
where the matrix $A$ and the matrix $B$ associate with the bilinear
form $a(\cdot, \cdot)$ and $b(\cdot, \cdot)$, respectively. The vector
$\bm U, \bm P$ is the solution vector corresponding to $\bm u_h, p_h$
and $\bm F, \bm G$ is the right hand side corresponding to $\bm f, \bm
g$.

Then the numerical inf-sup test is based on the following lemma.
\begin{lemma}
  Let $S$ and $T$ be symmetric matrices of the norms
  $\|\cdot\|_{\mathrm{DG}}$ in $\bm V_h^k$ and
  $\|\cdot\|_{L^2(\Omega)}$ in $Q_{h}^{k'}$, respectively, and let
  $\mu_{\min}$ be the smallest nonzero eigenvalue defined by the
  following generalized eigenvalue problem:
  \begin{displaymath}
    B^TS^{-1}B\bm{V}=\mu_{\min}^2T\bm{V},
  \end{displaymath}
  then the value of $\beta$ is simply $\mu_{\min}$.
  \label{le:numericalinfsup}
\end{lemma}

The proof of this lemma can be found in \cite{boffi2013mixed,
  malkus1981eigenproblems}. In numerical tests, we would consider a
sequence of successive refined meshes and monitor $\mu_{\min}$ of each
mesh. If a sharp decrease of $\mu_{\min}$ is observed while the mesh
size approaches to zero, we could predict that the pair of
approximation spaces violates the inf-sup condition. Otherwise, if
$\mu_{\min}$ stabilizes as the mesh is refined, we can conclude that
the inf-sup test is passed.

The numerical tests are conducted with following settings, let
$\Omega$ be the unit square domain in two dimension and we consider
two groups of quasi-uniform meshes which are generated by the software
\emph{gmsh} \cite{geuzaine2009gmsh}. The first ones are triangular
meshes(see Fig \ref{fig:infsuptesttriangularmesh}) and the second ones
consist of triangular and quadrilateral elements(see Fig
\ref{fig:infsuptestmixedmesh}). In both cases, the mesh size $h$ is
taken by $h=\frac1n,\ n=10,20,30,\cdots 80$.
\begin{figure}[!htp]
  \centering
  \includegraphics[width=0.4\textwidth]{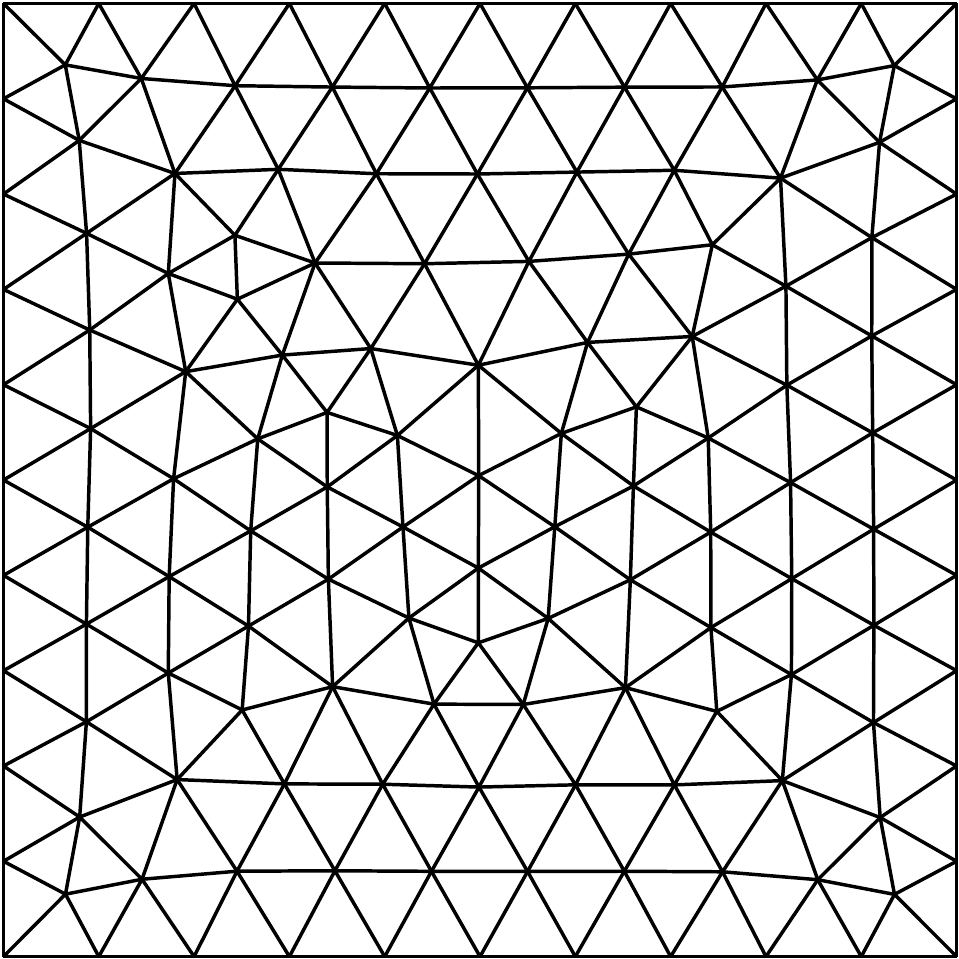}
  \hspace{25pt}
  \includegraphics[width=0.4\textwidth]{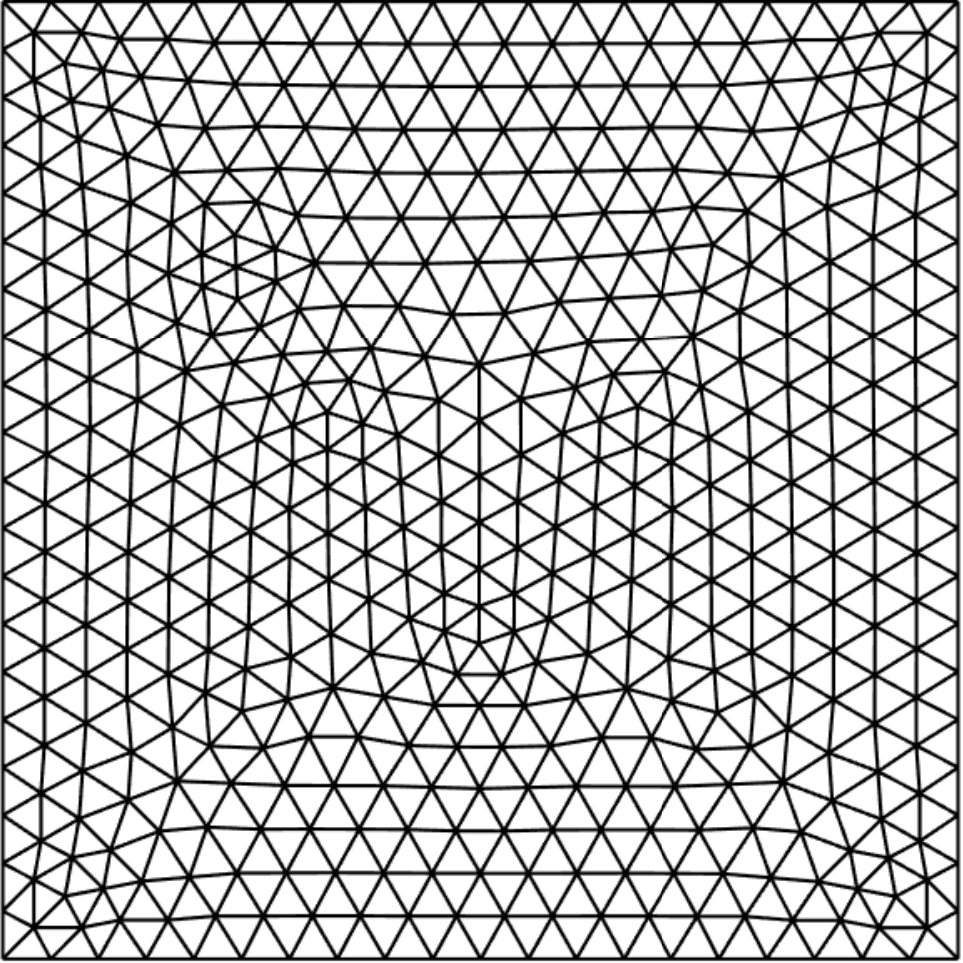}
  \caption{The triangular meshes, $h=\frac1{10}$(left)/$h=\frac1{
      20}$(right).}
  \label{fig:infsuptesttriangularmesh}
\end{figure}
\begin{figure}[!htp]
  \centering
  \includegraphics[width=0.4\textwidth]{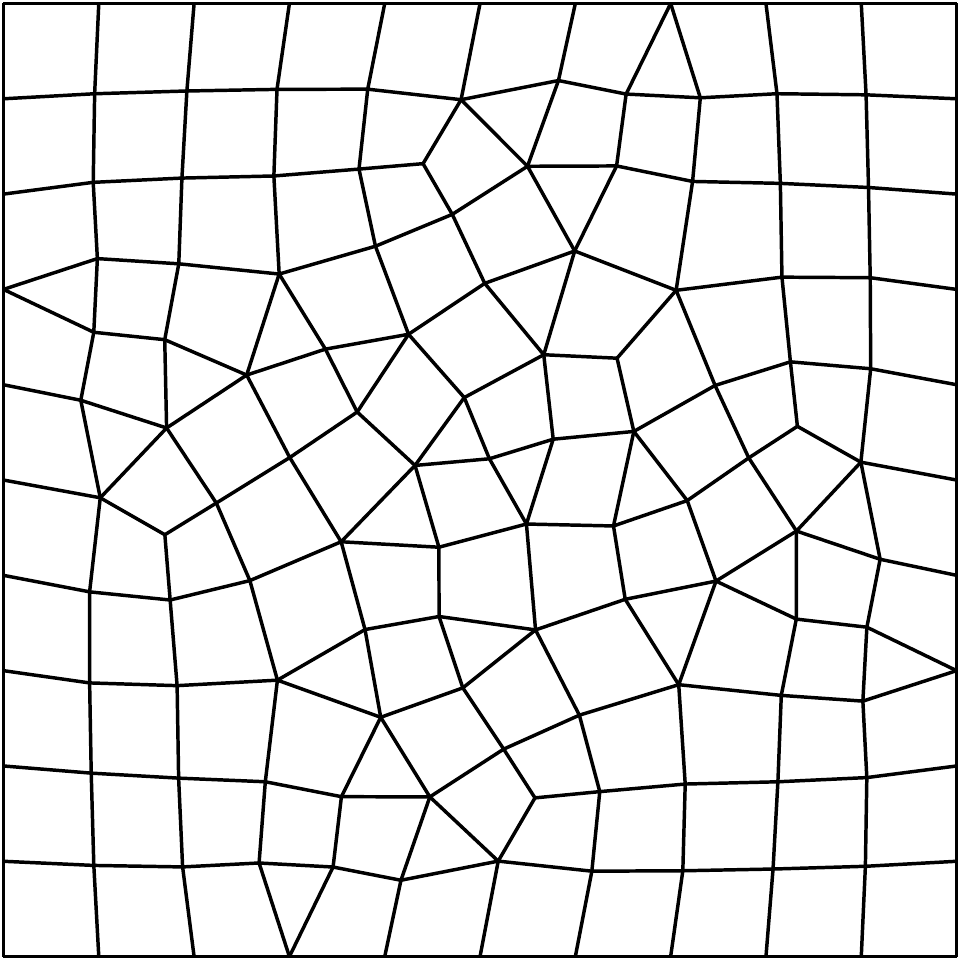}
  \hspace{25pt}
  \includegraphics[width=0.4\textwidth]{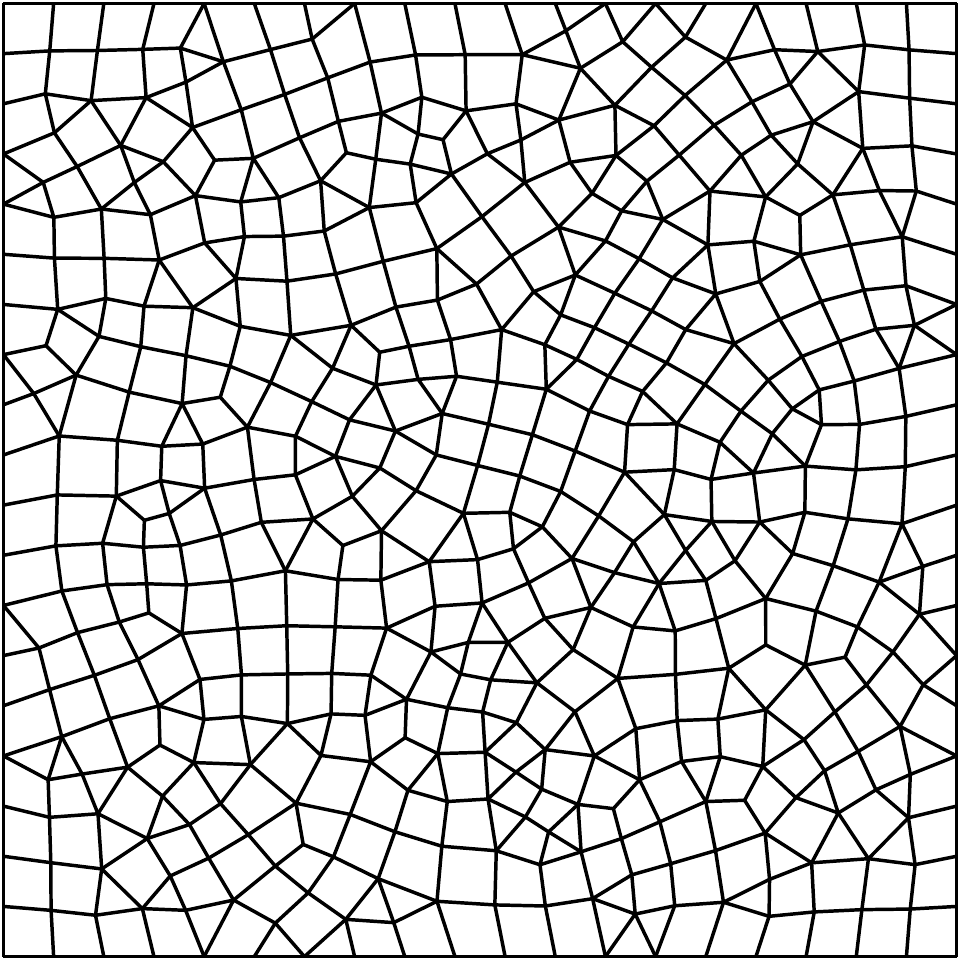}
  \caption{The mixed meshes, $h=\frac1{10}$(left)/$h=\frac1{2
      0}$(right).}
  \label{fig:infsuptestmixedmesh}
\end{figure}

With the given mesh partition, the finite element space can be
constructed. As we mention before, for element $K$, $\# S(K)$ should
be large enough to ensure the uniform upper bound $\Lambda_m$. For
simplicity, $\# S(K)$ is taken uniformly and for different order $k$
we list a group of reference values of $\# S(K)$ for both meshes in
Table \ref{tab:patchnumber2d}.

\begin{table}[htp]
  \centering
  \caption{choices of $\# S(K)$ for $1\leq k \leq 5$}
  \vspace{-8pt}
  \label{tab:patchnumber2d}
  \scalebox{1.10}{
  \begin{tabular}{|l|l|p{0.6cm}|p{0.6cm}|p{0.6cm}|p{0.6cm}|p{0.6cm}|}
    \hline \multicolumn{2}{|l|}{order $k$}& 1 & 2 & 3 & 4 & 5\\ \hline
    \multirow{2}{*}{$\# S(K)$} & triangular mesh & 5 & 9 & 18 & 25 &
    32\\ \cline{2-7} & mixed mesh& 6 & 10 & 20 & 28 & 35\\ \hline
  \end{tabular}
  }
\end{table}

\newcommand{\br}[1]{\uppercase\expandafter{\romannumeral#1}}

We consider three choices of velocity-pressure pairs:
\begin{itemize}
  \item \textbf{Method \br 1.} $ (\bm u_h, p_h)\in\bm V_{h}^{k}\times
    Q_{h}^{k-1},\ 1\leq k \leq 5.$
  \item \textbf{Method \br 2.} $ (\bm u_h, p_h)\in\bm V_{h}^{k}\times
    Q_{h}^{k},\ 1\leq k \leq 5.$
  \item \textbf{Method \br 3.} $ (\bm u_h, p_h)\in\bm V_{h}^{k}\times
    Q_{h}^{0},\ 1\leq k \leq 5.$
\end{itemize}
Here the space $Q_h^0$ is just the piecewise constant space. These
methods correspond to the choices $k'=k,\ k-1,\ 0$, respectively.

\textbf{Method \br 1.} The combination of polynomial degrees for the
velocity and pressure approximation spaces is common in traditional
FEM and DG while $k\geq 2$, known as Taylor-Hood elements. Numerical
results for the method \br 1 are shown in Fig \ref{fig:m1infsuptest}.
$\mu_{\min}$ appears to be bounded in every case, which clearly
indicates the method \br 1 has passed the inf-sup test. It is
noticeable that $\bm V_{h}^{1}\times Q_{h}^{0}$ is a stable pair which
will lead to the locking-phenomenon in traditional FEM.
\begin{figure}[htp]
  \centering
  \includegraphics[width=0.48\textwidth]{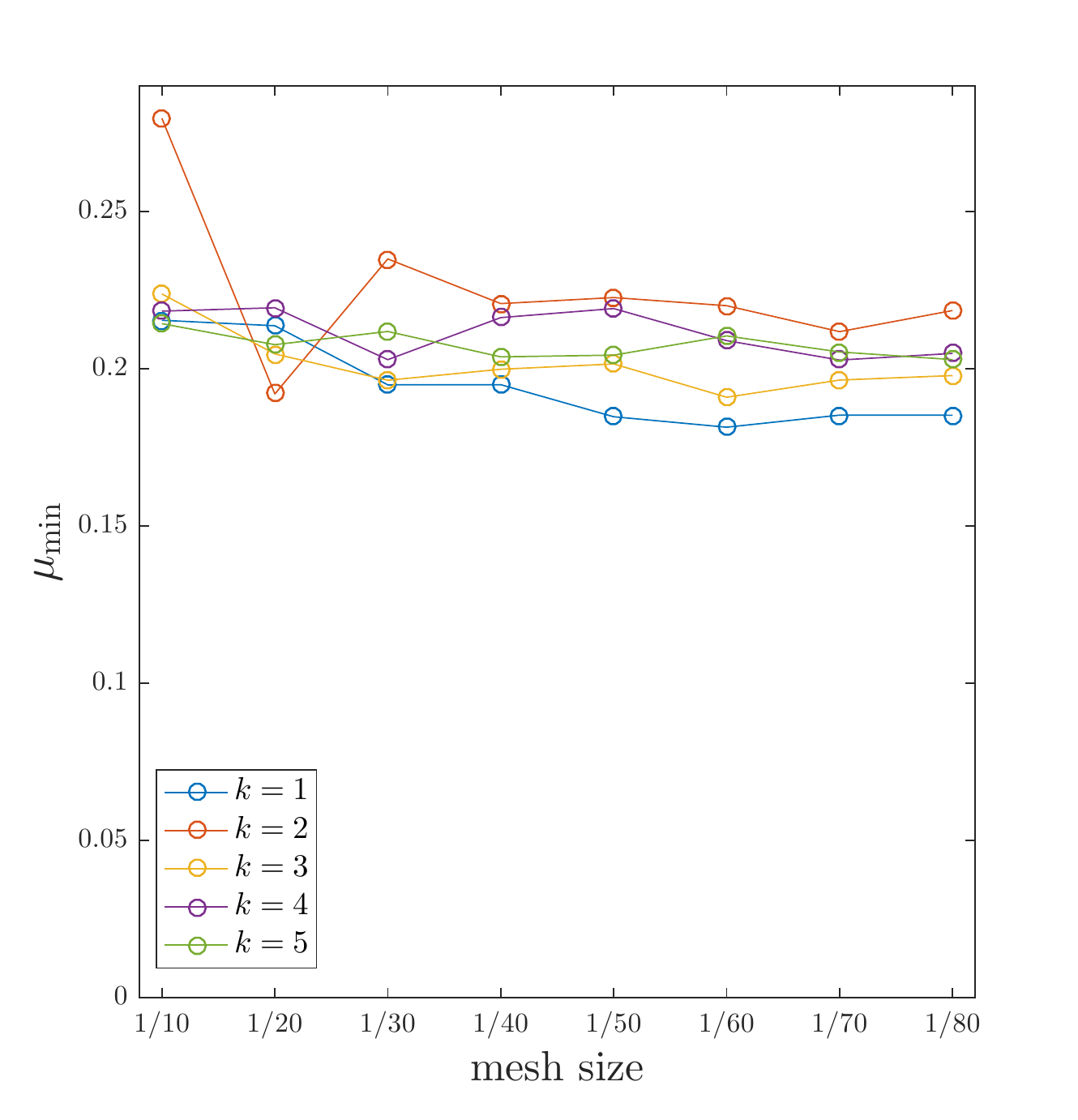}
  \includegraphics[width=0.48\textwidth]{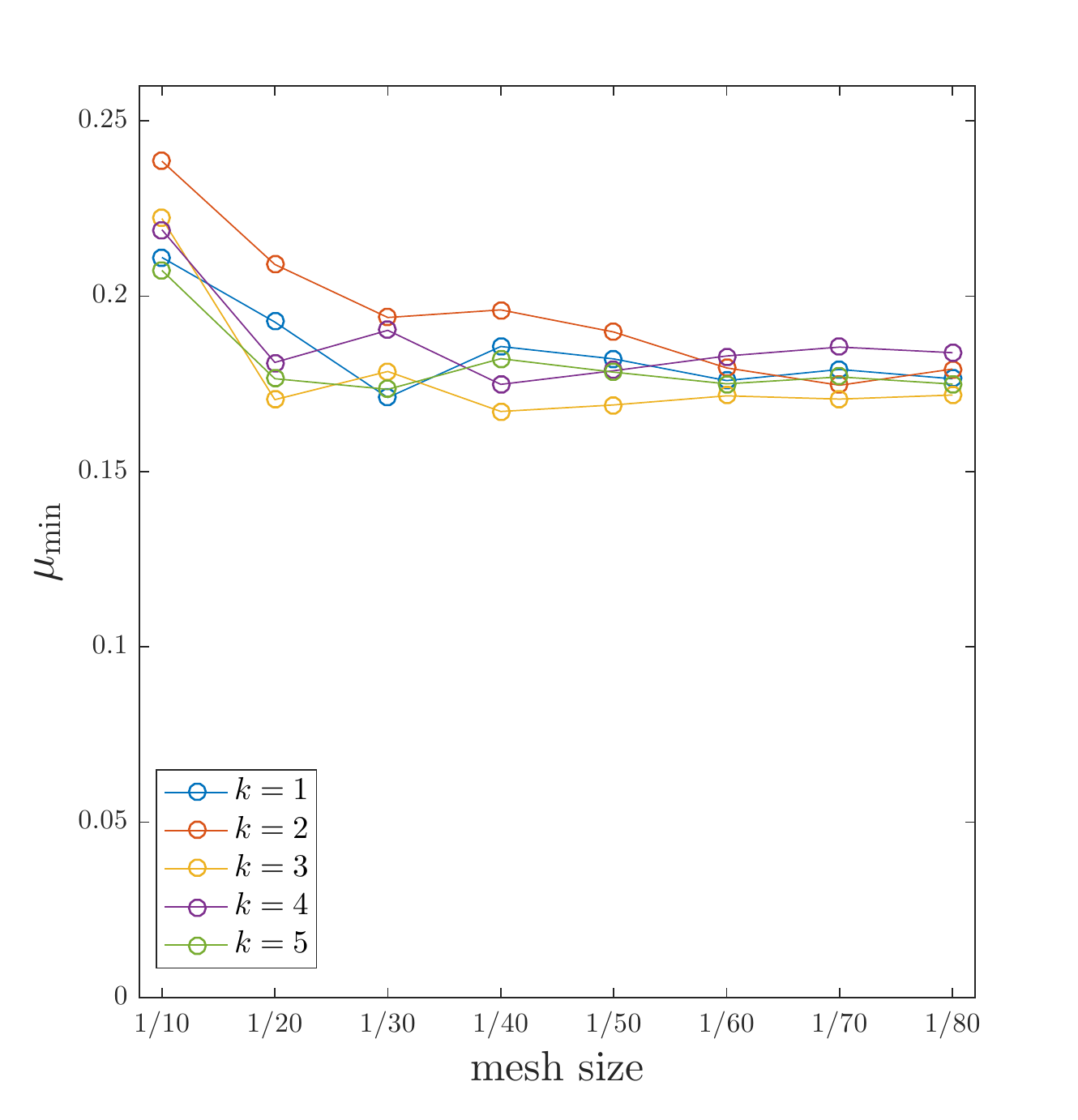}
  \caption{Inf-sup tests for method \br 1 on triangular meshes (left)
    / mixed meshes (right)}
  \label{fig:m1infsuptest}
\end{figure}

\textbf{Method \br 2.} We consider equal polynomial degrees for both
approximation spaces. This method is more efficient because the
reconstruction procedure is carried out only once.  Fig
\ref{fig:m2infsuptest} displays the history of $\mu_{\min}$. Similar
with method \br 1, the values of $\mu_{\min}$ stabilize as $h$
decreases to zero. This method surprisingly keeps valid with $\bm
V_{h}^{1}\times Q_{h}^{1}$ which is unstable due to the spurious
pressure models in traditional FEM.

\begin{figure}[htp]
  \centering
  \includegraphics[width=0.48\textwidth]{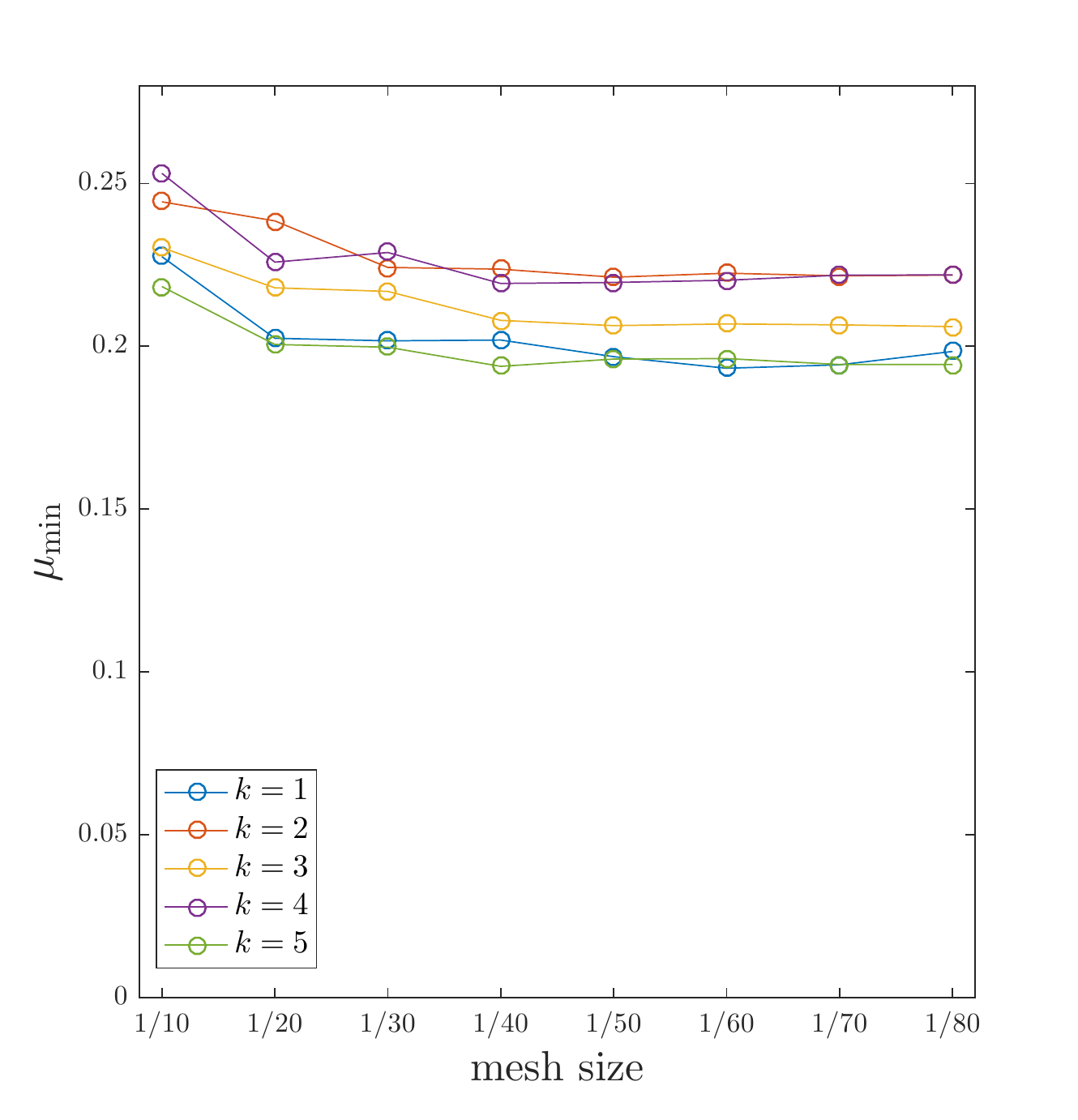}
  \includegraphics[width=0.48\textwidth]{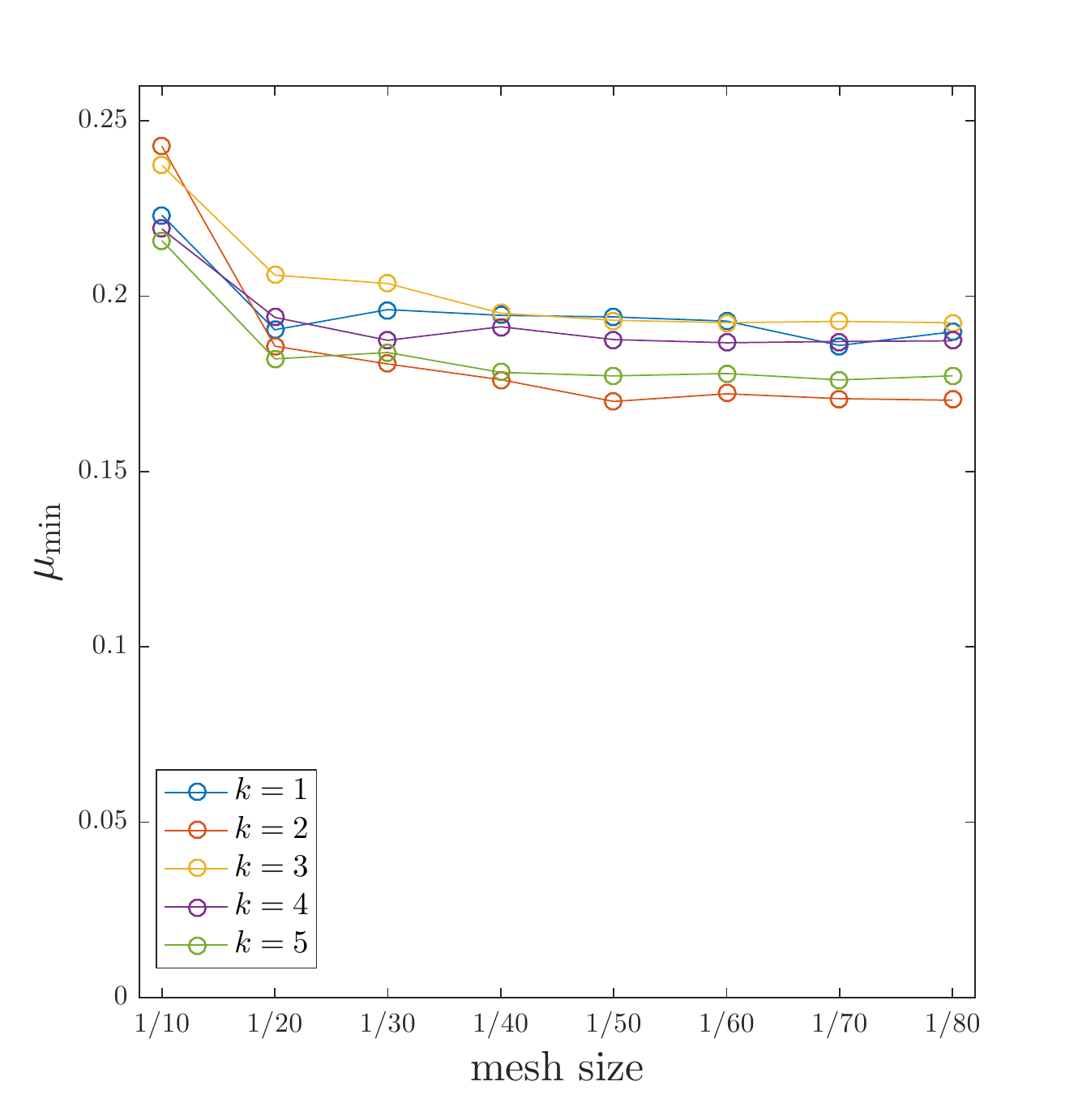}
  \caption{Inf-sup tests for method \br 2 on triangular meshes (left)
    / mixed meshes (right)}
  \label{fig:m2infsuptest}
\end{figure}

\textbf{Method \br 3.} We note that the number of DOFs of our finite
element space, which are always equal to the number of elements in
partition, has no concern to the order of approximation accuracy. In
the sense of that, for all $k$, high order space $V_{h}^{k}$ is in the
same size as the piecewise constant piece $Q_{k}^{0}$. Thus, we take
$\bm V_{h}^{k}$ as the velocity approximation space while we select
$Q_{h}^{0}$ for the pressure. Fig \ref{fig:m3infsuptest} summarizes
the results of this inf-sup test, which show that the inf-sup
condition holds.

\begin{figure}[htp]
  \centering
  \includegraphics[width=0.48\textwidth]{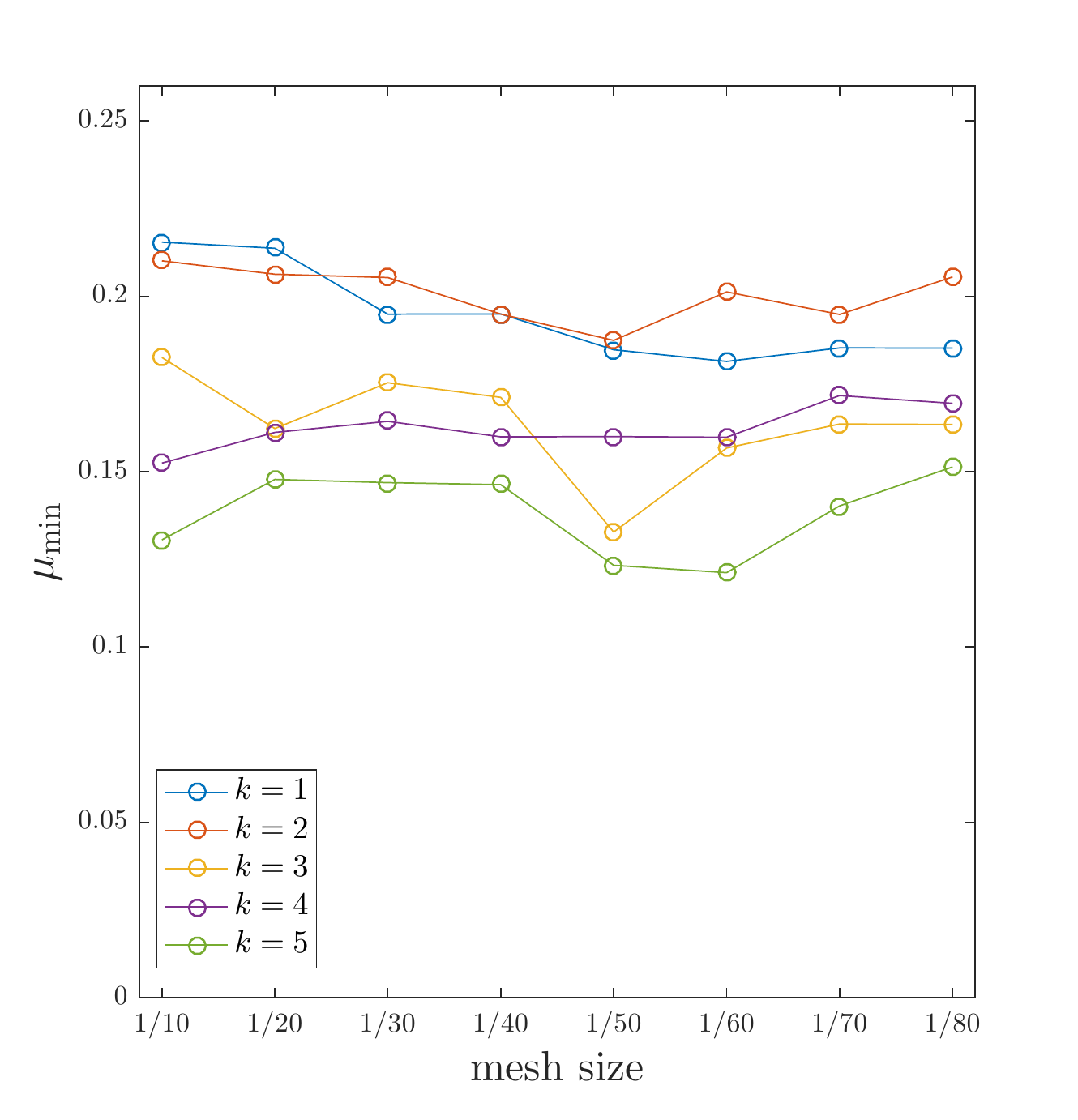}
  \includegraphics[width=0.48\textwidth]{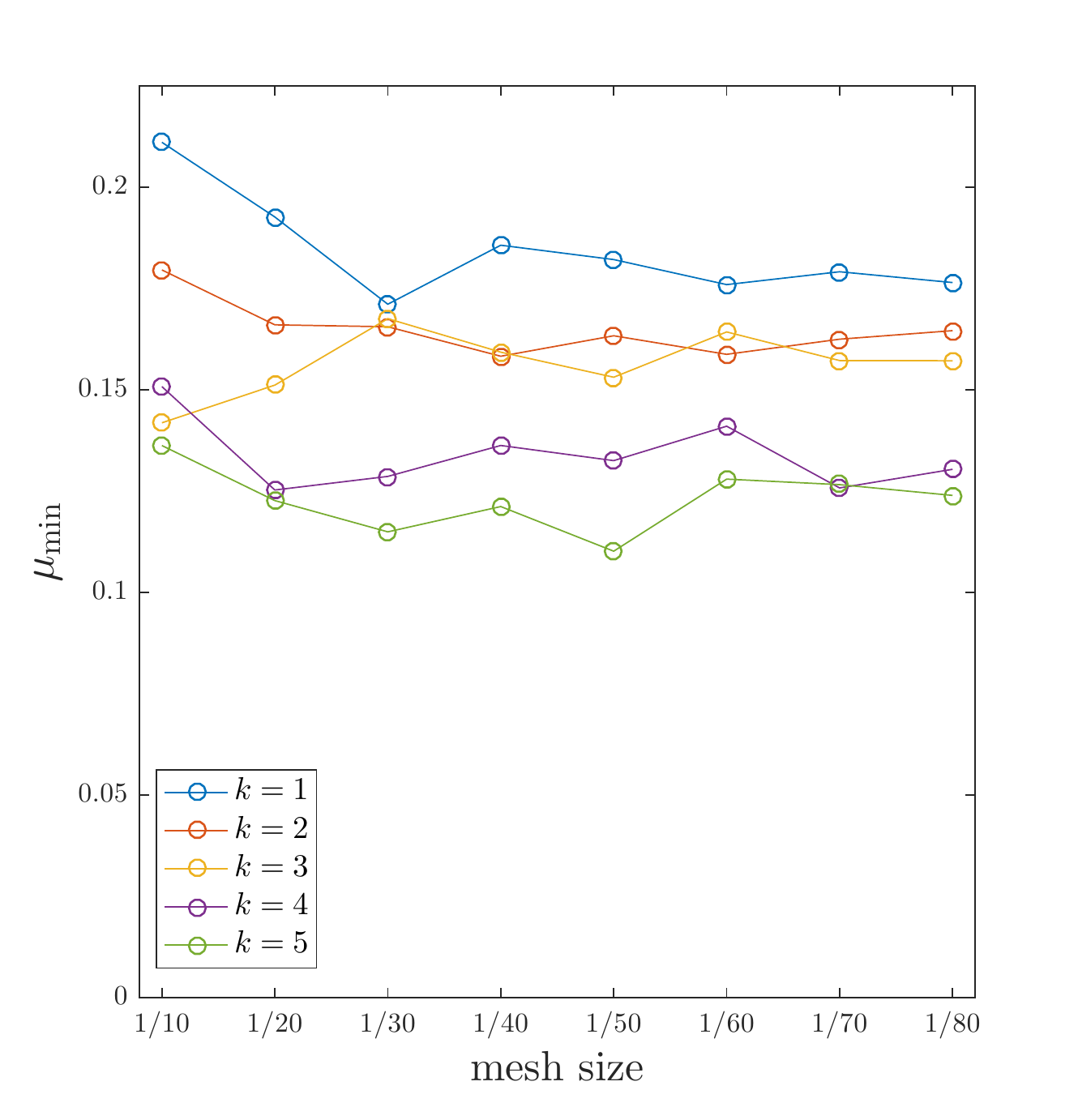}
  \caption{Inf-sup tests for method \br 3 on triangular meshes (left)
    / mixed meshes (right)}
  \label{fig:m3infsuptest}
\end{figure}

The satisfaction of the inf-sup condition has been checked in this
section by the numerical tests. All experiments show that the inf-sup
value $\mu_{\min}$ is bounded. In fact, the combination of two
approximation spaces can be more flexible, such as $\bm
V_{h}^{k}\times Q_{h}^{k+1}$ or $\bm V_{h}^{k}\times Q_{h}^{k+2}$, see
Fig \ref{fig:m4infsuptest} and Fig \ref{fig:m5infsuptest}. Both cases
could pass the inf-sup test.  The numerical results demonstrate that
our finite element space possesses more robust properties than the
traditional finite element method. An analytical proof of the
verification of the inf-sup condition is considered as the future
work.

\begin{figure}[htp]
  \centering
  \includegraphics[width=0.48\textwidth]{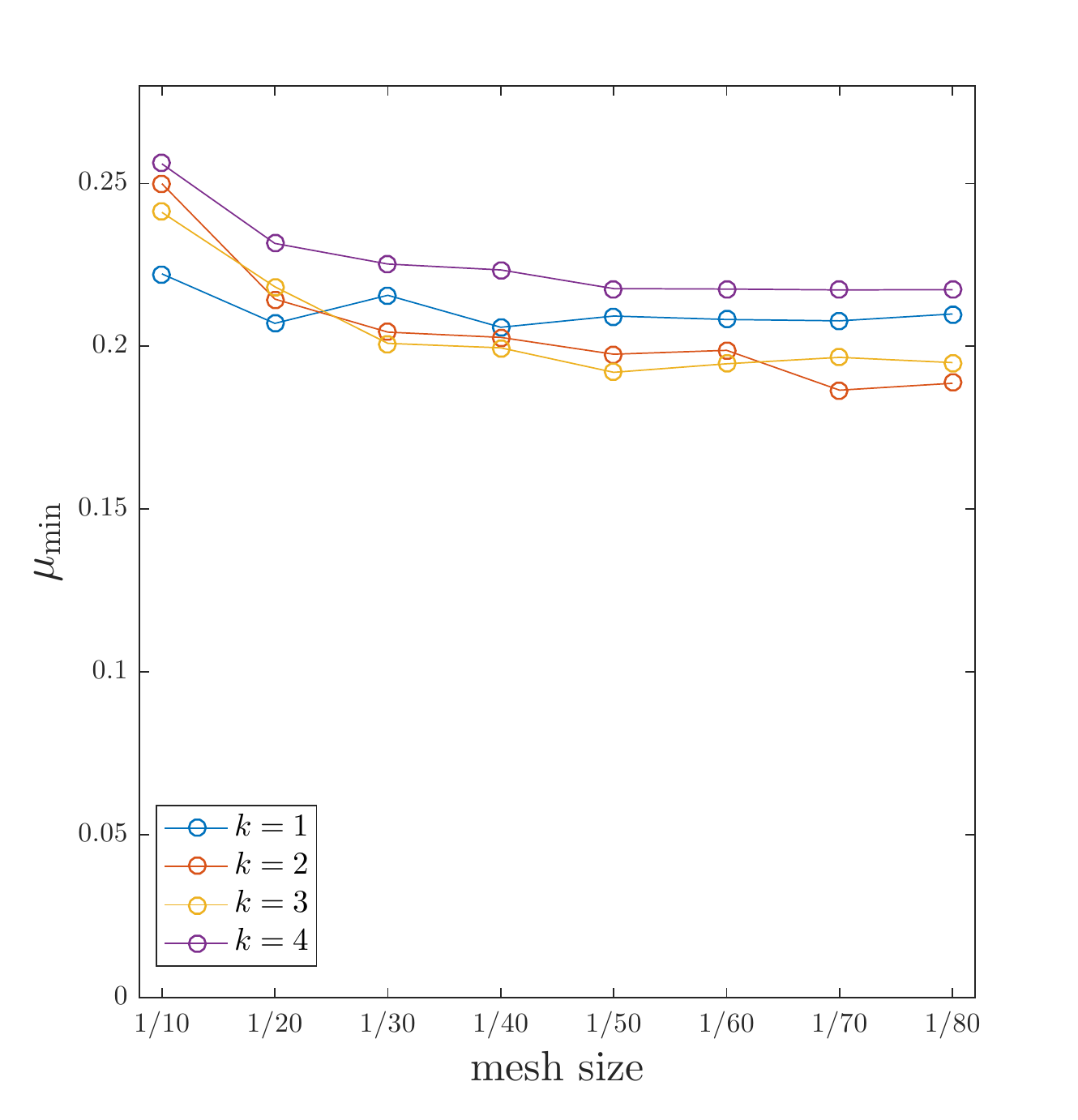}
  \includegraphics[width=0.48\textwidth]{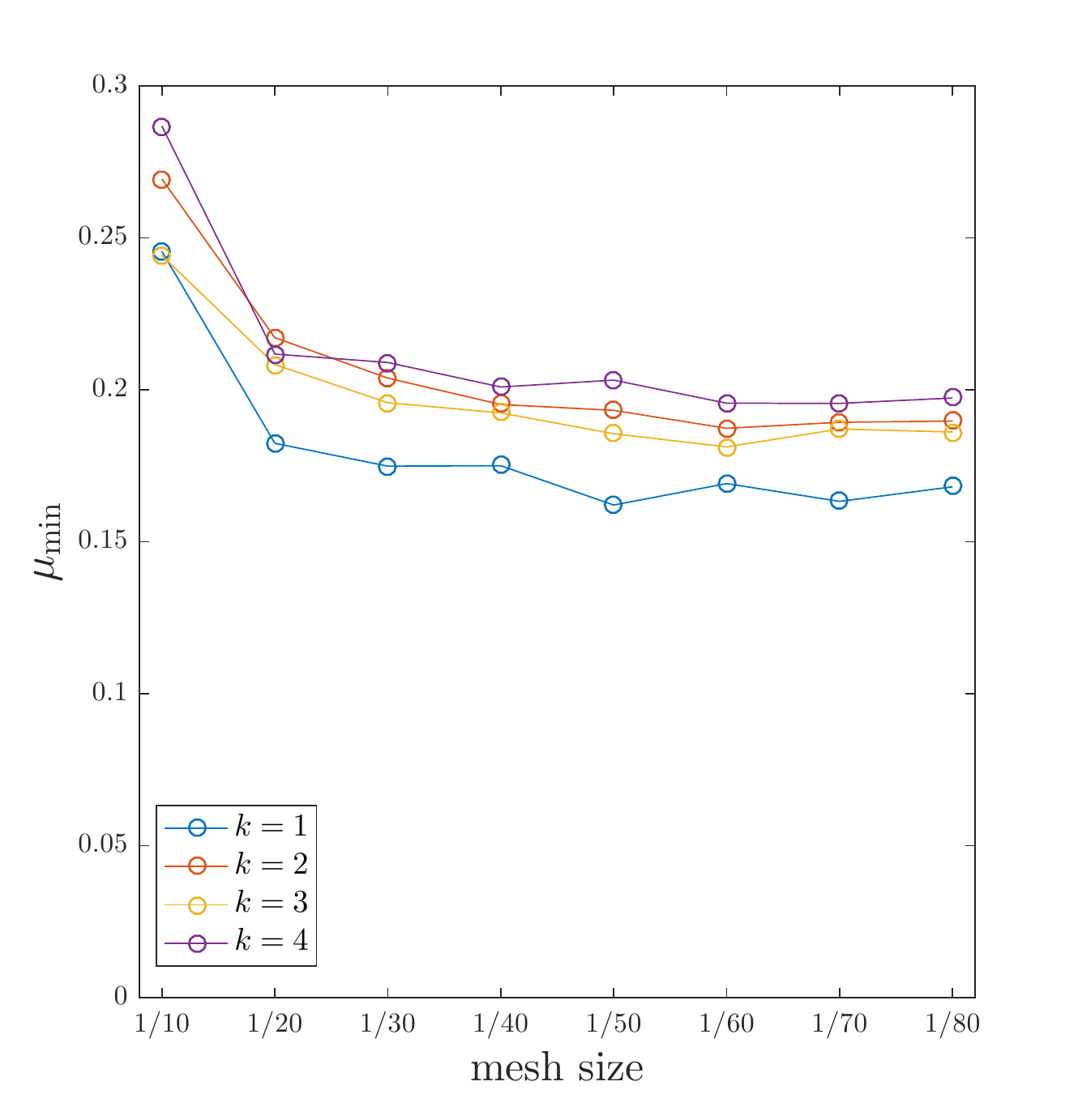}
  \caption{Inf-sup tests for $\bm V_{h}^{k}\times Q_{h}^{k+1}$ on
    triangular meshes (left) / mixed meshes (right)}
  \label{fig:m4infsuptest}
\end{figure}
\begin{figure}[htp]
  \centering
  \includegraphics[width=0.48\textwidth]{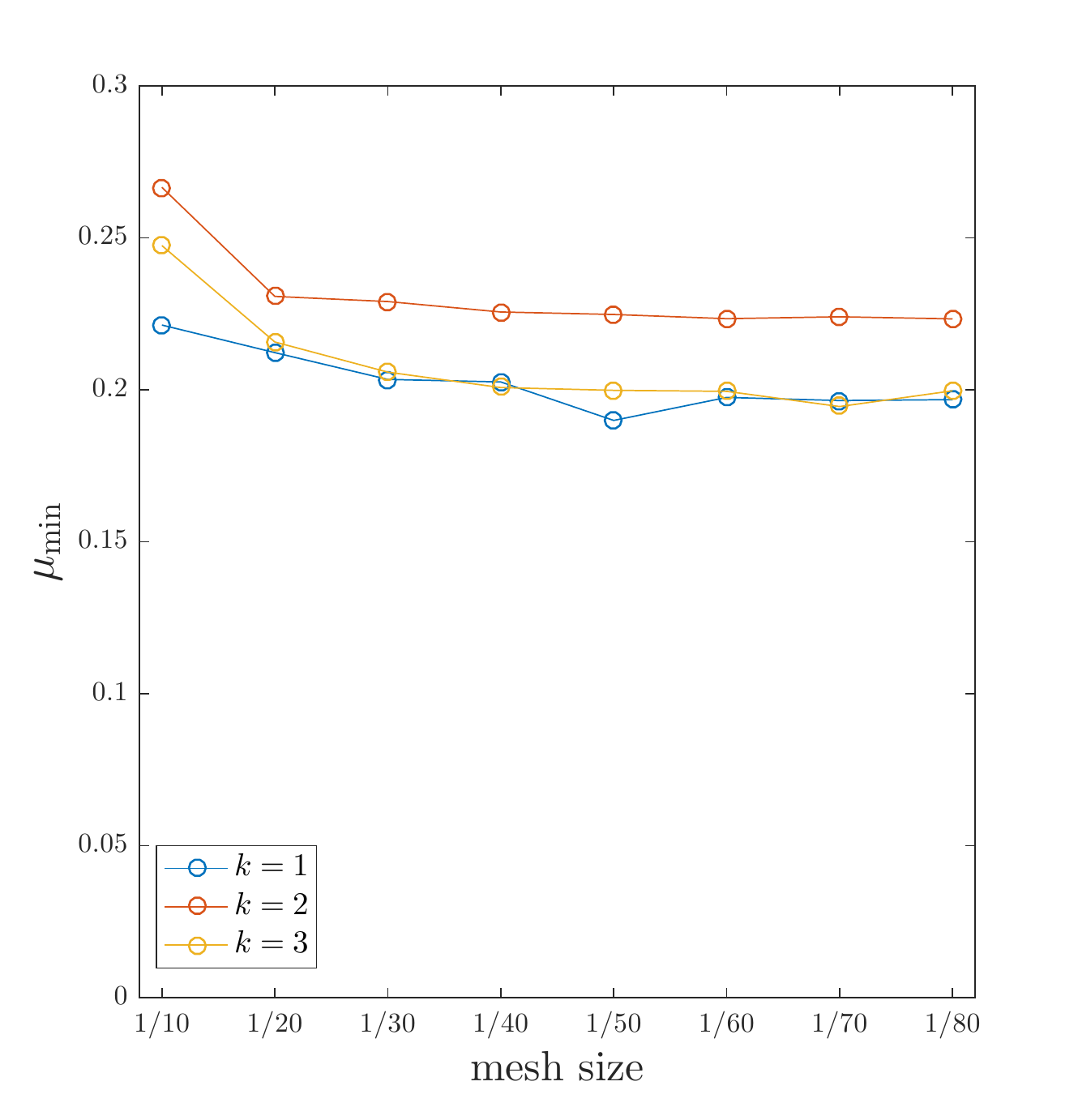}
  \includegraphics[width=0.48\textwidth]{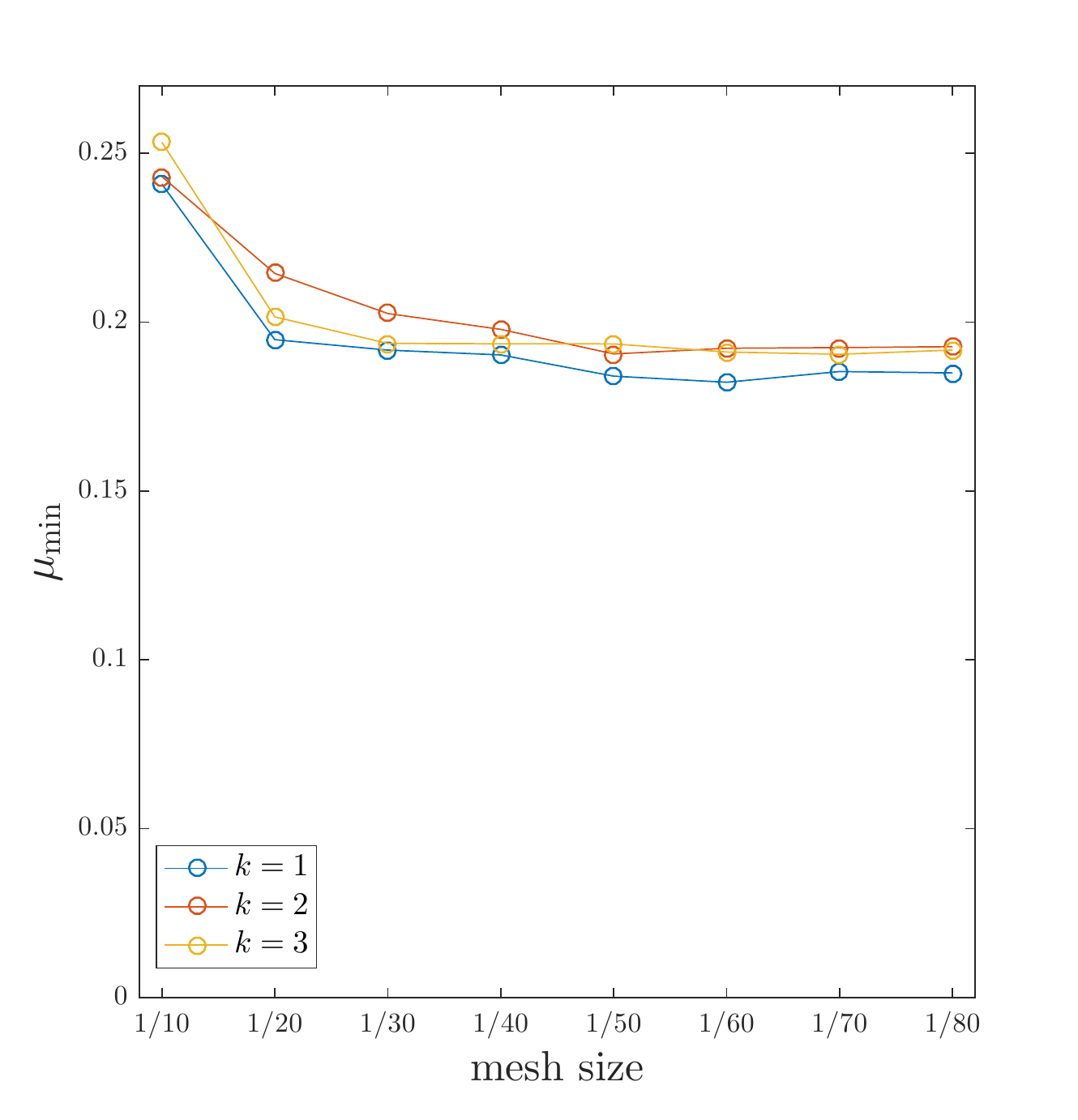}
  \caption{Inf-sup tests for $\bm V_{h}^{k}\times Q_{h}^{k+2}$ on
    triangular meshes (left) / mixed meshes (right)}
  \label{fig:m5infsuptest}
\end{figure}

\section{Numerical Results}
\label{sec:numericalresults}
In this section, we give { some implementation details and}
some numerical examples in two dimensions to verify the theoretical
error estimates in Theorem \ref{th:prioriestimate}. The numerical
settings remain unchanged as in the previous section. For the
resulting sparse system, a direct sparse solver is employed to solve
it.

{
\subsection{Implementation}
\label{sec:2dexample}
We present a 2D example on the domain $[0, 1]\times[0, 1]$ to
illustrate the implementation of our method. The key point is to
calculate the basis functions. We consider a quasi-uniform triangular
mesh, see Fig \ref{fig:2dquasi_uniform}.
\begin{figure}[htb]
  \centering
  \includegraphics[width=0.48\textwidth]{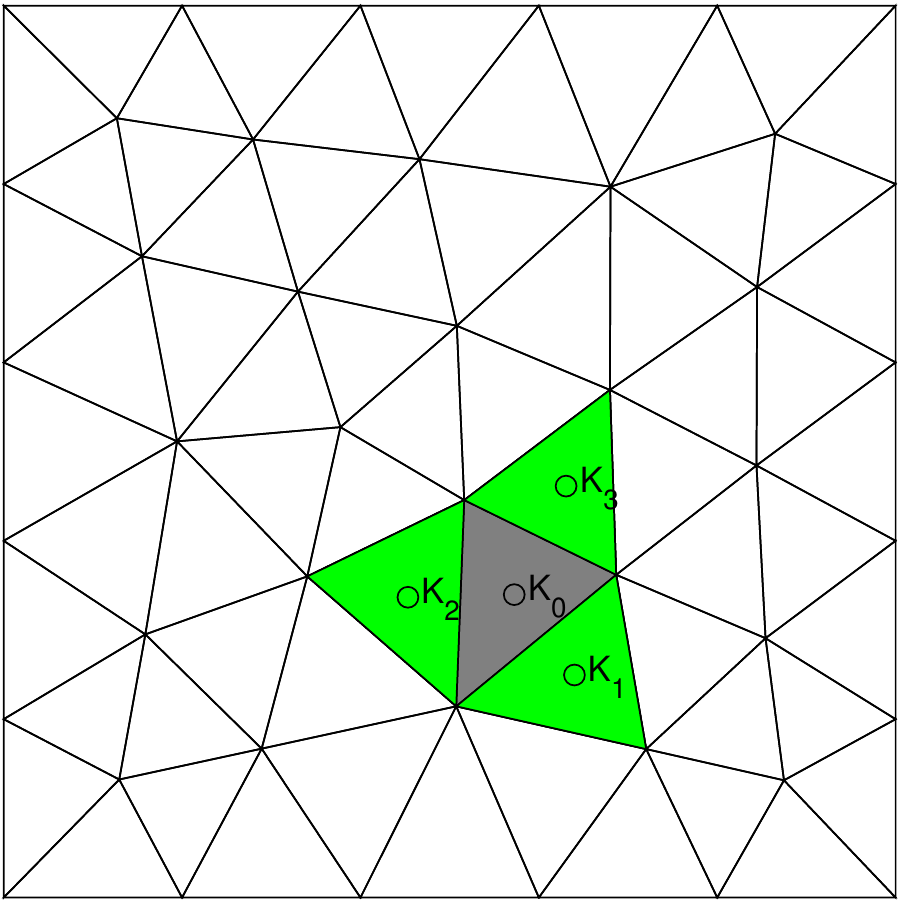}
  \includegraphics[width=0.48\textwidth, height=0.481\textwidth]{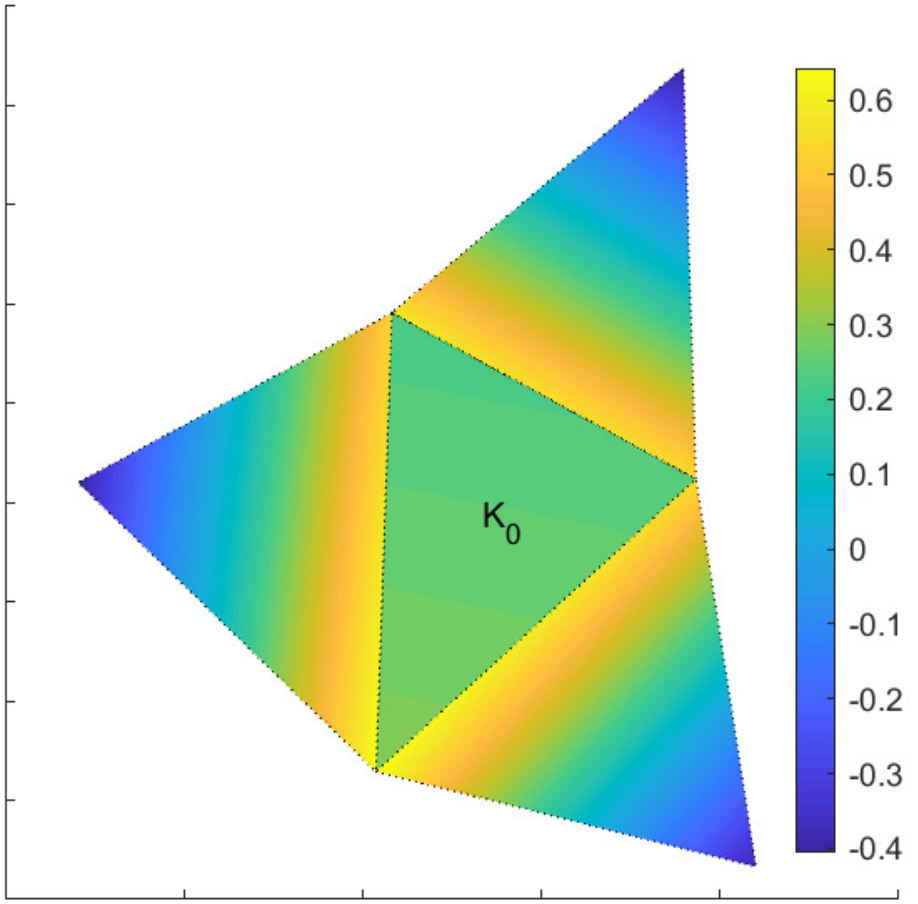}
  \caption{The triangulation and the element patch $S(K_0)$ and the
  collocation points set $\mc I_{K_0}$ (left) / the basis function
  $\lambda_{K_0}$ (right)}
  \label{fig:2dquasi_uniform}
\end{figure}
Here we consider a linear reconstruction. The barycenters of all
elements are assigned as the collocation points. For any element $K$,
we let $S(K)$ consist of $K$ itself and all edge-neighboring elements. 
Then we obtain the basis functions by solving the least squares
problem on every element.

We take $K_0$ as an example (see Fig \ref{fig:2dquasi_uniform}), the
element patch $S(K_0)$ is chosen as
\begin{displaymath}
  S(K_0) = \left\{ K_{0}, K_{1}, K_{2},K_{3} \right\},
\end{displaymath}
and the corresponding collocation points are
\begin{displaymath}
  \mc I_{K_0} = \left\{ (x_{K_{0}}, y_{K_{0}}),(x_{K_{1}},
  y_{K_{1}}),(x_{K_{2}}, y_{K_{2}}),(x_{K_{3}}, y_{K_{3}}) \right\},
\end{displaymath}
where $(x_{K_i}, y_{K_i})$ is the barycenter of $K_i$.

For a continuous function $g$, the least squares problem is
\begin{displaymath}
  \mc R_{K_0} = \mathop{\arg \min}_{ (a, b, c) \in \mathbb R}
  \sum_{(x_{K'},y_{K'}) \in \mc{I}_{K_0}} |g(x_{K'},y_{K'}) -
   (a + bx_{K'} + cy_{K'})|^2.
\end{displaymath}
By the Assumption 1, we obtain the unique solution
\begin{displaymath}
  [a, b, c]^T = (A^TA)^{-1}A^Tq,
\end{displaymath}
where
\begin{displaymath}
  A = \begin{bmatrix} 1 & x_{K_{0}}& y_{K_{0}} \\ 1 & x_{K_{1}}&
    y_{K_{1}} \\ 1 & x_{K_{2}} & y_{K_{2}}\\ 1 & x_{K_{3}} & y_{K_{3}}
  \end{bmatrix}, \quad 
  q = \begin{bmatrix} g(x_{K_{0}},y_{K_{0}}) \\g(x_{K_{1}},y_{K_{1}})
    \\g(x_{K_{2}},y_{K_{2}}) \\g(x_{K_{3}},y_{K_{3}})
  \end{bmatrix}.
\end{displaymath}
Thus the matrix $(A^TA)^{-1}A^T$ contains all necessary information of
the basis functions $\lambda_{K_0}, \lambda_{K_1}, \lambda_{K_2},
\lambda_{K_3}$ on $K_0$ and we just store it to represent the basis
functions. All the basis functions could be obtained by solving the
least squares problem on every element. Besides, the basis function
$\lambda_{K_0}$ is presented in Fig \ref{fig:2dquasi_uniform} and we
shall point out that the support of the basis function is not always
equal to the element patch, and vice versa. 
}

\subsection{2D smooth problem}
We first consider a 2D example on $\Omega=[0,1]^2$ with smooth
analytical solution to investigate the convergence properties. The
exact solution is taken as
\begin{displaymath}
  \bm u(x,y)=\begin{bmatrix} \sin(2\pi x)\cos(2\pi y) \\ -\cos(2\pi
  x)\sin(2\pi y) \\
  \end{bmatrix}, \quad
  p(x,y)=x^2+y^2,
\end{displaymath}
and the source term $\bm f$ and the boundary condition $\bm g$ are
chosen accordingly. We consider three methods in Section
\ref{sec:infsuptest} and solve the Stokes problem on the given
triangular meshes and mixed meshes, respectively, with mesh size
$h=\frac1n, n=10, 20, 40, 80$.

In Fig \ref{fig:m1L2error} and Fig \ref{fig:m1DGerror}, we present the
$L^2$ norm and the DG energy norm of the error in the approximation to
the exact velocity on both meshes when using method \br 1. And Fig
\ref{fig:m1Pressureerror} shows the pressure error in $L^2$ norm. Here
we observe that the optimal convergence rates for $\|\bm u - \bm
u_h\|_{L^2( \Omega)}$, $\|\bm u - \bm u_h\|_{\mathrm{DG}}$ and $\|p -
p_h\|_{L^2(\Omega)}$ are obtained, which are $O(h^{k+1})$, $O(h^{k})$
and $O(h^k)$, respectively. The numerical results confirm the estimate
\eqref{eq:prioriestimate}.

\begin{figure}
  \centering
  \includegraphics[width=0.48\textwidth]{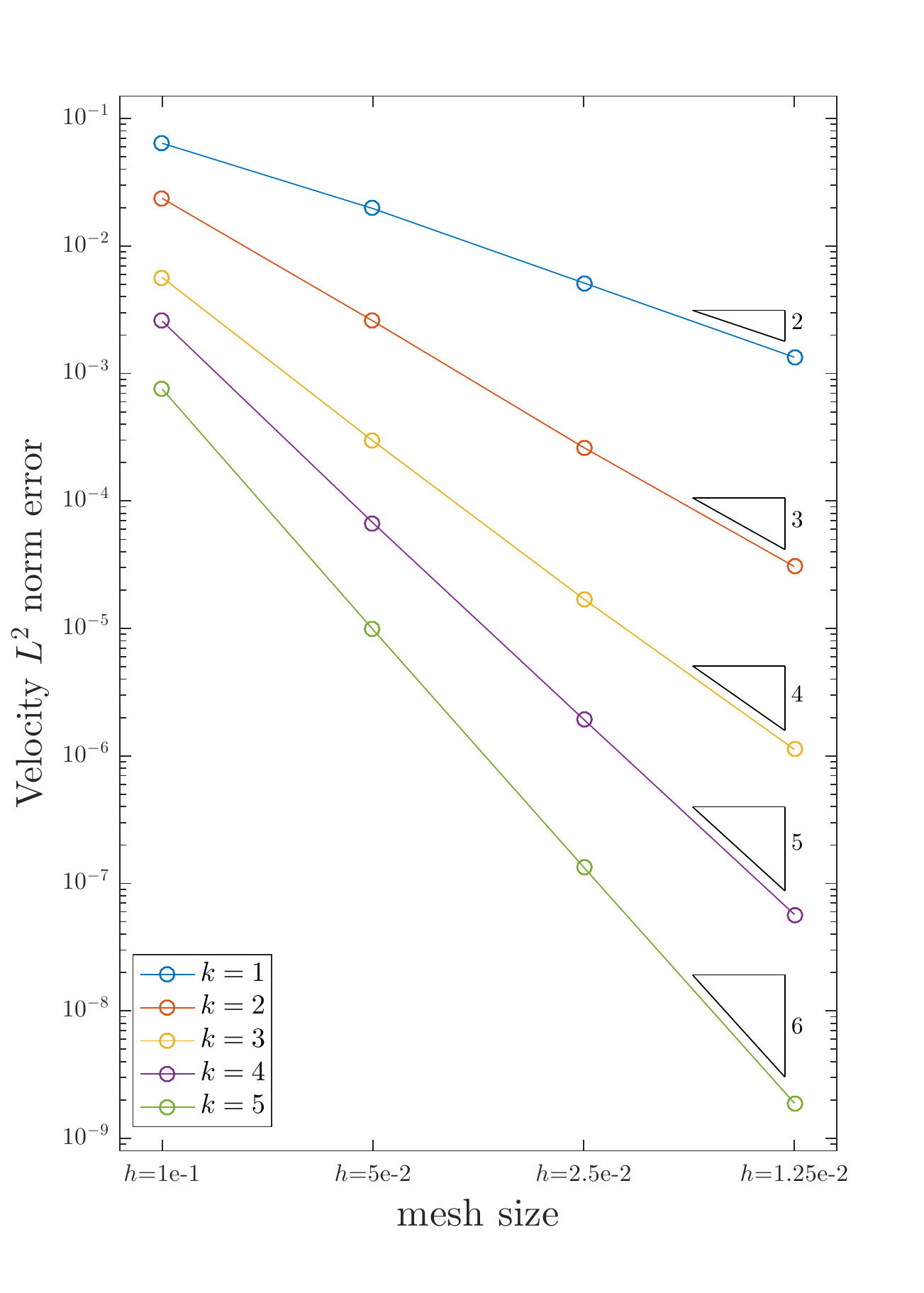}
  \includegraphics[width=0.48\textwidth]{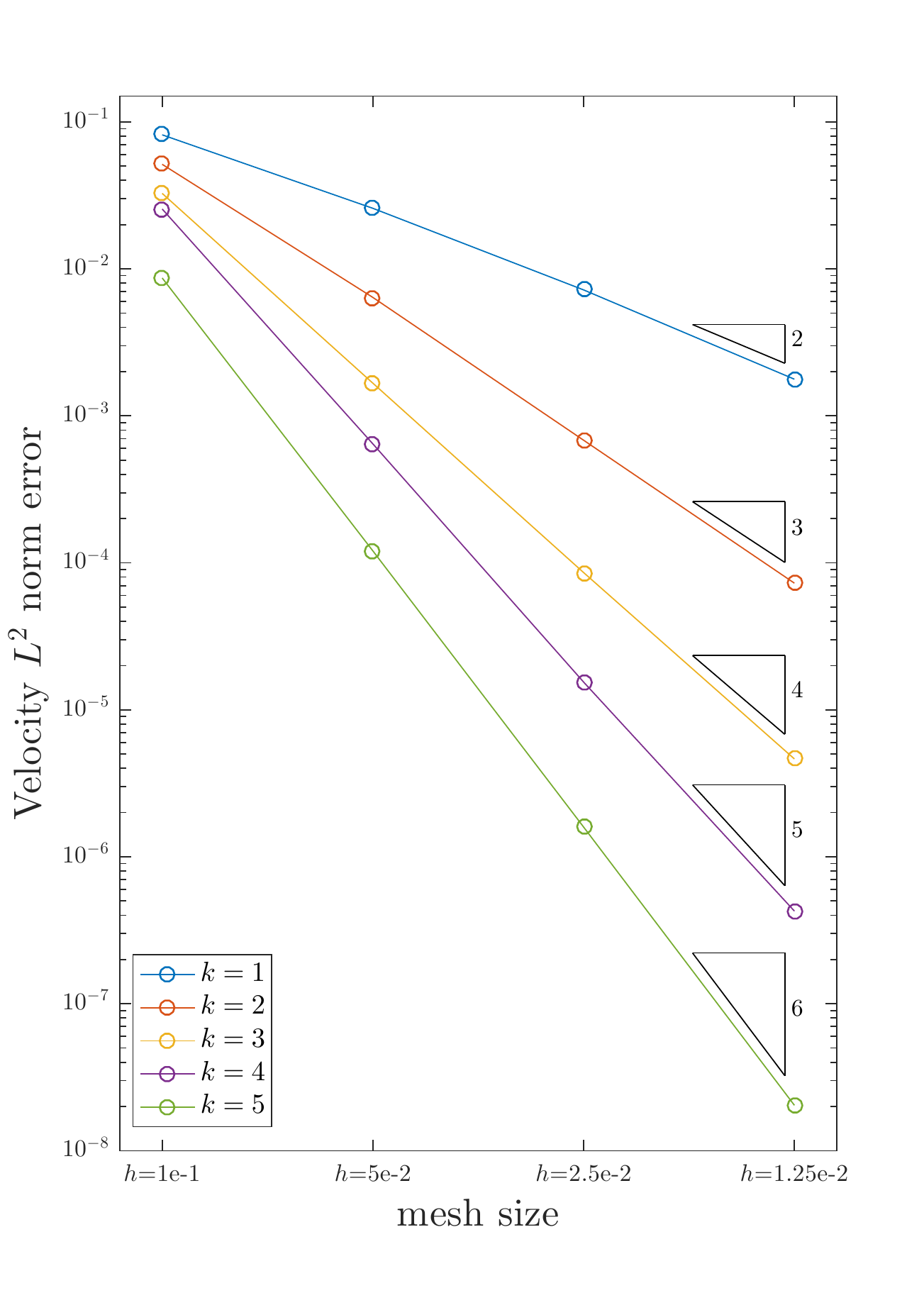}
  \caption{Velocity $L^2$ norm error with method \br 1 for the smooth
    case on triangular meshes (left) / mixed meshes (right)}
  \label{fig:m1L2error}
\end{figure}
\begin{figure}
  \centering
  \includegraphics[width=0.48\textwidth]{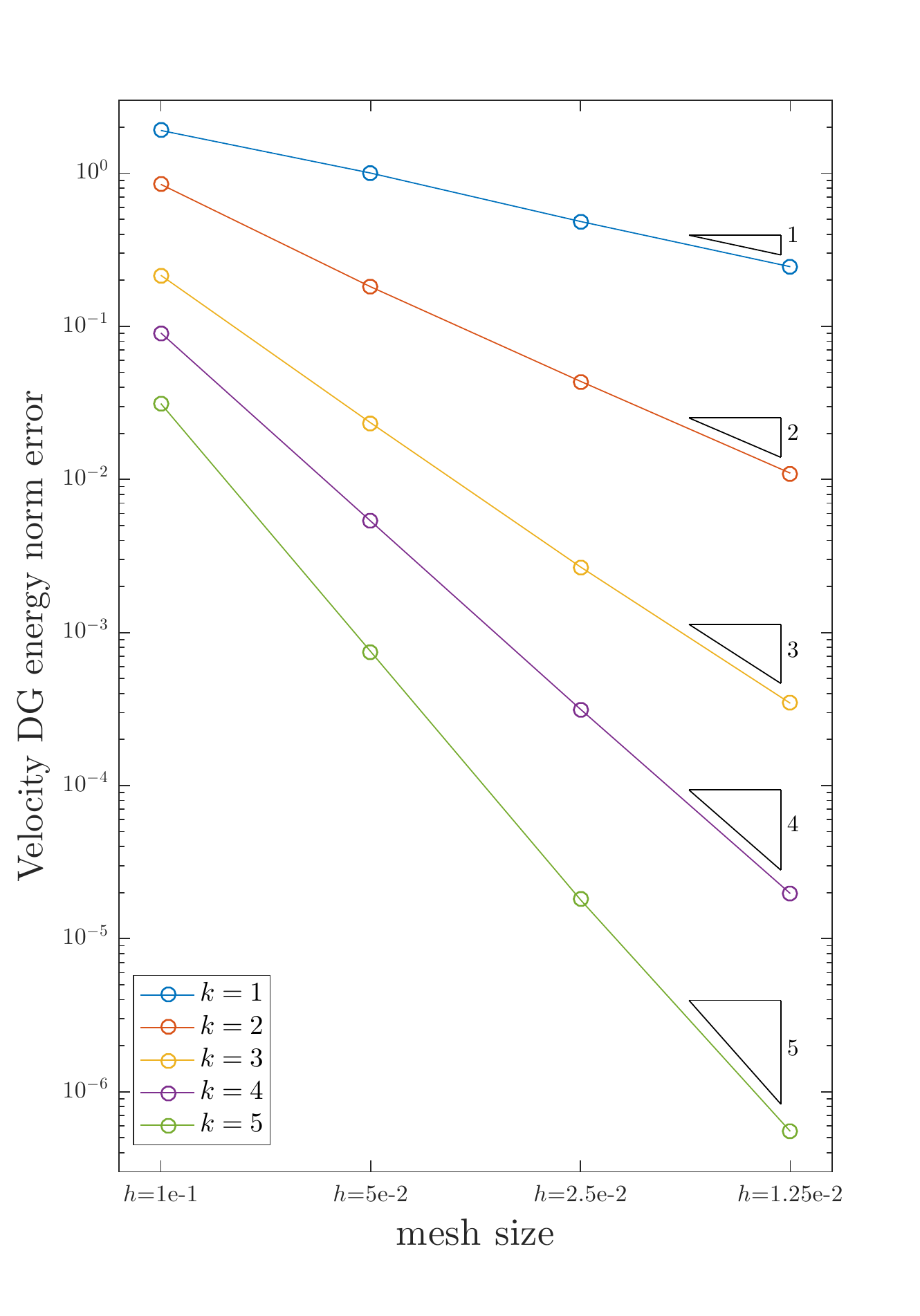}
  \includegraphics[width=0.48\textwidth]{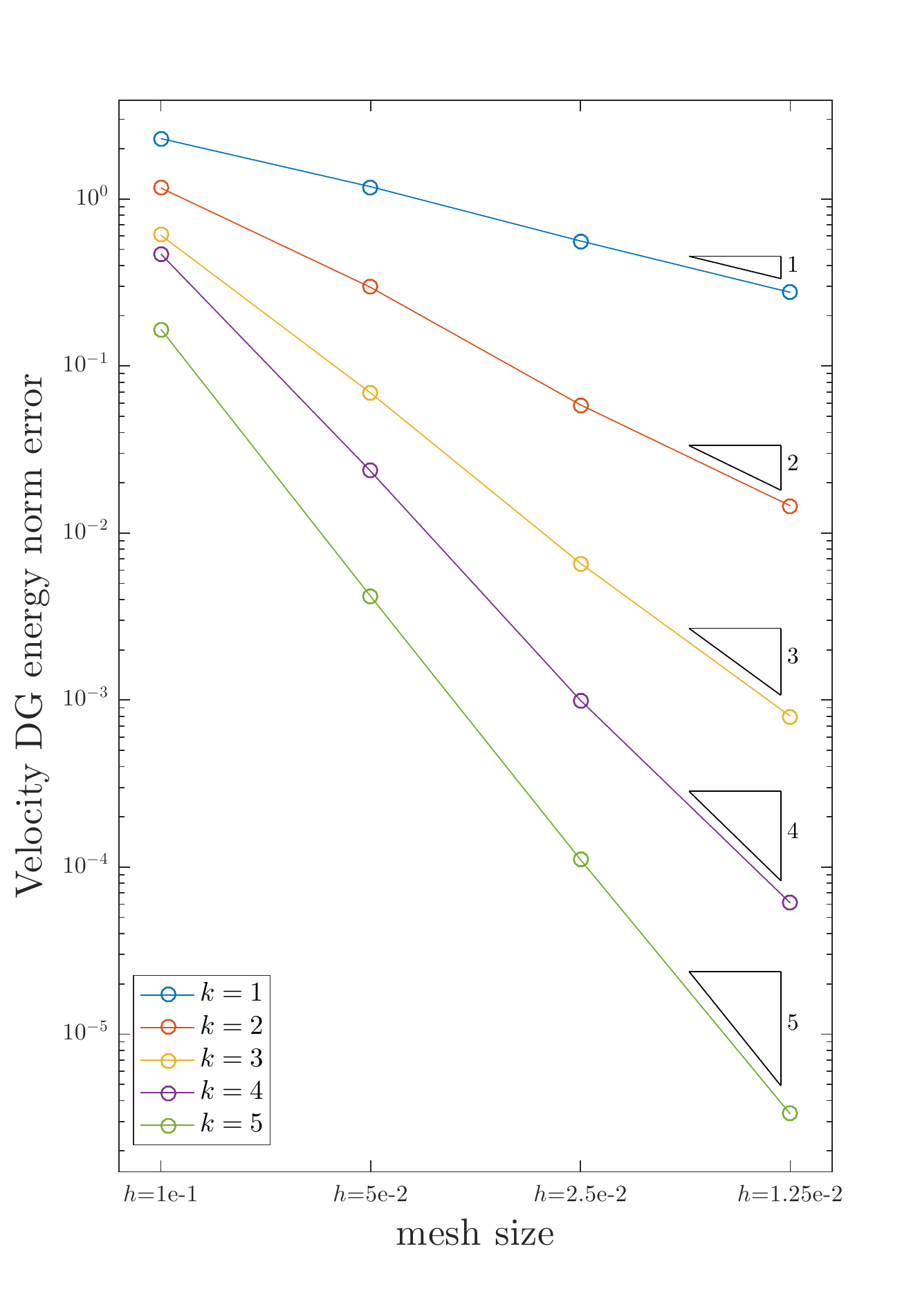}
  \caption{Velocity DG energy norm error with method \br 1 for the
    smooth case on triangular meshes (left) / mixed meshes (right)}
  \label{fig:m1DGerror}
\end{figure}
\begin{figure}
  \centering
  \includegraphics[width=0.48\textwidth]{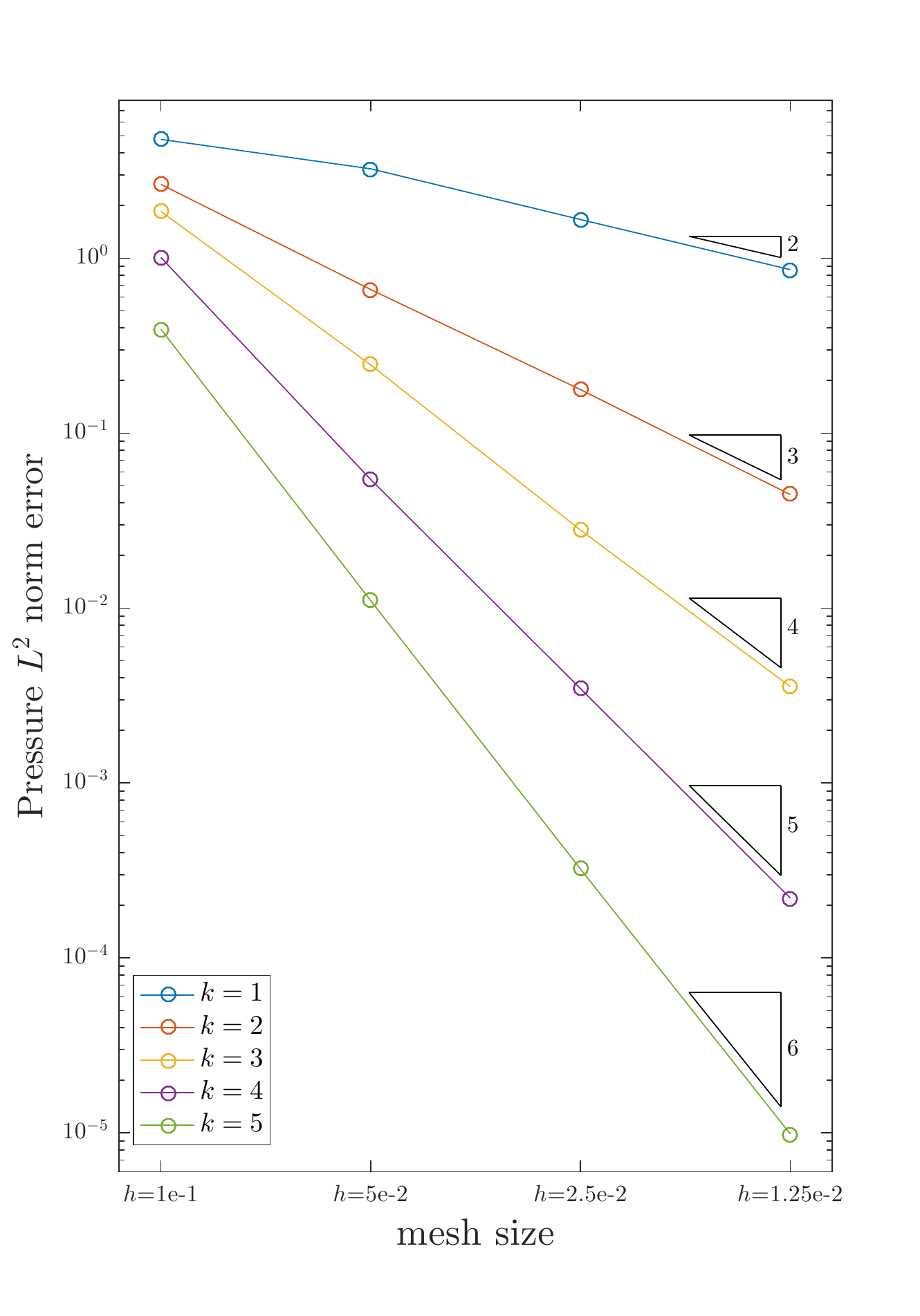}
  \includegraphics[width=0.48\textwidth]{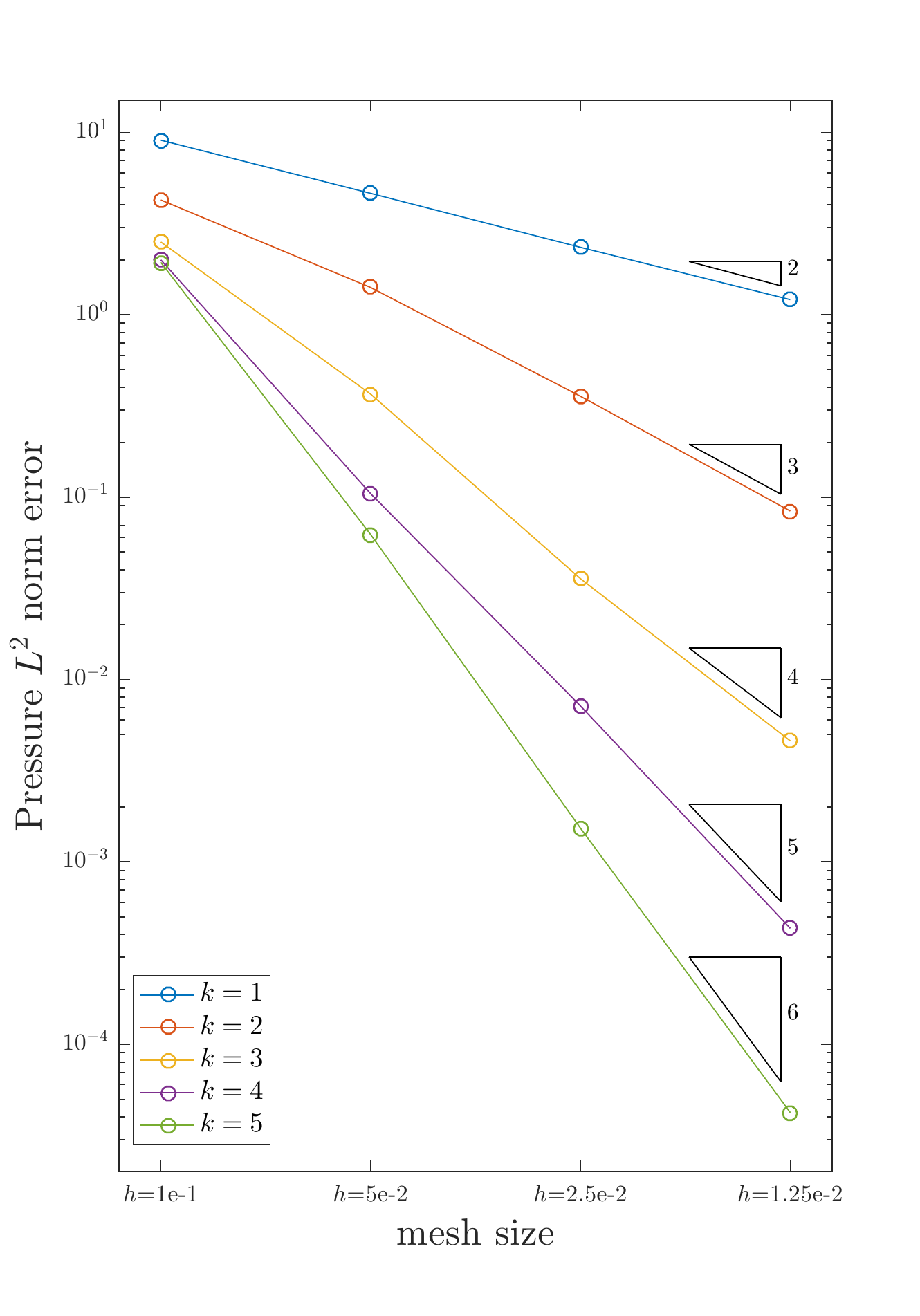}
  \caption{Pressure $L^2$ norm error with method \br 1 for the smooth
    case on triangular meshes (left) / mixed meshes (right)}
  \label{fig:m1Pressureerror}
\end{figure}

Now we consider the method \br 2, the convergence rates are displayed
in Fig \ref{fig:m2L2error}, \ref{fig:m2DGerror} and
\ref{fig:m2Pressureerror}. All convergence orders are identical to the
results in method \br{1}, which agrees with the developed theory. For
this method, the approximation to the pressure converges in a
suboptimal way, but we build the approximation space only once which
makes the method \br{2} more effective.

\begin{figure}
  \centering
  \includegraphics[width=0.48\textwidth]{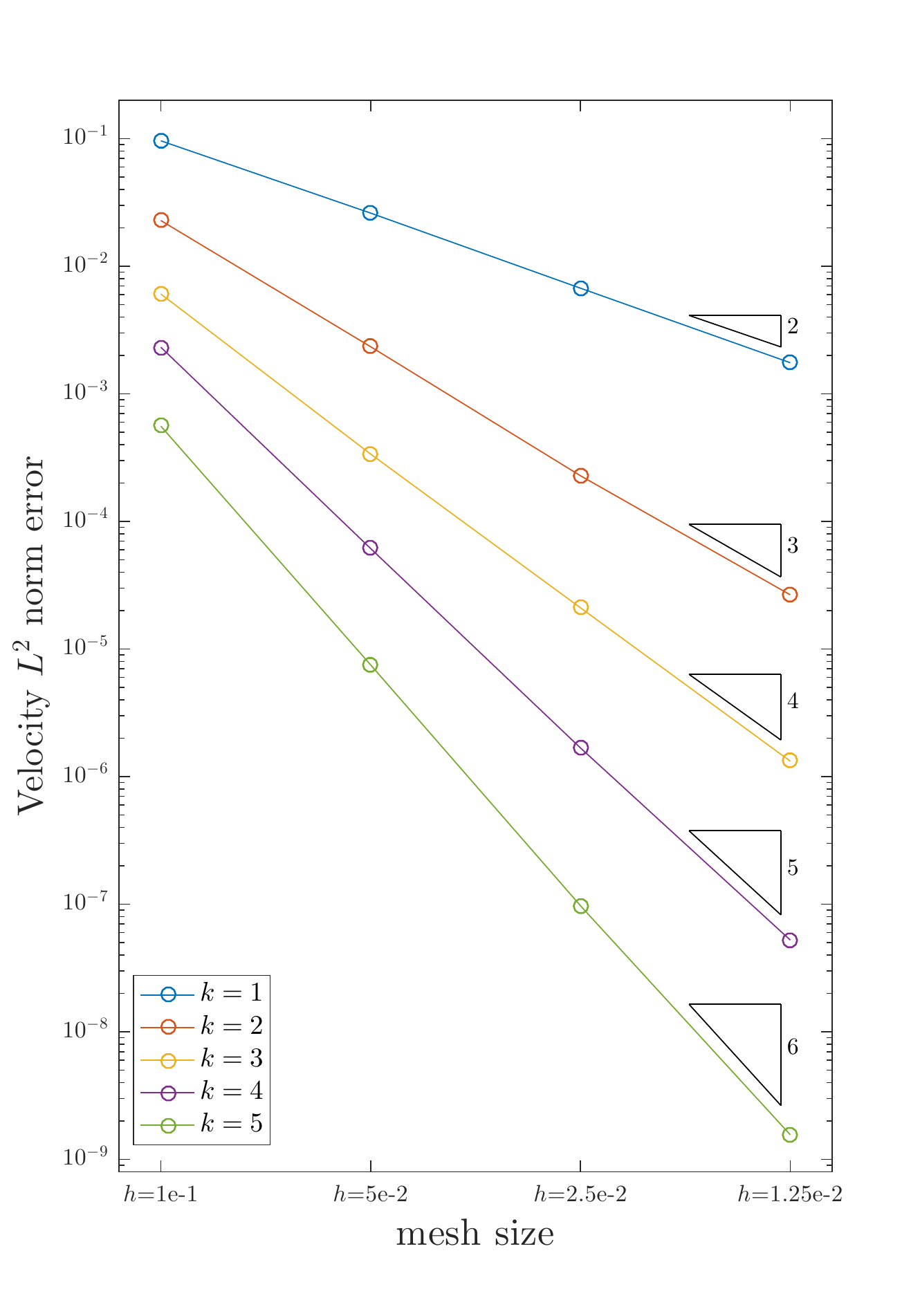}
  \includegraphics[width=0.48\textwidth]{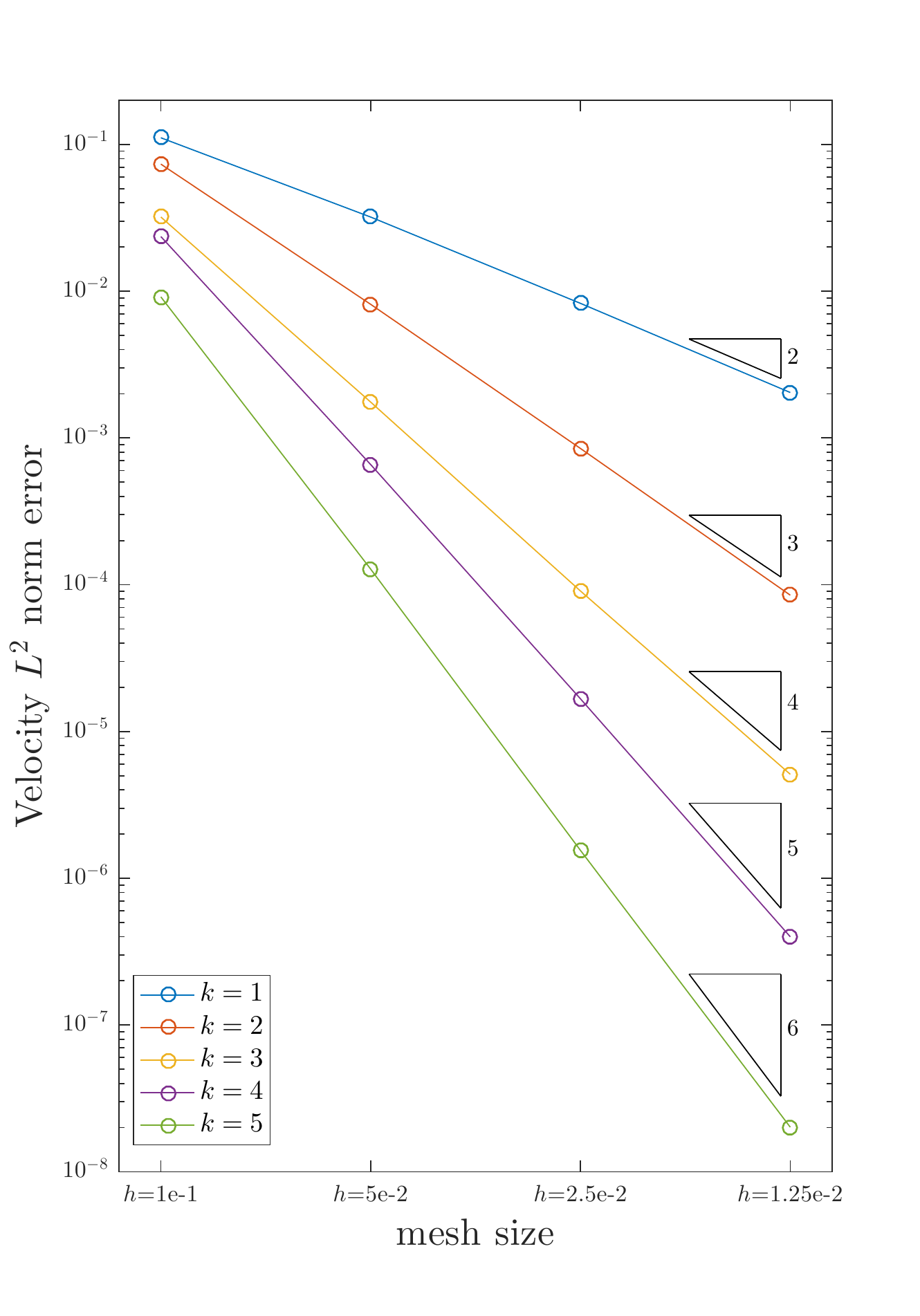}
  \caption{Velocity $L^2$ norm error with method \br 2 for the smooth
    case on triangular meshes(left) /mixed meshes(right)}
  \label{fig:m2L2error}
\end{figure}
\begin{figure}
  \centering
  \includegraphics[width=0.48\textwidth]{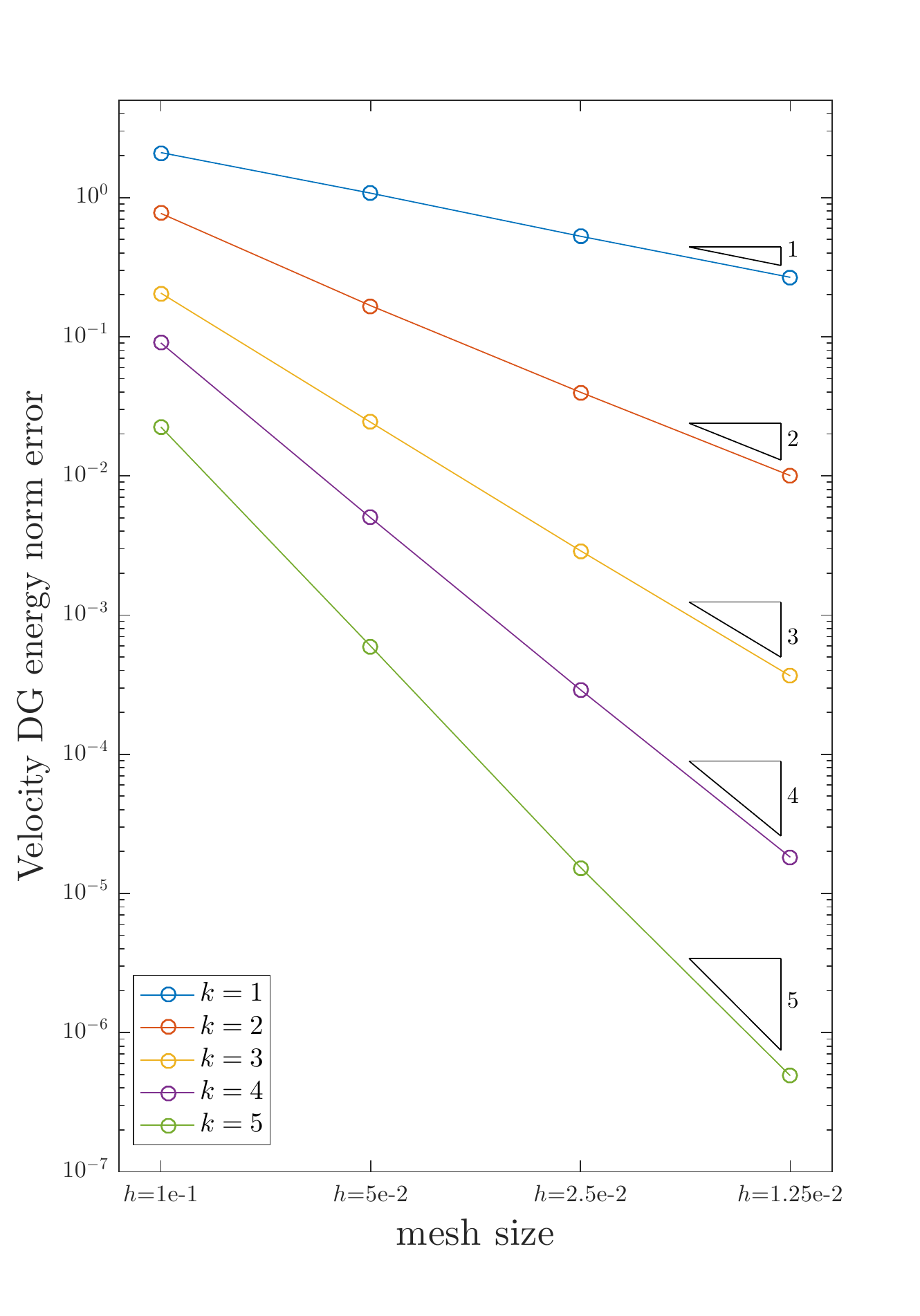}
  \includegraphics[width=0.48\textwidth]{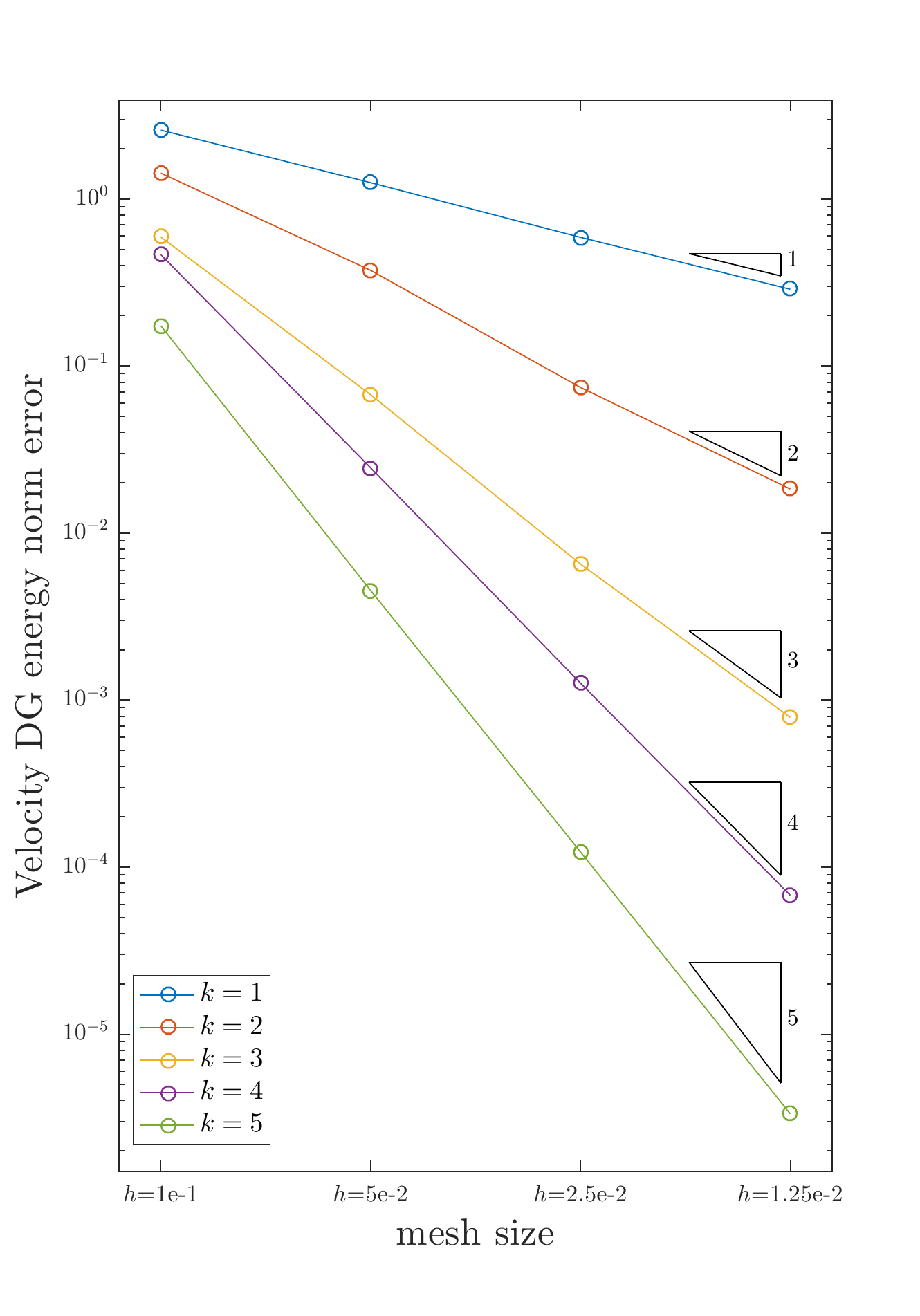}
  \caption{Velocity DG energy norm error with method \br 2 for the
    smooth case on triangular meshes(left)/mixed meshes(right)}
  \label{fig:m2DGerror}
\end{figure}
\begin{figure}
  \centering
  \includegraphics[width=0.48\textwidth]{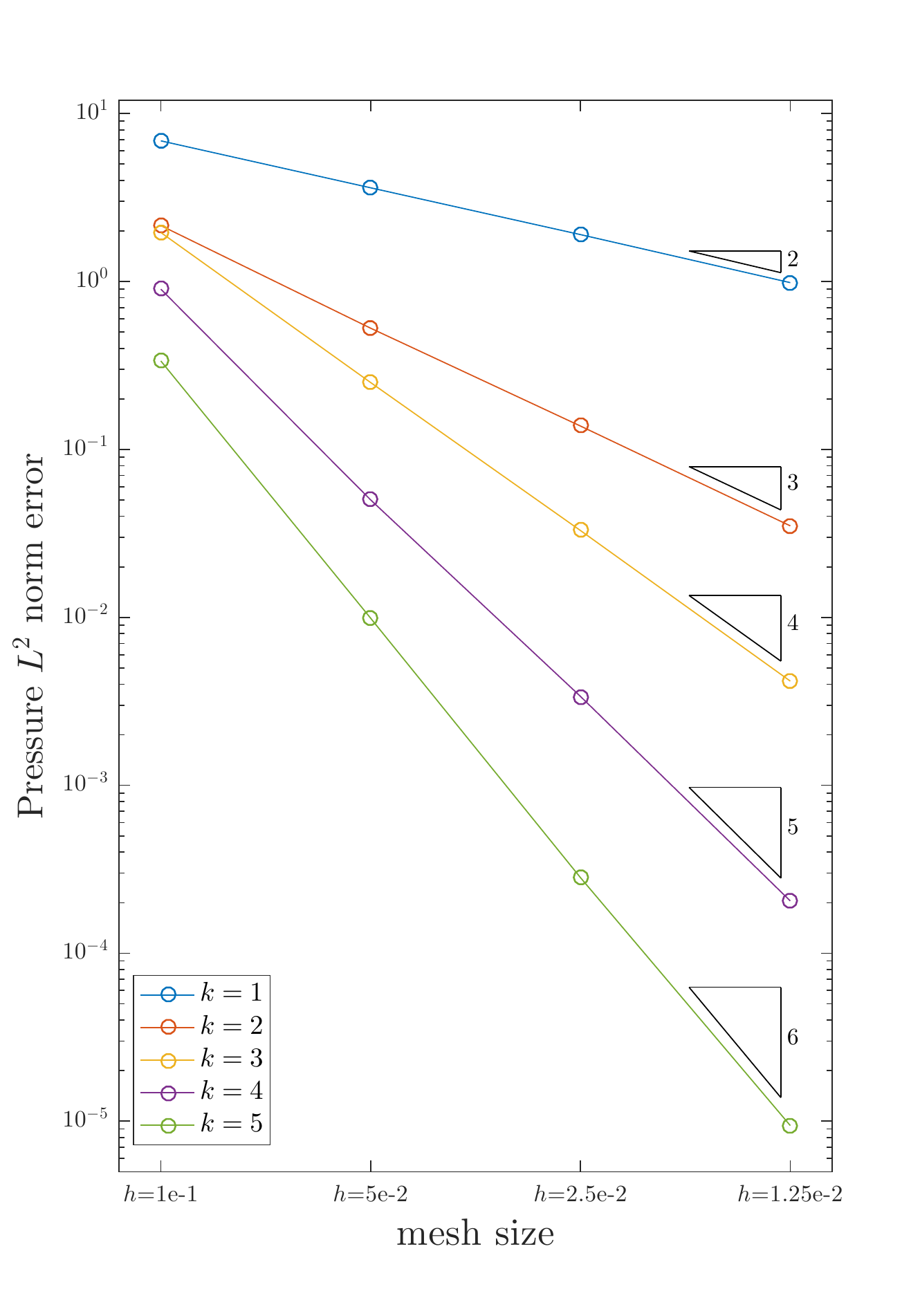}
  \includegraphics[width=0.48\textwidth]{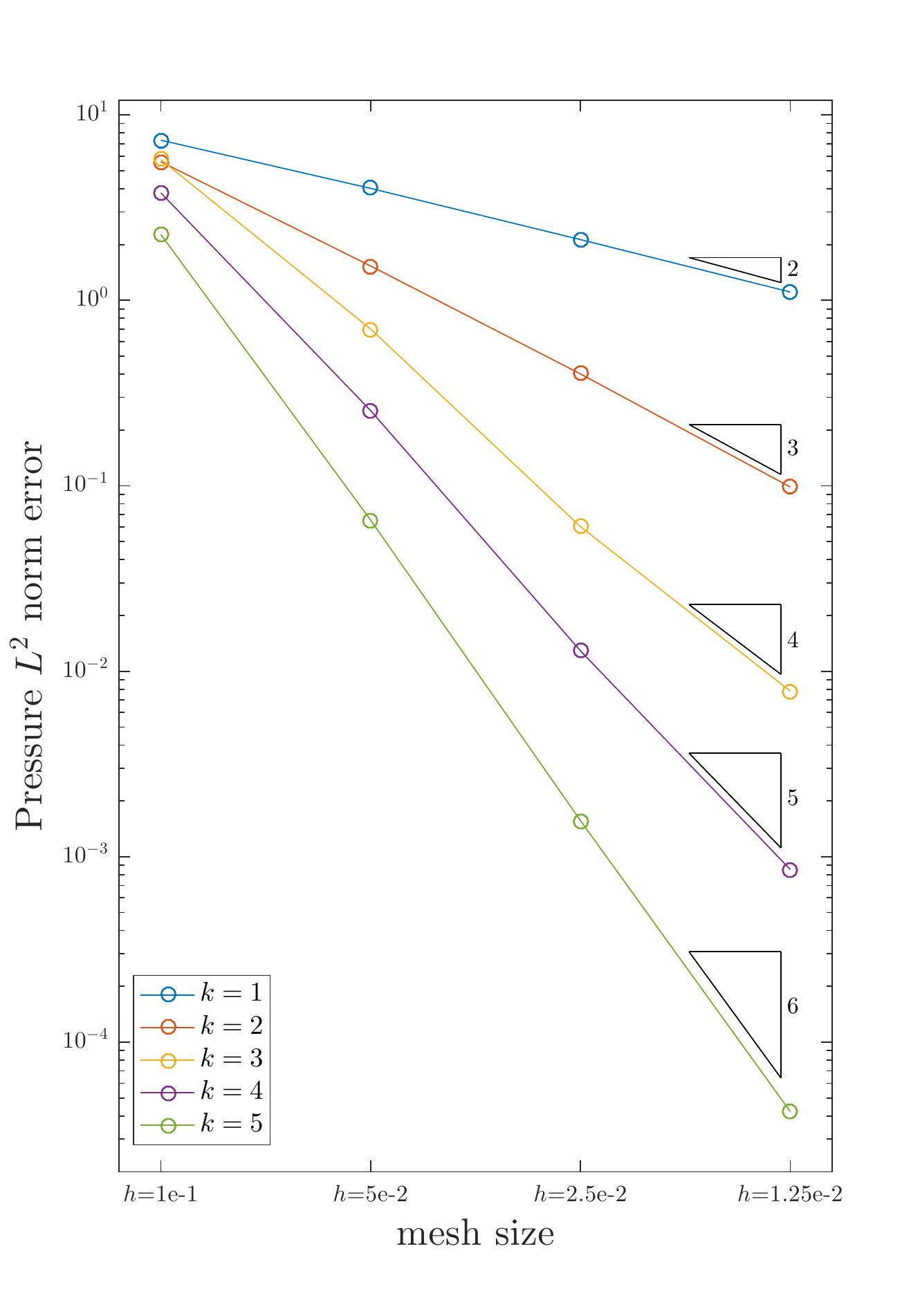}
  \caption{Pressure $L^2$ norm error with method \br 2 for the smooth
    case on triangular meshes (left)/mixed meshes(right)}
  \label{fig:m2Pressureerror}
\end{figure}

Finally, we investigate the numerical performance of the method \br 3.
The errors are plotted in Fig \ref{fig:m3L2error}, \ref{fig:m3DGerror}
and \ref{fig:m3Pressureerror}. Here the theoretical convergence rates
under norm $\|\bm u - \bm u_h\|_{\mathrm{DG}}$ and
$\|p-p_h\|_{L^2(\Omega)}$ are $O(h^1)$. We observe that the numerical
results do not coincide with the theory exactly which results from the
numerical error in approximation to pressure is much larger than the
interpolation error. The super convergence is spurious and the
convergence orders will drop to the expected values as the mesh size
$h$ approaches to zero. But it does not imply that the high order is
not preferred in method \br{3}, {the results show that
  using $\bm V_{h}^{k}$ with larger $k$ could give a more accurate
  approximation to the velocity and the pressure.}

\begin{figure}
  \centering
  \includegraphics[width=0.48\textwidth]{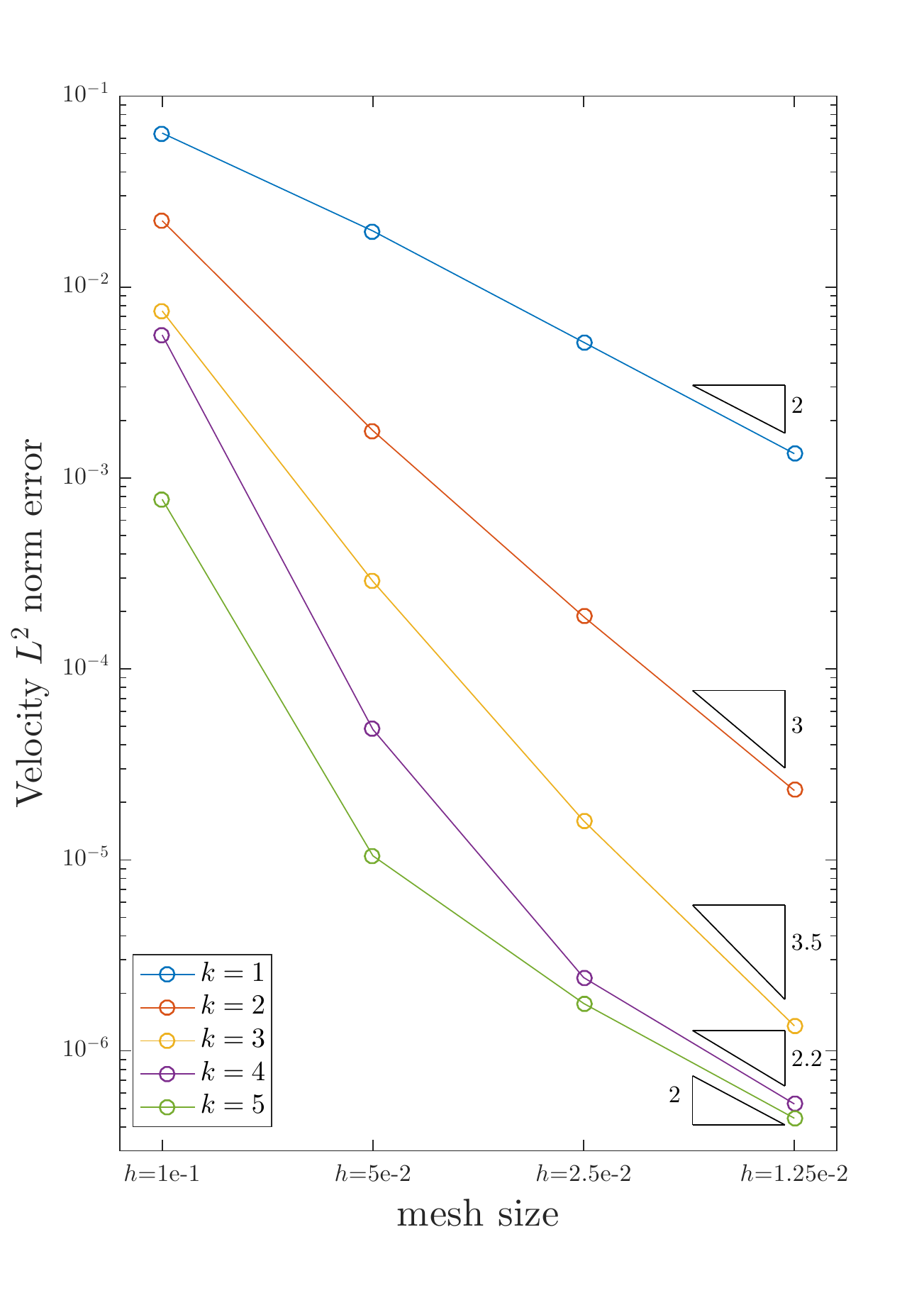}
  \includegraphics[width=0.48\textwidth]{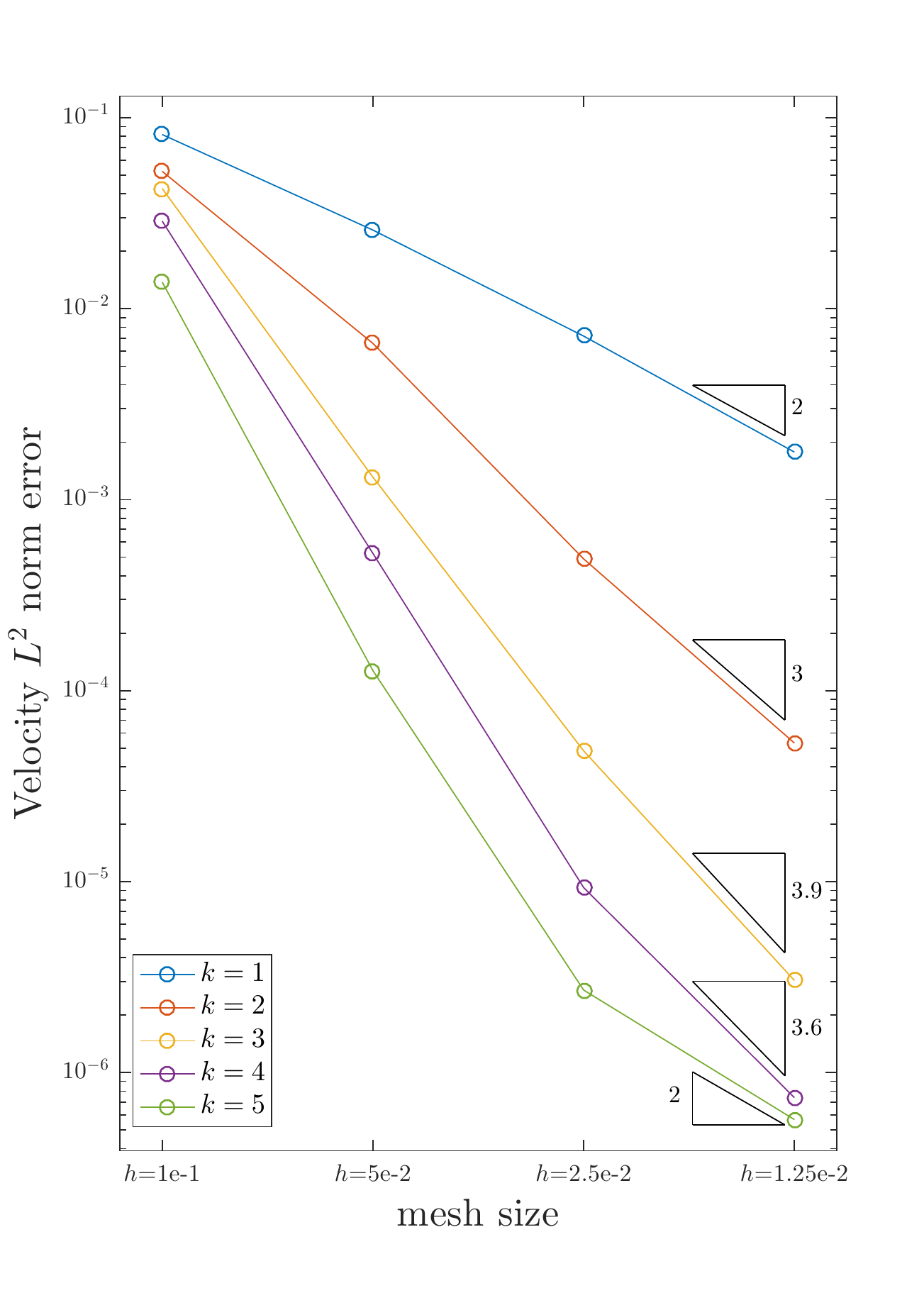}
  \caption{Velocity $L^2$ norm error on with method \br 3 for the
    smooth case triangular meshes(left)/mixed meshes(right)}
  \label{fig:m3L2error}
\end{figure}
\begin{figure}
  \centering
  \includegraphics[width=0.48\textwidth]{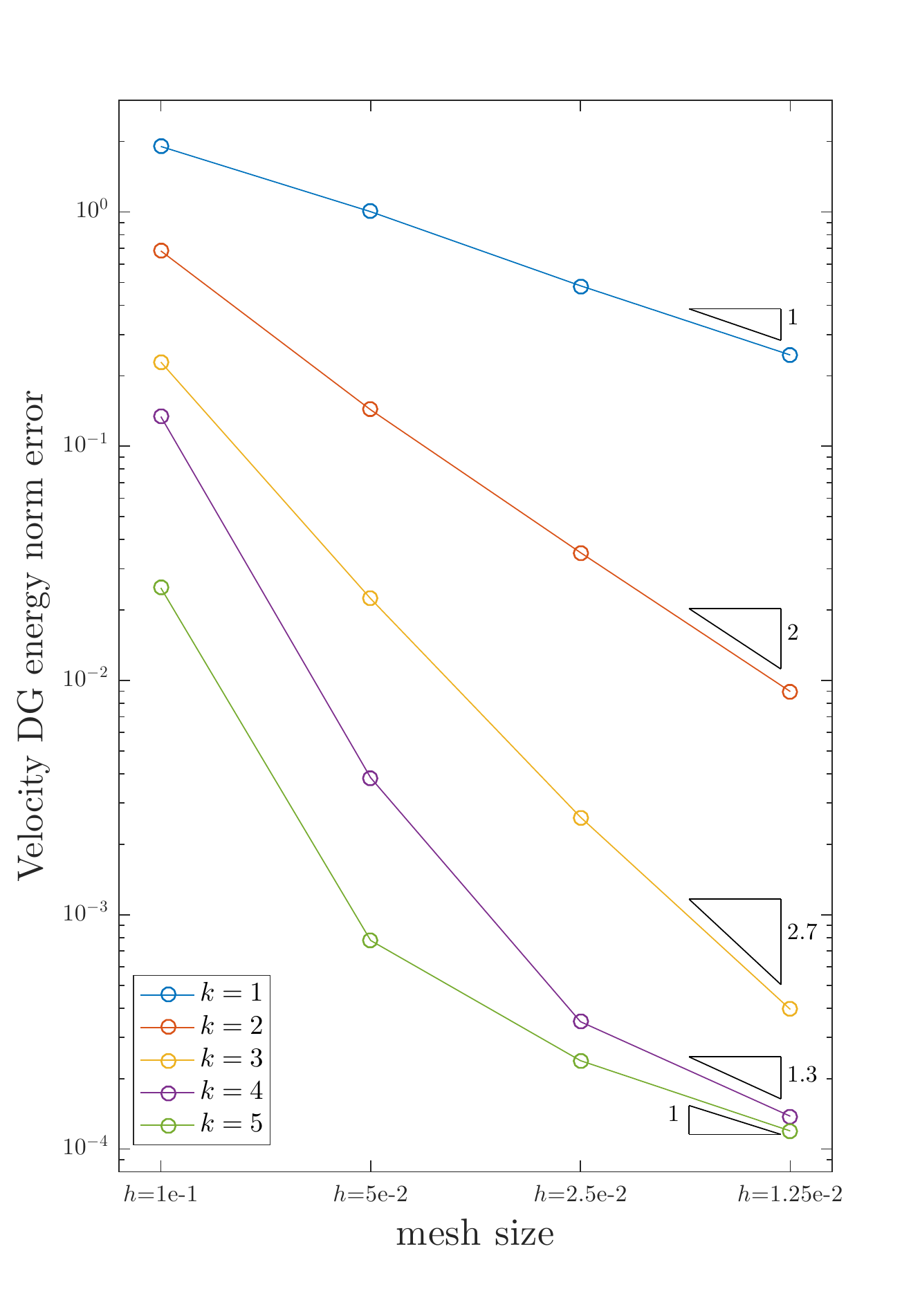}
  \includegraphics[width=0.48\textwidth]{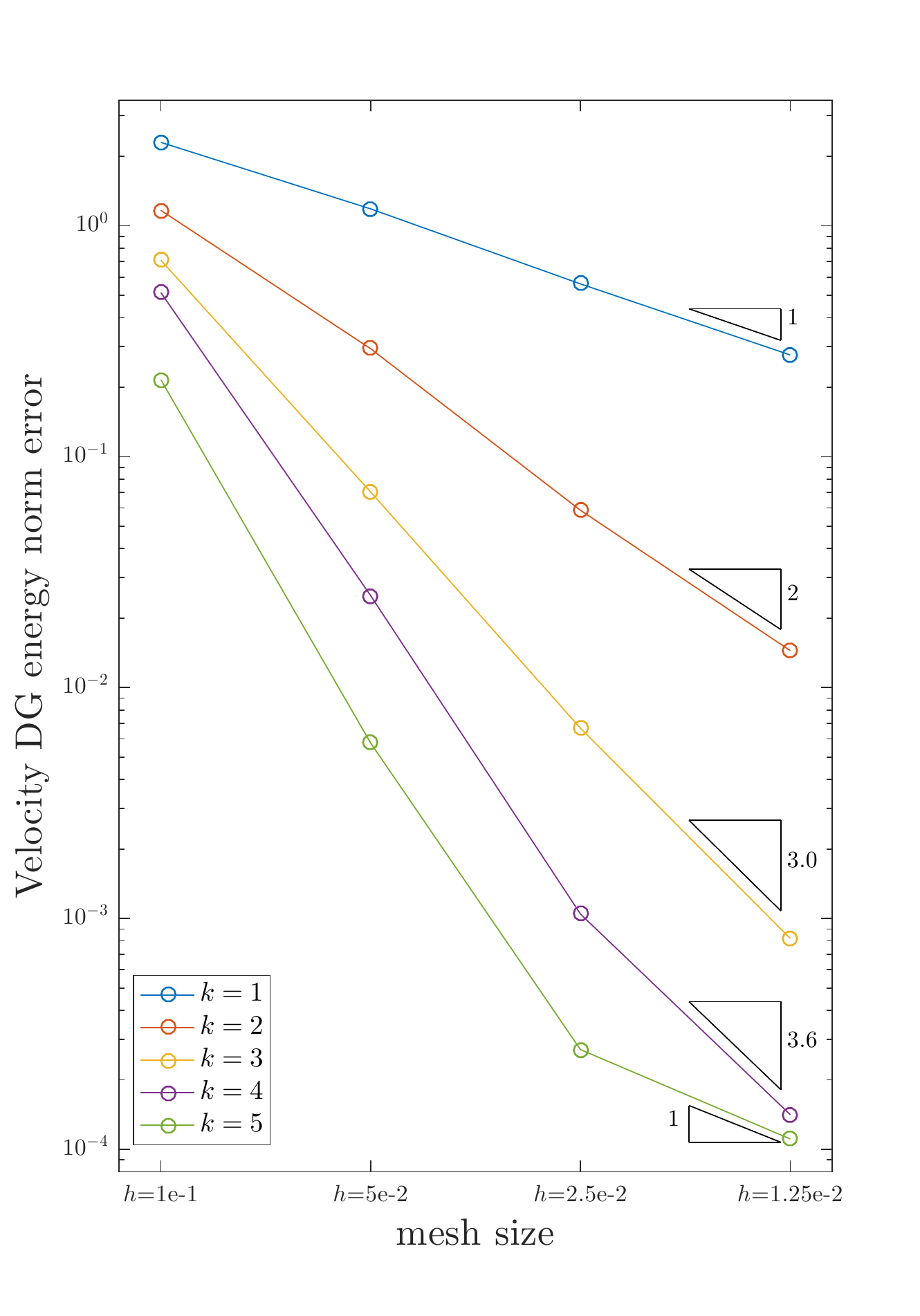}
  \caption{Velocity DG energy norm error with method \br 3 for the
    smooth case on triangular meshes(left)/mixed meshes(right)}
  \label{fig:m3DGerror}
\end{figure}
\begin{figure}
  \centering
  \includegraphics[width=0.48\textwidth]{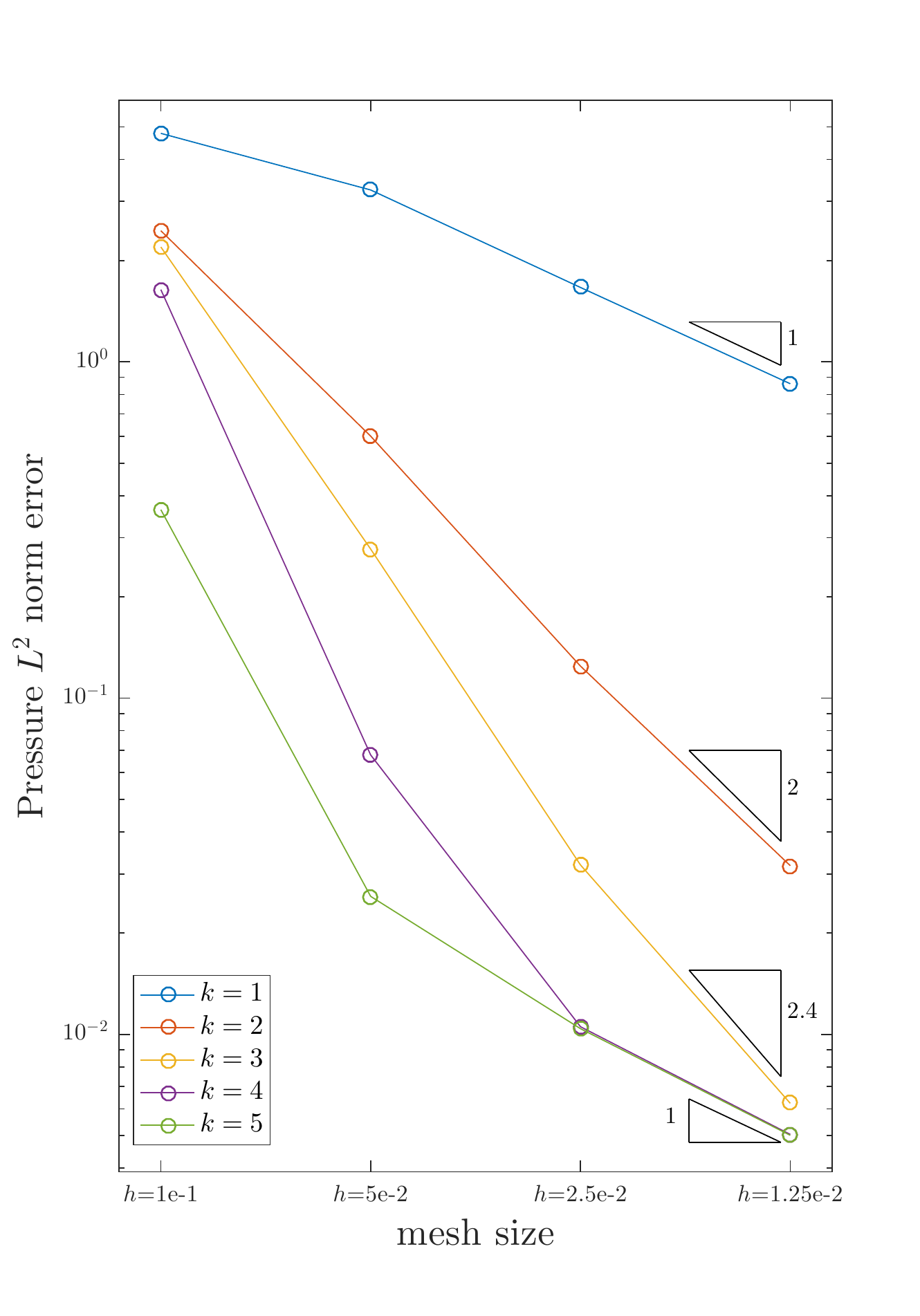}
  \includegraphics[width=0.48\textwidth]{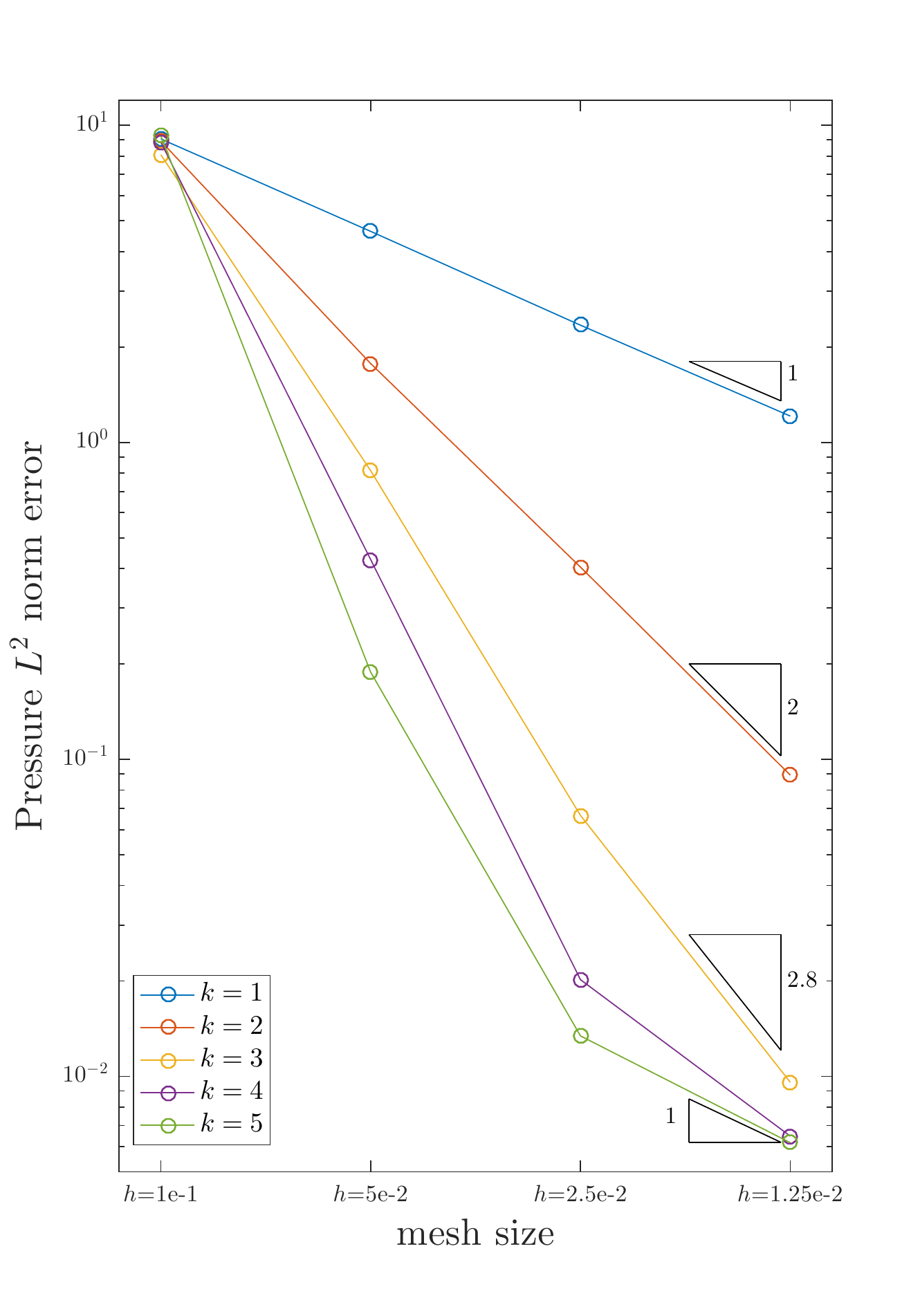}
  \caption{Pressure $L^2$ norm error on with method \br 3 for the
    smooth case triangular meshes (left)/mixed meshes(right)}
  \label{fig:m3Pressureerror}
\end{figure}

\subsection{Driven cavity problem}
The driven cavity problem is a standard benchmark test for the
incompressible flow. It models a plane flow of an isothermal fluid in
a unit square lid-driven cavity. The domain $\Omega$ is $[0,1]^2$ and
the boundary condition and the source term are given by
\begin{displaymath}
  \bm g(x,y)=\begin{cases} (1,0)^T,\quad
  0<x<1,\ y=1,\\ (0,0)^T,\quad\text{otherwise},\\
  \end{cases}\quad \bm f(x,y)=\begin{bmatrix}
    0\\0\\
  \end{bmatrix}.
\end{displaymath}

The domain is partitioned by triangular mesh with mesh size
$h=\frac1{60}$. Fig \ref{fig:m1liddriven} shows the velocity vectors
and the streamline of the flow for the discretization of $\bm
V_{h}^3\times Q_{h}^2$. Fig \ref{fig:m2liddriven} and
\ref{fig:m3liddriven} present the results for the pair of $\bm
V_{h}^3\times Q_{h}^3$ and $\bm V_h^3\times Q_h^0$, respectively.

\begin{figure}
  \centering
  \includegraphics[width=0.48\textwidth]{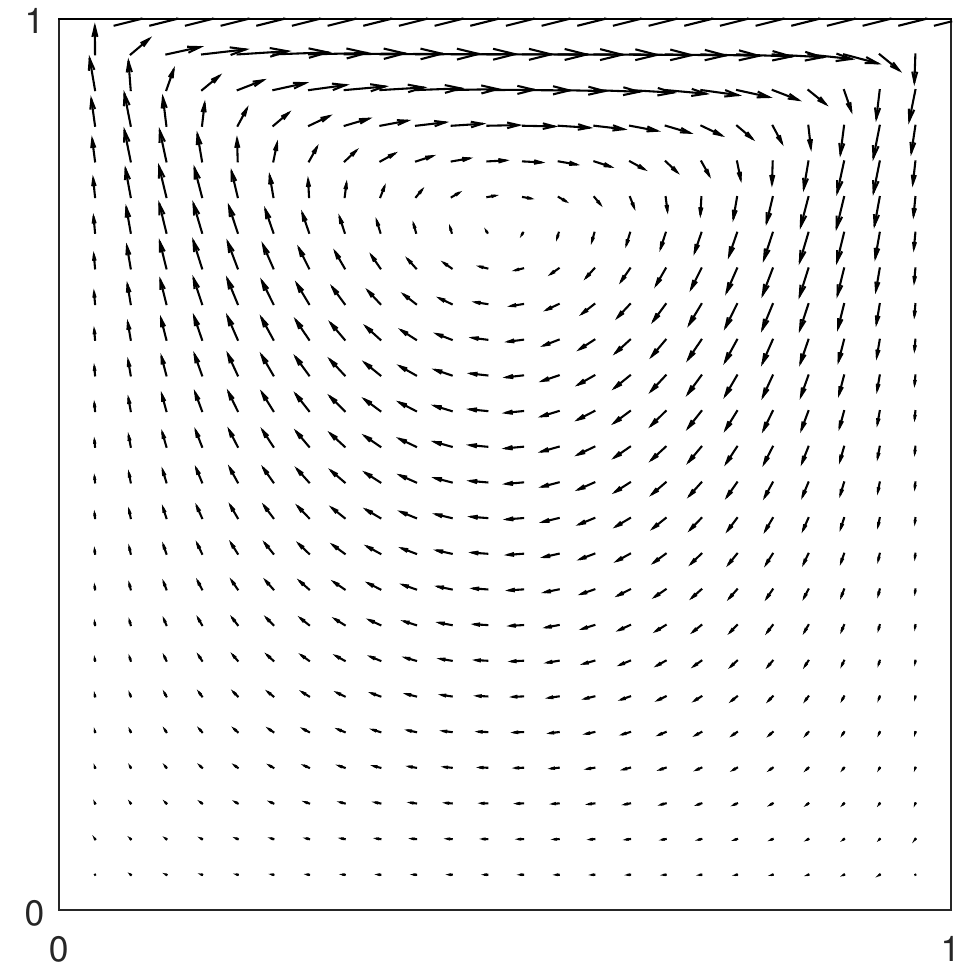}
  \includegraphics[width=0.48\textwidth]{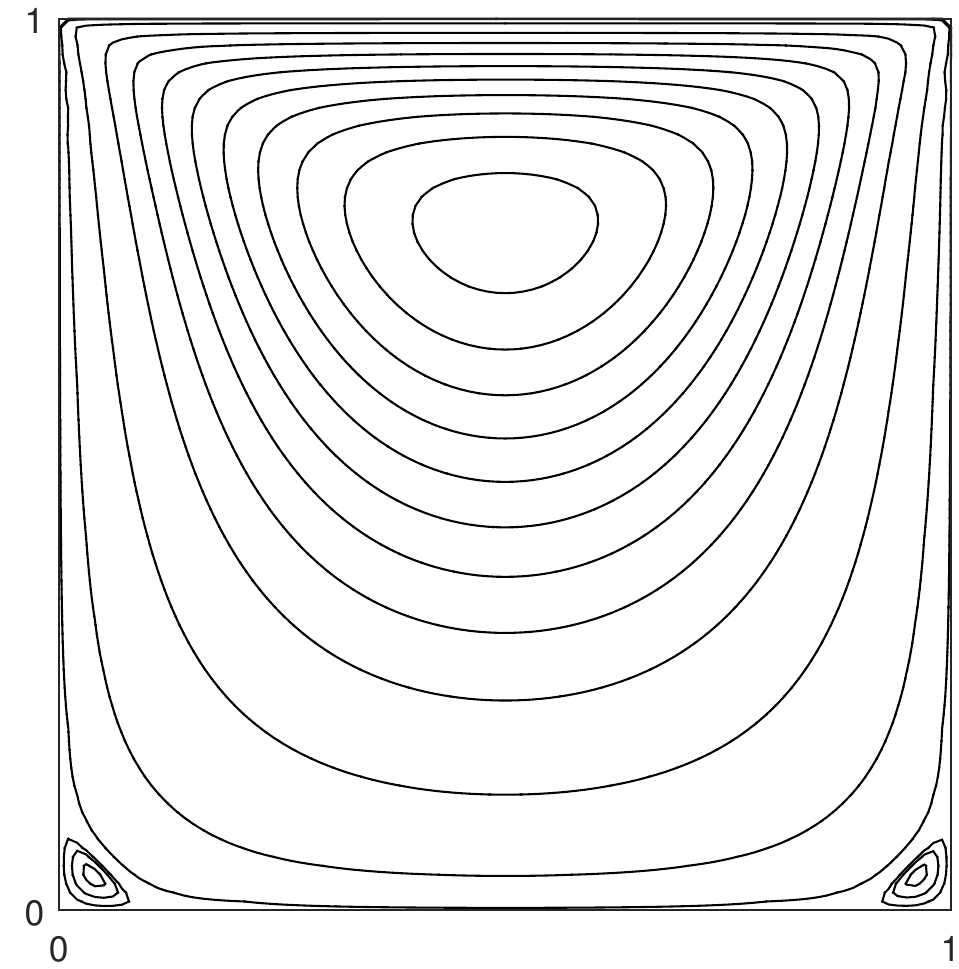}
  \caption{Velocity vectors (left) and the streamline of the flow
    (right) for $\bm V_h^3\times Q_h^2$}
  \label{fig:m1liddriven}
\end{figure}
\begin{figure}
  \centering
  \includegraphics[width=0.48\textwidth]{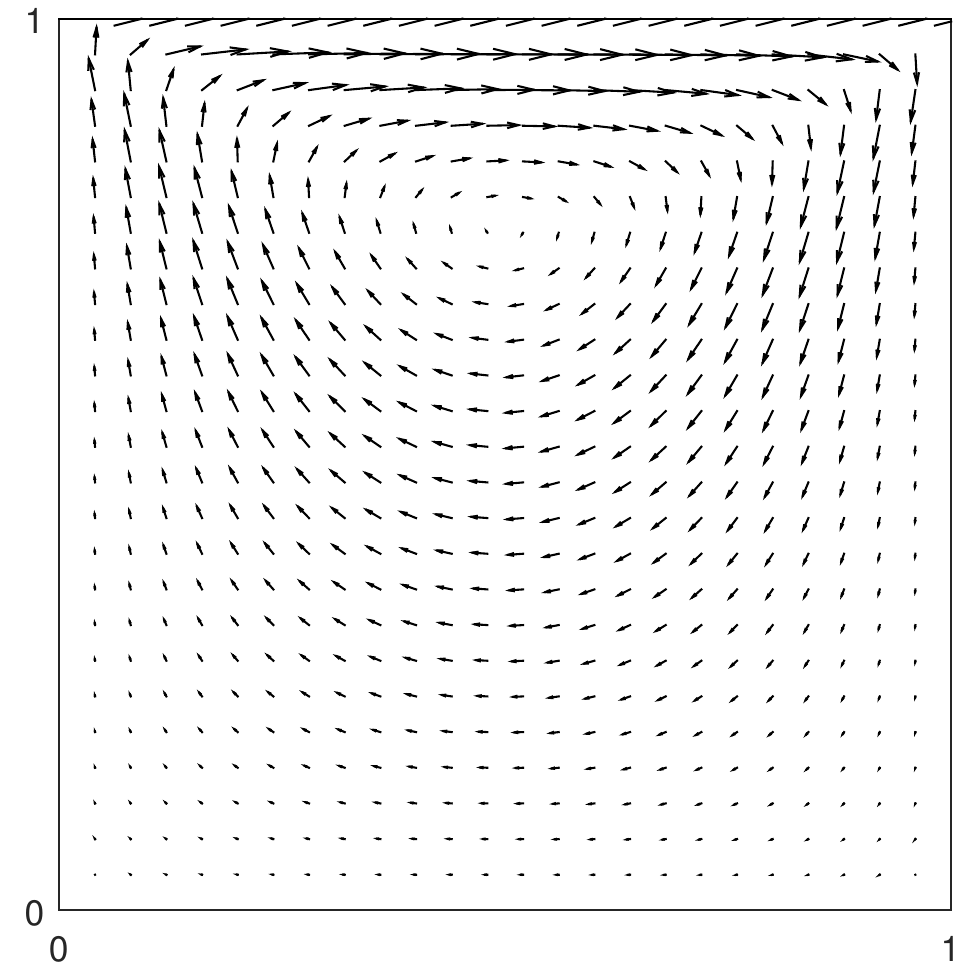}
  \includegraphics[width=0.48\textwidth]{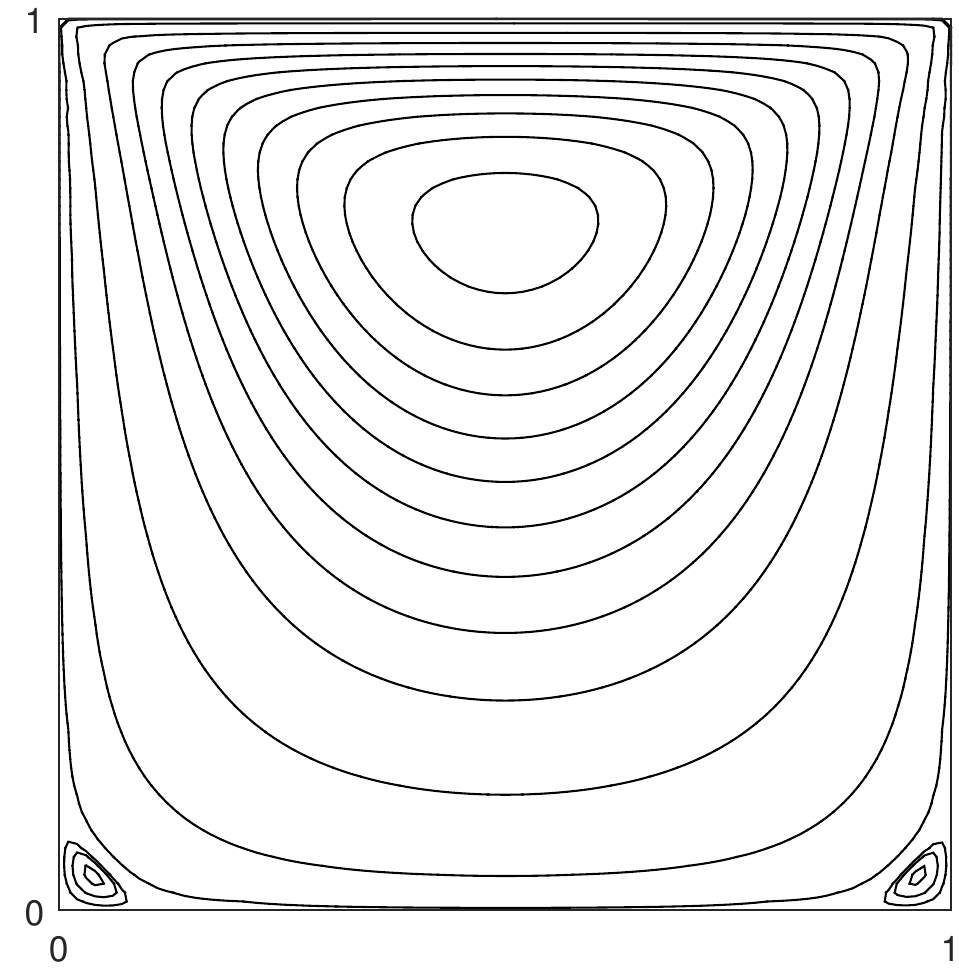}
  \caption{Velocity vectors(left) and the streamline of the
    flow(right) for $\bm V_h^3\times Q_{h}^3$}
  \label{fig:m2liddriven}
\end{figure}
\begin{figure}
  \centering
  \includegraphics[width=0.48\textwidth]{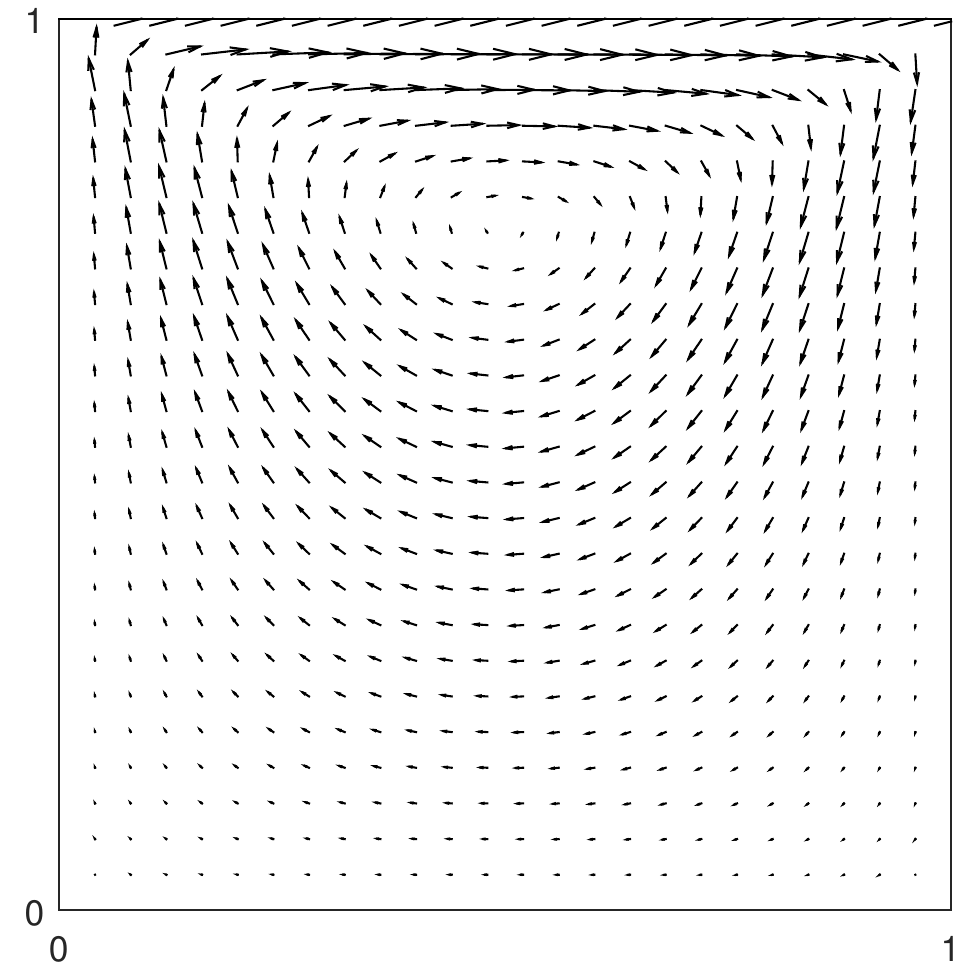}
  \includegraphics[width=0.48\textwidth]{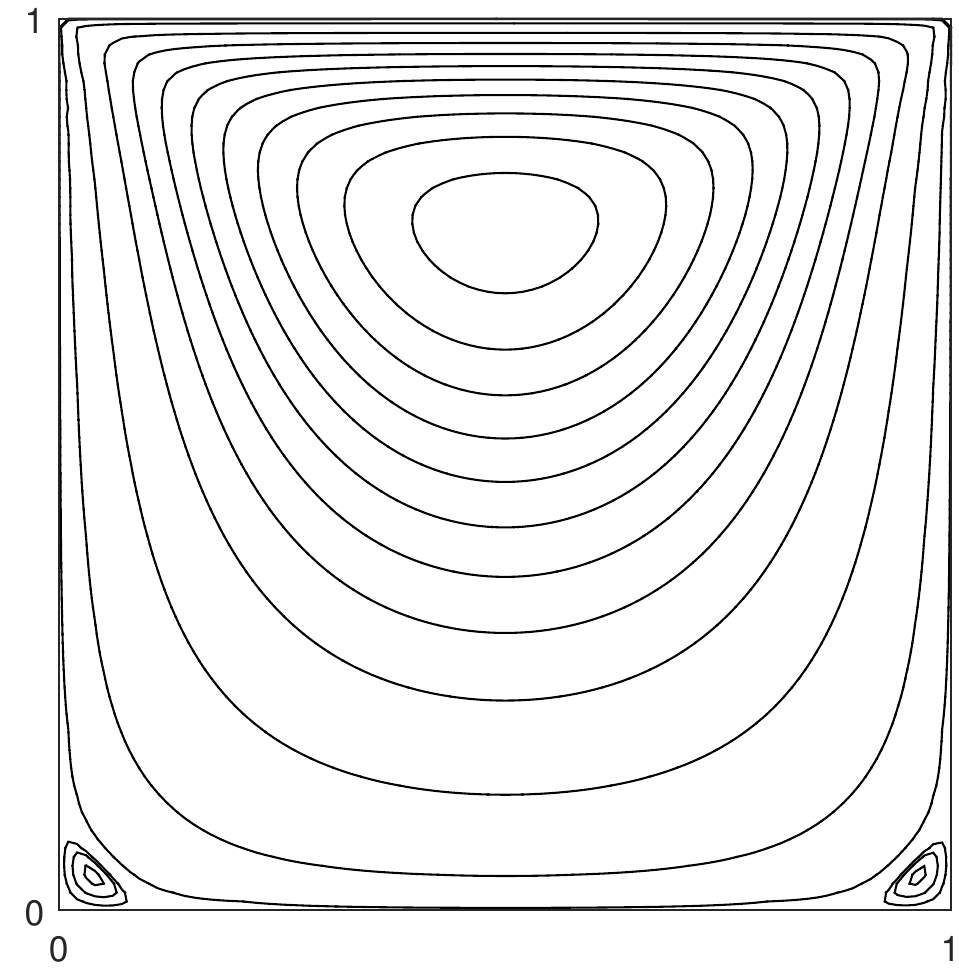}
  \caption{Velocity vectors(left) and the streamline of the
    flow(right) for $\bm V_{h}^3\times Q_{h}^0$}
  \label{fig:m3liddriven}
\end{figure}

\subsection{Non-smooth problem}
In this example, we investigate the performance of our method dealing
with the Stokes problem with a corner singularity in the analytical
solution. Let $\Omega$ be the L-shaped domain
$[-1,1]\times[-1,1]\backslash [0,1)\times(-1,0]$ and the meshes we
use, which are generated by \emph{gmsh}, are refinements of a
triangular mesh of 250 triangles(see Fig \ref{fig:Lshape}). The exact
solution(from \cite{verfurth1996review, hansbo2008piecewise}) is given
by
\begin{displaymath}
  \bm u(r, \theta)=r^\lambda\begin{bmatrix}
  (1+\lambda)\sin(\theta)\psi(\theta)+\cos(\theta)\psi'(\theta)
  \\ \sin(\theta)\psi'(\theta)-(1+\lambda)\cos(\theta)\psi(\theta) \\
  \end{bmatrix},
\end{displaymath}
in polar coordinates, where
\begin{displaymath}
  \begin{aligned}
    \psi(\theta)=&\frac1{1+\lambda}\sin( (1+\lambda)\theta)
    \cos(\lambda\omega)-\cos(
    (1+\lambda)\theta)\\ &-\frac1{1-\lambda}\sin( (1-\lambda)\theta)
    \cos(\lambda\omega) + \cos( (1-\lambda)\theta),\\
  \end{aligned}
\end{displaymath}
with $\omega=\frac32\pi$ and $\lambda\approx0.5444837$ as the smallest
positive root to
\begin{displaymath}
  \sin(\lambda\omega) + \lambda\sin\omega=0.
\end{displaymath}
At the corner $(0,0)$, the exact solution contains a singularity which
indicates $\bm u(r, \theta)$ does not belong to $H^{2}(\Omega)$.
\begin{figure}[htp]
  \centering
  \includegraphics[width=0.4\textwidth]{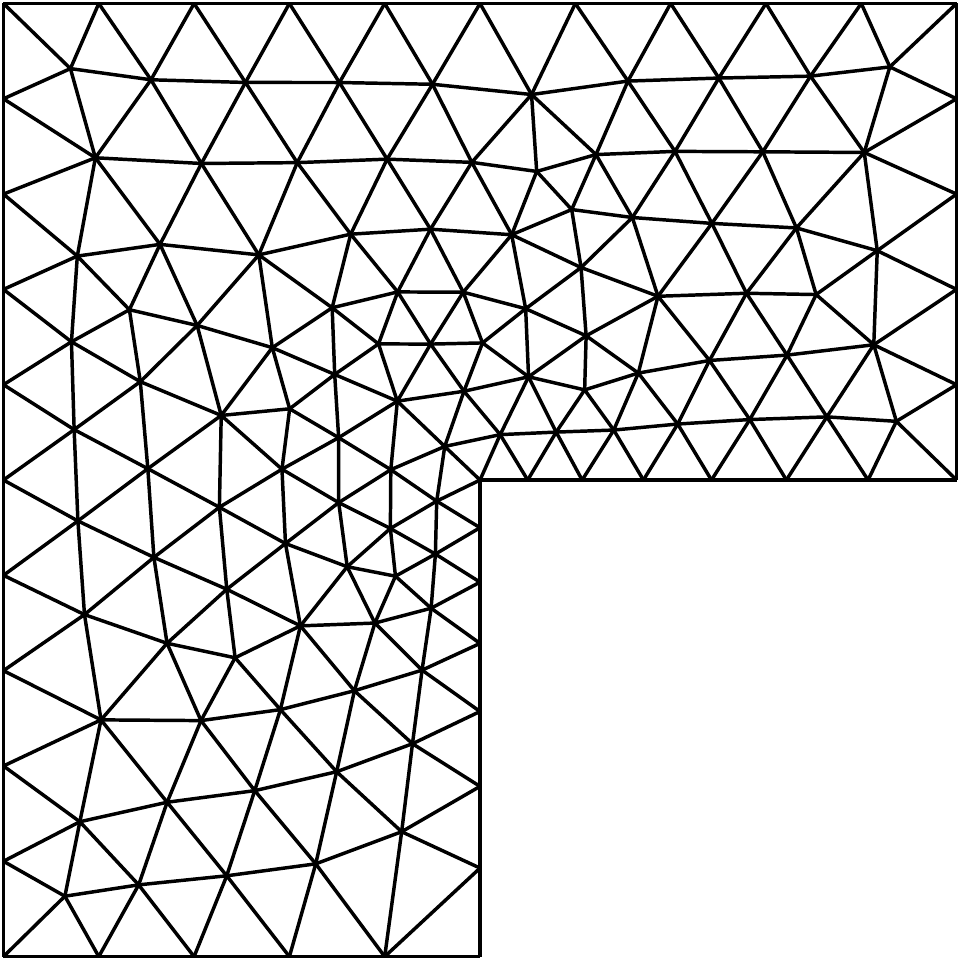}
  \hspace{25pt}
  \includegraphics[width=0.4\textwidth]{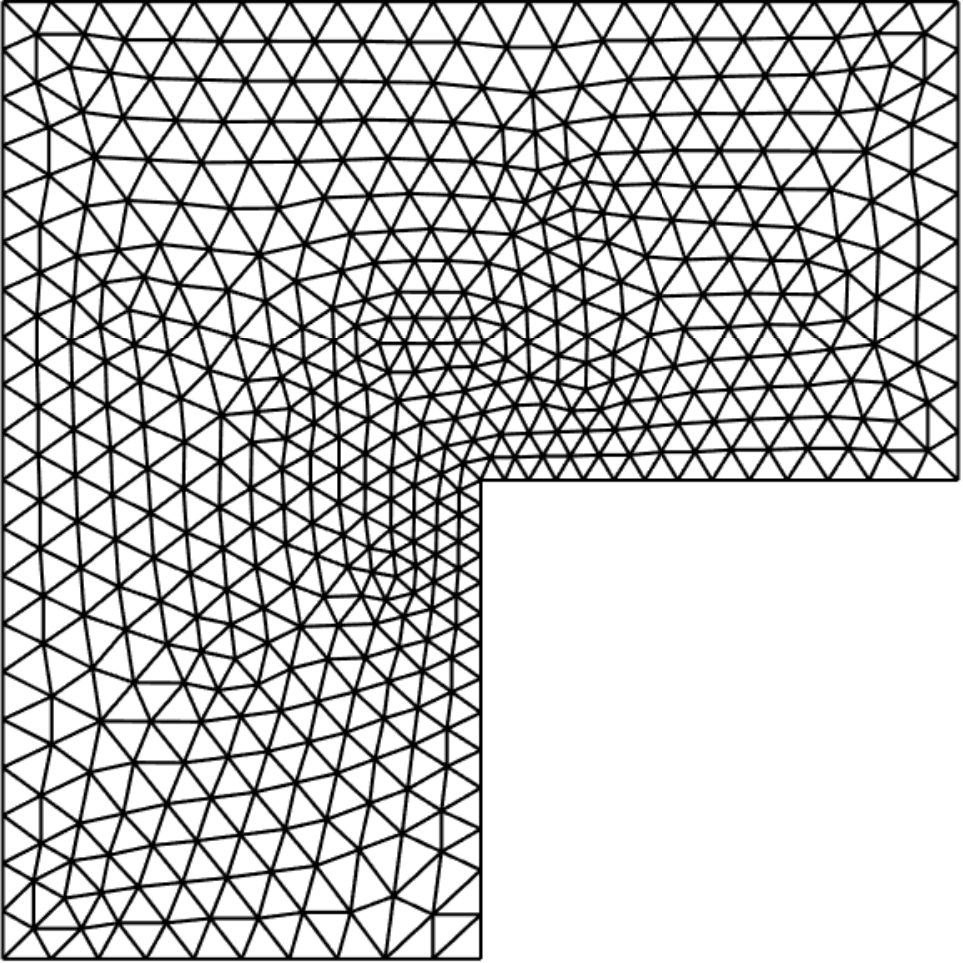}
  \caption{The triangular meshes of L-shaped domain, 250 elements
    (left)/ 1000 elements (right)}
  \label{fig:Lshape}
\end{figure}

The $\# S(K)$ is chosen also as the Tab \ref{tab:patchnumber2d} shows.
In Tab \ref{tab:Lshapeerror} we list the $L^2$ norm error of the
velocity against the degrees of freedom for different pairs of
approximation spaces. We observe that all convergence orders are about
1, which are consistent with the results in \cite{hansbo2008piecewise}
where a piece divergence-free discontinuous Galerkin method is
developed to solve this problem.

\begin{table}[!htp]
  \newcommand\me{\mathrm e} \renewcommand\arraystretch{1.8} \centering
  \caption{Convergence orders of nonsmooth example in L-shaped domain}
  \vspace{-8pt}
  \label{tab:Lshapeerror}
  \scalebox{0.76}{
  \begin{tabular}{|l|l|l|l|l|l|l|l|l|l|}
    \hline \multirow{2}{*}{Method} & 250 DOFs &
    \multicolumn{2}{|l|}{1000 DOFs} & \multicolumn{2}{|l|}{4000 DOFs}
    & \multicolumn{2}{|l|}{16000 DOFs}& \multicolumn{2}{|l|}{64000
      DOFs}\\ \cline{2-10} & $L^2$ error & $L^2$ error & order& $L^2$
    error & order& $L^2$ error & order& $L^2$ error & order\\ \hline
    $V_{h}^2\times Q_{h}^1$ & $5.57\mathrm{E}{-2}$ &
    $2.40\mathrm{E}{-2}$& 1.21 & $9.56\mathrm{E}{-3}$& 1.33&
    $4.61\mathrm{E}{-3}$& 1.05& $2.13\mathrm{E}{-3}$& 1.10\\ \hline
    $V_{h}^3\times Q_{h}^2$ & $4.97\mathrm{E}{-2}$ &
    $1.89\mathrm{E}{-2}$& 1.33 & $9.29\mathrm{E}{-3}$& 1.09&
    $4.51\mathrm{E}{-3}$& 1.04& $2.21\mathrm{E}{-3}$& 1.03\\ \hline
    $V_{h}^2\times Q_{h}^2$ & $4.44\mathrm{E}{-2}$ &
    $1.58\mathrm{E}{-2}$& 1.49 & $7.46\mathrm{E}{-3}$& 1.08&
    $3.61\mathrm{E}{-3}$& 1.05& $1.73\mathrm{E}{-3}$& 1.06\\ \hline
    $V_{h}^3\times Q_{h}^3$ & $4.83\mathrm{E}{-2}$ &
    $1.87\mathrm{E}{-2}$& 1.37 & $8.66\mathrm{E}{-3}$& 1.11&
    $4.11\mathrm{E}{-3}$& 1.07& $1.96\mathrm{E}{-3}$& 1.07\\ \hline
    $V_{h}^2\times Q_{h}^0$ & $6.59\mathrm{E}{-2}$ &
    $2.22\mathrm{E}{-2}$& 1.57 & $1.03\mathrm{E}{-2}$& 1.11&
    $5.10\mathrm{E}{-3}$& 1.01& $2.49\mathrm{E}{-3}$& 1.03\\ \hline
    $V_{h}^3\times Q_{h}^0$ & $6.88\mathrm{E}{-2}$ &
    $2.71\mathrm{E}{-2}$& 1.34 & $9.75\mathrm{E}{-3}$& 1.47&
    $4.58\mathrm{E}{-3}$& 1.09& $2.18\mathrm{E}{-3}$& 1.07\\ \hline
  \end{tabular}
  }
\end{table}

\section{Conclusion}
In this paper, we have introduced a new discontinuous Galerkin method
to solve the Stokes problem. A novelty of this method is the new
piecewise polynomial space that is reconstructed by solving local
least squares problem. A variety of numerical inf-sup tests
demonstrate the stability of this method. The optimal error estimates
in $L^2$ norm and DG energy norm are presented and the numerical
results are reported to show good agreement with the theoretical
predictions.

\bibliographystyle{elsarticle-num}

\begin{thebibliography}{}
\expandafter\ifx\csname url\endcsname\relax
  \def\url#1{\texttt{#1}}\fi
\expandafter\ifx\csname urlprefix\endcsname\relax\def\urlprefix{URL }\fi
\expandafter\ifx\csname href\endcsname\relax
  \def\href#1#2{#2} \def\path#1{#1}\fi

\end{thebibliography}


\begin{thebibliography}{10}
\expandafter\ifx\csname url\endcsname\relax
  \def\url#1{\texttt{#1}}\fi
\expandafter\ifx\csname urlprefix\endcsname\relax\def\urlprefix{URL }\fi
\expandafter\ifx\csname href\endcsname\relax
  \def\href#1#2{#2} \def\path#1{#1}\fi

\bibitem{boffi2013mixed}
D.~Boffi, F.~Brezzi, M.~Fortin,
  \href{https://doi.org/10.1007/978-3-642-36519-5}{{Mixed Finite Element
  Methods and Applications}}, Vol.~44 of Springer Series in Computational
  Mathematics, Springer, Heidelberg, 2013.
\newline\urlprefix\url{https://doi.org/10.1007/978-3-642-36519-5}

\bibitem{girault1986finite}
V.~Girault, P.~A. Raviart,
  \href{http://dx.doi.org/10.1007/978-3-642-61623-5}{Finite Element Methods for
  Navier-Stokes Equations: Theory and Algorithms}, Springer-Verlag, 1986.
\newline\urlprefix\url{http://dx.doi.org/10.1007/978-3-642-61623-5}

\bibitem{taylor1973numerical}
C.~Taylor, P.~Hood, \href{https://doi.org/10.1016/0045-7930(73)90027-3}{A
  numerical solution of the {N}avier-{S}tokes equations using the finite
  element technique}, Internat. J. Comput. \& Fluids 1~(1) (1973) 73--100.
\newline\urlprefix\url{https://doi.org/10.1016/0045-7930(73)90027-3}

\bibitem{cockburn2000development}
B.~Cockburn, G.~E. Karniadakis, C.~W. Shu,
  \href{https://doi.org/10.1007/978-3-642-59721-3_1}{The development of
  discontinuous {G}alerkin methods}, in: Discontinuous {G}alerkin methods
  ({N}ewport, {RI}, 1999), Vol.~11 of Lect. Notes Comput. Sci. Eng., Springer,
  Berlin, 2000, pp. 3--50.
\newline\urlprefix\url{https://doi.org/10.1007/978-3-642-59721-3_1}

\bibitem{hansbo2002discontinuous}
P.~Hansbo, M.~G. Larson,
  \href{https://doi.org/10.1016/S0045-7825(01)00358-9}{Discontinuous {G}alerkin
  methods for incompressible and nearly incompressible elasticity by
  {N}itsche's method}, Comput. Methods Appl. Mech. Engrg. 191~(17-18) (2002)
  1895--1908.
\newline\urlprefix\url{https://doi.org/10.1016/S0045-7825(01)00358-9}

\bibitem{toselli2002hp}
A.~Toselli, \href{https://doi.org/10.1142/S0218202502002240}{{$hp$}
  discontinuous {G}alerkin approximations for the {S}tokes problem}, Math.
  Models Methods Appl. Sci. 12~(11) (2002) 1565--1597.
\newline\urlprefix\url{https://doi.org/10.1142/S0218202502002240}

\bibitem{schotzau2002mixed}
D.~Sch\"otzau, C.~Schwab, A.~Toselli,
  \href{https://doi.org/10.1137/S0036142901399124}{Mixed {$hp$}-{DGFEM} for
  incompressible flows}, SIAM J. Numer. Anal. 40~(6) (2002) 2171--2194 (2003).
\newline\urlprefix\url{https://doi.org/10.1137/S0036142901399124}

\bibitem{cockburn2002local}
B.~Cockburn, G.~Kanschat, D.~Sch\"otzau, C.~Schwab,
  \href{https://doi.org/10.1137/S0036142900380121}{Local discontinuous
  {G}alerkin methods for the {S}tokes system}, SIAM J. Numer. Anal. 40~(1)
  (2002) 319--343.
\newline\urlprefix\url{https://doi.org/10.1137/S0036142900380121}

\bibitem{carrero2006hybridized}
J.~Carrero, B.~Cockburn, D.~Sch\"otzau,
  \href{https://doi.org/10.1090/S0025-5718-05-01804-1}{Hybridized globally
  divergence-free {LDG} methods. {I}. {T}he {S}tokes problem}, Math. Comp.
  75~(254) (2006) 533--563.
\newline\urlprefix\url{https://doi.org/10.1090/S0025-5718-05-01804-1}

\bibitem{Lederer2018Hybrid}
P.~Lederer, C.~Lehrenfeld, J.~Sch\"oberl,
  \href{https://doi.org/10.1137/17M1138078}{Hybrid discontinuous {G}alerkin
  methods with relaxed h(div)-conformity for incompressible flows. part {I}},
  SIAM Journal on Numerical Analysis 56~(4) (2018) 2070--2094.
\newline\urlprefix\url{https://doi.org/10.1137/17M1138078}

\bibitem{Nguyen2010Hybridizable}
N.~Nguyen, J.~Peraire, B.~Cockburn,
  \href{https://doi.org/10.1016/j.cma.2009.10.007}{A hybridizable discontinuous
  {G}alerkin method for stokes flow}, Computer Methods in Applied Mechanics and
  Engineering 199~(9) (2010) 582 -- 597.
\newline\urlprefix\url{https://doi.org/10.1016/j.cma.2009.10.007}

\bibitem{Cockburn2009Hybridization}
B.~Cockburn, J.~Gopalakrishnan, R.~Lazarov,
  \href{https://doi.org/10.1137/070706616}{Unified hybridization of
  discontinuous galerkin, mixed, and continuous galerkin methods for second
  order elliptic problems}, SIAM Journal on Numerical Analysis 47~(2) (2009)
  1319--1365.
\newline\urlprefix\url{https://doi.org/10.1137/070706616}

\bibitem{baker1990piecewise}
G.~A. Baker, W.~N. Jureidini, O.~A. Karakashian,
  \href{https://doi.org/10.1137/0727085}{Piecewise solenoidal vector fields and
  the {S}tokes problem}, SIAM J. Numer. Anal. 27~(6) (1990) 1466--1485.
\newline\urlprefix\url{https://doi.org/10.1137/0727085}

\bibitem{karakashian1998nonconforming}
O.~A. Karakashian, W.~N. Jureidini,
  \href{https://doi.org/10.1137/S0036142996297199}{A nonconforming finite
  element method for the stationary {N}avier-{S}tokes equations}, SIAM J.
  Numer. Anal. 35~(1) (1998) 93--120.
\newline\urlprefix\url{https://doi.org/10.1137/S0036142996297199}

\bibitem{cockburn2005incompressiblei}
B.~Cockburn, J.~Gopalakrishnan,
  \href{https://doi.org/10.1137/04061060X}{Incompressible finite elements via
  hybridization. {I}. {T}he {S}tokes system in two space dimensions}, SIAM J.
  Numer. Anal. 43~(4) (2005) 1627--1650.
\newline\urlprefix\url{https://doi.org/10.1137/04061060X}

\bibitem{cockburn2005incompressibleii}
B.~Cockburn, J.~Gopalakrishnan,
  \href{https://doi.org/10.1137/040610659}{Incompressible finite elements via
  hybridization. {II}. {T}he {S}tokes system in three space dimensions}, SIAM
  J. Numer. Anal. 43~(4) (2005) 1651--1672.
\newline\urlprefix\url{https://doi.org/10.1137/040610659}

\bibitem{cockburn2007note}
B.~Cockburn, G.~Kanschat, D.~Sch\"otzau,
  \href{https://doi.org/10.1007/s10915-006-9107-7}{A note on discontinuous
  {G}alerkin divergence-free solutions of the {N}avier-{S}tokes equations}, J.
  Sci. Comput. 31~(1-2) (2007) 61--73.
\newline\urlprefix\url{https://doi.org/10.1007/s10915-006-9107-7}

\bibitem{montlaur2008discontinuous}
A.~Montlaur, S.~Fernandez-Mendez, A.~Huerta,
  \href{https://doi.org/10.1002/fld.1716}{Discontinuous {G}alerkin methods for
  the {S}tokes equations using divergence-free approximations}, Internat. J.
  Numer. Methods Fluids 57~(9) (2008) 1071--1092.
\newline\urlprefix\url{https://doi.org/10.1002/fld.1716}

\bibitem{liu2011penalty}
J.~Liu, \href{https://doi.org/10.1137/10079094X}{Penalty-factor-free
  discontinuous {G}alerkin methods for 2-dim {S}tokes problems}, SIAM J. Numer.
  Anal. 49~(5) (2011) 2165--2181.
\newline\urlprefix\url{https://doi.org/10.1137/10079094X}

\bibitem{zienkiewicz2003discontinuous}
O.~C. Zienkiewicz, R.~L. Taylor, S.~J. Sherwin, J.~Peir\'o,
  \href{https://doi.org/10.1002/nme.884}{On discontinuous {G}alerkin methods},
  Internat. J. Numer. Methods Engrg. 58~(8) (2003) 1119--1148.
\newline\urlprefix\url{https://doi.org/10.1002/nme.884}

\bibitem{montlaur2009high}
A.~D. V.~D. Montlaur, High-order discontinuous galerkin methods for
  incompressible flows, Ph.D. thesis, Universitat Politècnica de Catalunya
  (2009).

\bibitem{li2012efficient}
R.~Li, P.~B. Ming, F.~Tang, \href{https://doi.org/10.1137/110836626}{An
  efficient high order heterogeneous multiscale method for elliptic problems},
  Multiscale Model. Simul. 10~(1) (2012) 259--283.
\newline\urlprefix\url{https://doi.org/10.1137/110836626}

\bibitem{2018arXiv180300378L}
R.~{Li}, P.~{Ming}, Z.~{Sun}, Z.~{Yang}, {An Arbitrary-Order Discontinuous
  Galerkin Method with One Unknown Per Element}, ArXiv e-prints\href
  {http://arxiv.org/abs/1803.00378} {\path{arXiv:1803.00378}}.

\bibitem{bathe2000inf}
K.-J. Bathe, A.~Iosilevich, D.~Chapelle,
  \href{http://dx.doi.org/10.1016/S0045-7949(99)00213-8}{An inf-sup test for
  shell finite elements}, Computers \& Structures 75~(5) (2000) 439--456.
\newline\urlprefix\url{http://dx.doi.org/10.1016/S0045-7949(99)00213-8}

\bibitem{chapelle1993inf}
D.~Chapelle, K.-J. Bathe,
  \href{https://doi.org/10.1016/0045-7949(93)90340-J}{The inf-sup test},
  Comput. \& Structures 47~(4-5) (1993) 537--545.
\newline\urlprefix\url{https://doi.org/10.1016/0045-7949(93)90340-J}

\bibitem{Brezzi:2009}
F.~Brezzi, A.~Buffa, K.~Lipnikov,
  \href{http://dx.doi.org/10.1051/m2an:2008046}{Mimetic finite differences for
  elliptic problems}, ESAIM Numer. Anal. 43 (2009) 277--295.
\newline\urlprefix\url{http://dx.doi.org/10.1051/m2an:2008046}

\bibitem{DaVeiga2014}
L.~Beir\~ao~da Veiga, K.~Lipnikov, G.~Manzini,
  \href{https://doi.org/10.1007/978-3-319-02663-3}{The {M}imetic {F}inite
  {D}ifference {M}ethod for {E}lliptic {P}roblems}, Vol.~11 of MS\&A. Modeling,
  Simulation and Applications, Springer, Cham, 2014.
\newline\urlprefix\url{https://doi.org/10.1007/978-3-319-02663-3}

\bibitem{ciarlet:1978}
P.~G. Ciarlet, \href{http://dx.doi.org/10.1115/1.3424474}{The {F}inite
  {E}lement {M}ethod for {E}lliptic {P}roblems}, North-Holland, Amsterdam,
  1978.
\newline\urlprefix\url{http://dx.doi.org/10.1115/1.3424474}

\bibitem{antonietti2013hp}
P.~F. Antonietti, S.~Giani, P.~Houston,
  \href{https://doi.org/10.1137/120877246}{{$hp$}-version composite
  discontinuous {G}alerkin methods for elliptic problems on complicated
  domains}, SIAM J. Sci. Comput. 35~(3) (2013) A1417--A1439.
\newline\urlprefix\url{https://doi.org/10.1137/120877246}

\bibitem{li2016discontinuous}
R.~Li, P.~B. Ming, Z.~Y. Sun, Z.~J. Yang,
  \href{http://adsabs.harvard.edu/abs/2018arXiv180300378L}{An arbitrary-order
  discontinuous {G}alerkin method with one unknown per element},
  arXiv:1803.00378\href {http://arxiv.org/abs/1803.00378}
  {\path{arXiv:1803.00378}}.
\newline\urlprefix\url{http://adsabs.harvard.edu/abs/2018arXiv180300378L}

\bibitem{arnold1982interior}
D.~N. Arnold, \href{https://doi.org/10.1137/0719052}{An interior penalty finite
  element method with discontinuous elements}, SIAM J. Numer. Anal. 19~(4)
  (1982) 742--760.
\newline\urlprefix\url{https://doi.org/10.1137/0719052}

\bibitem{arnold2002unified}
D.~N. Arnold, F.~Brezzi, B.~Cockburn, L.~D. Marini,
  \href{https://doi.org/10.1137/S0036142901384162}{Unified analysis of
  discontinuous {G}alerkin methods for elliptic problems}, SIAM J. Numer. Anal.
  39~(5) (2001/02) 1749--1779.
\newline\urlprefix\url{https://doi.org/10.1137/S0036142901384162}

\bibitem{malkus1981eigenproblems}
D.~S. Malkus, \href{https://doi.org/10.1016/0020-7225(81)90013-6}{Eigenproblems
  associated with the discrete {LBB} condition for incompressible finite
  elements}, Internat. J. Engrg. Sci. 19~(10) (1981) 1299--1310.
\newline\urlprefix\url{https://doi.org/10.1016/0020-7225(81)90013-6}

\bibitem{geuzaine2009gmsh}
C.~Geuzaine, J.~F. Remacle, \href{https://doi.org/10.1002/nme.2579}{Gmsh: {A}
  3-{D} finite element mesh generator with built-in pre- and post-processing
  facilities}, Internat. J. Numer. Methods Engrg. 79~(11) (2009) 1309--1331.
\newline\urlprefix\url{https://doi.org/10.1002/nme.2579}

\bibitem{verfurth1996review}
R.~Verf{\"u}rth, {A Review of A Posteriori Error Estimation and Adaptive
  Mesh-Refinement Techniques}, Wiley-Teubner, 1996.

\bibitem{hansbo2008piecewise}
P.~Hansbo, M.~G. Larson, \href{https://doi.org/10.1002/cnm.975}{Piecewise
  divergence-free discontinuous {G}alerkin methods for {S}tokes flow}, Comm.
  Numer. Methods Engrg. 24~(5) (2008) 355--366.
\newline\urlprefix\url{https://doi.org/10.1002/cnm.975}

\end{thebibliography}

\end{document}